\documentclass[11pt]{amsart}

\usepackage{amsfonts}
\usepackage{amssymb}
\usepackage{amsmath}
\usepackage{amsthm}
\usepackage{pinlabel}
\usepackage{hyperref}
\usepackage{tikz}
\usepackage{subcaption}

\newcommand{\nc}{\newcommand}
\nc{\dmo}{\DeclareMathOperator}
\nc{\nt}{\newtheorem}

\title{Stable Subgroups of the Genus Two Handlebody Group}
\author{Marissa Chesser}
\nt{theorem}{Theorem}[section]
\nt{lemma}[theorem]{Lemma}
\nt{prop}[theorem]{Proposition}
\nt{question}[theorem]{Question}
\nt{cor}[theorem]{Corollary}
\nt{fact}[theorem]{Fact}
\nt{problem}[theorem]{Problem}
\nt{conjecture}[theorem]{Conjecture}

\theoremstyle{remark}
\nt{remark}[theorem]{Remark}

\theoremstyle{definition}
\nt{definition}[theorem]{Definition}

\nc{\Z}{\mathbb{Z}} %integers
\nc{\R}{\mathbb{R}} %reals
\nc{\Q}{\mathbb{Q}} %rationals
\nc{\N}{\mathbb{N}} %naturals

\nc{\M}{\mathcal{M}} %model
\nc{\Pnm}{\mathcal{P}} %tree underlying model
\nc{\hg}{\mathcal{H}_g} %handlebody group of genus g
\nc{\hh}{\mathcal{H}_2} %handlebody group of genus 2
\nc{\fs}{\mathfrak{F}} %factor system
\nc{\is}{\mathfrak{S}} %index set
\nc{\cat}{\operatorname{CAT}(0)} %CAT(0)
\nc{\bd}{\partial} %boundary operator
\nc{\contact}[1]{\mathcal{C}#1} %contact graph
\nc{\fcontact}[1]{\hat{\mathcal{C}}#1} %factored contact graph
\nc{\gate}[2]{\mathfrak{g}_{#1}(#2)} %gate map
\nc{\bdM}{\bd_M} %Morse boundary
\nc{\DVg}{\mathcal{D}(V_g)} % disk graph
\nc{\margin}[1]{\marginpar{\tiny #1}}

\nc{\p}[1]{\smallskip\noindent{{\bf #1}}}

\include{diagram}

\begin{document}

%%%
%%%
%%%

\begin{abstract}
We show that a finitely generated subgroup of the genus two handlebody group is stable if and only if the orbit map to the disk graph is a quasi-isometric embedding. To this end, we prove that the genus two handlebody group is a hierarchically hyperbolic group, and that the maximal hyperbolic space in the hierarchy is quasi-isometric to the disk graph of a genus two handlebody by appealing to a construction of Hamenst\"adt-Hensel. We then utilize the characterization of stable subgroups of hierarchically hyperbolic groups provided by Abbott-Behrstock-Berlyne-Durham-Russell. We also present several applications of the main theorems, and show that the higher genus analogues of the genus two results do not hold.
\end{abstract}

\maketitle

%%%%%%%%%%%%%%%%%%%%%%%%%%%%%%%%%%%%%%%%%%%%%%%%%%%%%%%%%%%%%%%%%%%%%%%%
\section{Introduction}\label{sec:intro}

In the setting of hyperbolic groups, quasiconvex subgroups are particularly well-behaved subgroups. Specifically, quasiconvex subgroups of hyperbolic groups are precisely the subgroups that are finitely generated and quasi-isometrically embedded, (see for instance Bridson and Haefliger \cite[Corollary III.$\Gamma$.3.6]{BH}). However, unlike the situation for hyperbolic groups, in the setting of arbitrary finitely generated groups, quasiconvexity is not a quasi-isometric invariant.

One generalization of a quasiconvex subgroup to arbitrary finitely generated groups that is a quasi-isometric invariant is a stable subgroup. Stable subgroups were introduced by Durham and Taylor in \cite{durham2015} as a way of characterizing convex cocompact subgroups of mapping class groups, in the sense of Farb-Mosher \cite{FarbMosher}. Another useful characterization of convex cocompact subgroups of the mapping class group is that the orbit map to the curve graph is a quasi-isometric embedding, which was proven independently by Hamenst\"adt \cite{hamenstaedt2005word} and Kent-Leninger \cite{shadows}. In this paper, we prove an analogous result for the handlebody group of genus two, ie the group of isotopy classes of orientation preserving homeomorphisms of a genus two handlebody.

\begin{theorem}\label{thm:stable-orbit}
Let $V_2$ be a genus two handlebody and suppose $G$ is a finitely generated subgroup of the handlebody group of genus two, $\hh$. Then the following are equivalent.
\begin{enumerate}
    \item $G$ is a stable subgroup of $\hh$.
    \item Any orbit map of $G$ into the disk graph $\mathcal{D}(V_2)$ is a quasi-isometric embedding.
\end{enumerate}
\end{theorem}
The disk graph is a $\delta$-hyperbolic graph akin to the curve graph whose vertices correspond to disk-bounding curves on the boundary of the handlebody, (called meridians), and whose edges correspond to disjointness.

The equivalence of stability and quasi-isometrically embedding in the curve graph is particularly notable because analogous characterizations have been proven in a number of other settings closely related to mapping class groups.
\begin{enumerate}
    \item In the setting of right angled Artin groups, Koberda, Mangahas, and Taylor \cite{RAAGs} show that stability is equivalent to quasi-isometrically embedding in the extension graph, and is equivalent to being purely loxodromic.
    \item In the setting of relatively hyperbolic groups, Aougab, Durham, and Taylor \cite{Aougab_2017} show that stability is equivalent to quasi-isometrically embedding in the cusped space or the coned off Cayley graph, under mild assumptions on the peripheral subgroups.
    \item In the setting of $\operatorname{Out}(F_n)$, Aougab, Durham, and Taylor \cite{Aougab_2017} show that quasi-isometrically embedding in the free factor graph implies stability.
    \item In the setting of hierarchically hyperbolic groups (HHGs), Abbott, Behrstock, Berlyne, Durham, and Russell \cite{abbott2017largest} show that stability is equivalent to quasi-isometrically embedding in the maximal $\delta$-hyperbolic space, and is equivalent to having uniformly bounded projections. We include in Section \ref{subsec:coarsestable} a version of this theorem that will be used in this paper.
\end{enumerate}

Our main theorem provides yet another instance of this type of characterization of stable subgroups, at least for genus two. Indeed, in the final section of this paper, we provide a counterexample showing that the higher dimensional analogue of Theorem \ref{thm:stable-orbit} does not hold.

%%%%%%%%%%%%%%%%%%%%%%%%%%%%%%%%%%%%%%%%%
\subsection{Methodology}\label{subsec:method}

Let $V_g$ be a genus $g$ handlebody, and $\hg$ the genus $g$ handlebody group. As the boundary of a handlebody is homeomorphic to a surface of genus $g$, one can view $\hg$ as a subgroup of the surface mapping class group $MCG(\bd V_g)$. Similarly, one can view the disk graph $\mathcal{D}(V_g)$ as a subgraph of the curve graph $\mathcal{C}(\bd V_g)$. Given this natural relationship between surface mapping class groups and handlebody groups, one might expect the characterization of stable subgroups of $\hg$ to be straightforward from the characterization of stable subgroups of surface mapping class groups. However, for genus $g\geq 2$, Hamenst\"adt and Hensel \cite{hamenstdt2011geometry} show that $\hg$ is exponentially distorted in $MCG(\bd V_g)$. Furthermore, even though Masur and Minsky \cite{diskQuasiconvex} showed that $\mathcal{D}(V_g)$ is quasiconvex in $\mathcal{C}(\bd V_g)$, Masur and Schleimer \cite{MasSchleim} proved that the inclusion $\mathcal{D}(V_g) \hookrightarrow \mathcal{C}(\bd V_g)$ is in general not a quasi-isometric embedding. This means that much of the toolkit used in the surface mapping class group setting cannot be easily utilized in the handlebody group setting. Indeed, even though subgroups $G\leq \hh \leq MCG(\bd V_2)$ that are stable in $MCG(\bd V_2)$ must be stable in $\hh$, via Aougab, Durham, and Taylor \cite[Theorem 1.6]{Aougab_2017}, even for cyclic subgroups of $\hh$, being stable in $\hh$ does not necessarily imply stability in $MCG(\bd V_2)$, (see for instance Hensel's survey paper \cite[Example 10.2]{primerHg}).

The main method employed to prove Theorem \ref{thm:stable-orbit} is to use the machinery of hierarchically hyperbolic spaces. In particular, we show that the $\cat$ cube complex $\M$ constructed by Hamenst\"adt and Hensel \cite{hamHenDehn}, on which $\hh$ acts properly, cocompactly, and by isometries, is a hierarchically hyperbolic space (HHS). It has been conjectured that any group acting properly, comcompactly, and by isometries on a $\cat$ cube complex is in fact an HHG, but as of yet, this has not been proven, (see for instance the discussion in the introduction of Hagen and Susse \cite{Hagen_2020}). Thus, in order to prove that $\M$ is an HHS, (and hence that $\hh$ is an HHG), we construct a factor system for $\M$ following a framework developed by Behrstock, Hagen, and Sisto \cite{behrstock2017} and using techniques developed by Hagen and Susse \cite{Hagen_2020}.

We furthermore show that the maximal $\delta$-hyperbolic space in the HHS structure of $\M$ is quasi-isometric to the disk graph. The maximal $\delta$-hyperbolic space in the setting of $\cat$ cube complexes is the factored contact graph of the entire cube complex. The factored contact graph is an augmentation of the contact graph, which is the incidence graph of hyperplane carriers in the cube complex. To show that the factored contact graph is quasi-isometric to the disk graph, we characterize the hyperplanes of $\M$ and demonstrate how these hyperplanes correspond to specific meridians. The above leads us to the second theorem.

\begin{theorem}\label{thm:hhg-diskgraph}
The handlebody group of genus two, $\hh$, is an HHG with top level hyperbolic space coarsely $\hh$-equivariantly quasi-isometric to the disk graph, $\mathcal{D}(V_2)$.
\end{theorem}
Here, coarsely equivariantly means that the quasi-isometry fails to be equivariant by some uniformly bounded distance.

Theorem \ref{thm:hhg-diskgraph} allows us to use the characterization of stable subgroups in the context of HHGs afforded by Abbott, Behrstock, Berlyne, Durham, and Russell \cite{abbott2017largest}, and to replace ``quasi-isometric embedding into the maximal $\delta$-hyperbolic space" with ``quasi-isometric embedding into the disk graph". The characterization of stable subgroups of HHGs also gives us the following additional characterization of stable subgroups of $\hh$.

\begin{cor}\label{cor:unifboundedproj}
Suppose $G$ is a finitely generated subgroup of the handlebody group of genus two, $\hh$. Then the following are equivalent.
\begin{enumerate}
    \item $G$ is a stable subgroup of $\hh$.
    \item $G$ is undistorted in $\hh$ and has uniformly bounded projections.
\end{enumerate}
\end{cor}

Several applications of Theorems \ref{thm:stable-orbit} and \ref{thm:hhg-diskgraph} are discussed in Section \ref{sec:applications}. First, we show that $\mathcal{D}(V_2)$ is a quasi-tree; this follows from the fact that $\mathcal{D}(V_2)$ is quasi-isometric to the factored contact graph, which is a quasi-tree. Theorem \ref{thm:stable-orbit} then implies that the stable subgroups of $\hh$ are all virtually free. Additionally, we show that the Morse boundary of $\hh$ is an $\omega$-Cantor space.

Lastly, we note that we can now fully answer Question C posed Behrstock, Hagen, and Sisto \cite{HHSII} which asks whether handlebody groups are HHGs. For genus $0$ and $1$, the answer is yes because $\mathcal{H}_0$ is trivial and $\mathcal{H}_1 \cong \Z$, (generated by the Dehn twist about the only merdian). Theorem \ref{thm:hhg-diskgraph} tells us that $\hh$ is also an HHG. For $g\geq 3$, we know via Hamenst\"adt and Hensel \cite[Theorem 1.1]{hamHenDehn} that $\hg$ has exponential Dehn function. Since HHGs have quadratic Dehn functions via Behrstock, Hagen, and Sisto \cite[Corollary 7.5]{HHSII}, $\hg$ cannot be an HHG for $g\geq 3$.

%%%%%%%%%%%%%%%%%%%%%%%%%%%%%%%%%%%%%%%%%
\subsection{Outline of the paper}\label{subsec:outline}

In Section \ref{sec:back}, we provide the necessary background on the handlebody group, the disk graph, meridian surgeries, coarse geometry, stability of subgroups, the geometry of $\cat$ cube complexes, and the characterization of $\cat$ cube complexes as HHSs.

In Section \ref{sec:model}, we describe the $\cat$ cube complex $\M$ constructed by Hamenst\"adt and Hensel that will serve as the model for the handlebody group of genus two. This includes a description of the overall structure of $\M$, as well as an in depth account of the two types of (combinatorial) hyperplanes found in $\M$, a classification of the parallelism classes of these hyperplanes, and a discussion of the ways these hyperplanes can contact one another.

Section \ref{sec:fsandup} starts with an explicit characterization of the convex subcomplexes that are included in our factor system. Following this, we use this characterization of the factor system to prove that $\hh$ is an HHG with unbounded products.

In section \ref{sec:contactdisk}, we prove that the factored contact graph of $\M$ is quasi-isometric to the disk graph $\mathcal{D}(V_2)$, and prove Theorems \ref{thm:hhg-diskgraph} and \ref{thm:stable-orbit}.

In Section \ref{sec:applications}, we provide several applications of the main theorems regarding additional properties of the stable subgroups of $\hh$, and a topological characterization of the Morse boundary, as described above.

Finally, in Section \ref{sec:highergenus}, we provide a counterexample to the higher dimensional analogue of Theorem \ref{thm:stable-orbit}.

%%%%%%%%%%%%%%%%%%%%%%%%%%%%%%%%%%%%%%%%%
\subsection{Acknowledgements}\label{subsec:thanks}

The author would like to thank Christopher Leininger for suggesting the problem, for many helpful conversations, for thoroughly reading drafts of the paper, and for general support throughout the writing process. The author would additionally like to thank Sebastian Hensel for answering initial questions regarding what is known about handlebody groups. The author would also like to thank Mark Hagen for clarifying some of the author's questions regarding contact graphs, and for pointing to \cite{Hagen_2020}, which gave a framework to simplify the proof of Proposition \ref{prop:contentsoffs}. The author would like to thank Alessandro Sisto for discussing acylindrical hyperbolicity, and Saul Schleimer for discussing witnesses for the disk graph, which led to the discovery of the counterexample presented in Section \ref{sec:highergenus}. Further, the author would like to thank Jacob Russell for pointing out the applications discussed in Section \ref{sec:applications}. Lastly, the author would like to thank the anonymous referee for their careful reading of the paper and helpful comments.

%%%%%%%%%%%%%%%%%%%%%%%%%%%%%%%%%%%%%%%%%%%%%%%%%%%%%%%%%%%%%%%%%%%%%%%%
\section{Background}\label{sec:back}

%%%%%%%%%%%%%%%%%%%%%%%%%%%%%%%%%%%%%%%%%
\subsection{The handlebody group and the disk graph}\label{subsec:handledisk}

A \emph{handlebody} $V_g$ \emph{of genus} $g$ is a three-manifold constructed by attaching $g$ one-handles to the boundary of a three-ball. The boundary $\bd V_g$ is homeomorphic to a surface of genus $g$. We will occasionally refer to a \emph{handlebody with spots}, which is a handlebody $V_g$ along with a collection of disjoint, embedded disks $D_i\subset \bd V_g$, referred to as \emph{spots}. Note that we define the boundary surface of a spotted handlebody to be the complement of the interior of the disks $D_i$, so the boundary of a spotted handlebody is a surface with boundary components.

The \emph{handlebody group}, $\hg$, is the mapping class group of a handlebody; that is, the group of isotopy classes of orientation preserving self-homeomorphisms of $V_g$. We can view the handlebody group as a subgroup of a surface mapping class group via the injective restriction homomorphism \[\iota: \hg \to MCG(\bd V_g).\]
Similarly to surface mapping class groups, handlebody groups are finitely generated, (see \cite{wajnryb}, \cite{suzuki}). We will view the handlebody group $\hg$ as a metric space by fixing some finite generating set and equipping $\hg$ with the word metric. Early investigation of handlebody groups was conducted by Birman \cite{birman1975} and Masur \cite{masur1986}. A survey of properties of handlebody groups can be found in \cite{primerHg}.

A essential curve $\alpha$ on $\bd V_g$ is called a \emph{meridian} if it bounds an embedded disk in $V_g$. A \emph{multimeridian} is a finite collection of pairwise disjoint, pairwise nonhomotopic meridians. Note that whenever we discuss multiple curves in relation to one another, we assume they are in pairwise minimal position. In the setting of handlebody groups, the \emph{disk graph}, denoted $\mathcal{D}(V_g)$, is a graph whose vertices correspond to isotopy classes of meridians, and for which there is an edge between two vertices when the corresponding isotopy classes of meridians have disjoint representatives. The disk graph can be viewed as a subgraph of the curve graph of $\bd V_g$, denoted $\mathcal{C}(\bd V_g)$, which is a graph whose vertices correspond to isotopy classes of essential simple closed curves on $\bd V_g$, and for which there is an edge between two vertices when the corresponding isotopy classes of curves can be made disjoint. Like the curve graph, the disk graph is $\delta$-hyperbolic \cite{MasSchleim}.

Just as there is a natural action of a surface mapping class group on the corresponding curve graph, there is a natural action of the handlebody group $\hg$ on the disk graph $\mathcal{D}(V_g)$. In particular, for an element $h\in \hg$ and a meridian $\alpha$, the image $h(\alpha)$ is also a meridian, (see for instance \cite[Corollary 5.11]{primerHg}). Additionally, because homeomorphisms of a surface preserve disjointness, if $\alpha$ and $\beta$ are two disjoint meridians, $h(\alpha)$ and $h(\beta)$ will also be disjoint. Thus the action of $\hg$ on the vertices $\mathcal{D}(V_g)^{(0)}$ will preserve edges.

One particular class of elements in the handlebody group that will be relevant in this paper are Dehn twists along meridians. Intuitively, a Dehn twist along a meridian $\alpha$, denoted in this paper by $T_{\alpha}$, corresponds to cutting $V_g$ along a disk bounded by $\alpha$, twisting the handle one full twist, and then re-gluing. Clearly this restricts to a Dehn twist in the typical sense on $\bd V_g$.

%%%%%%%%%%%%%%%%%%%%%%%%%%%%%%%%%%%%%%%%%
\subsection{Meridian surgeries}\label{subsec:propertieshandle}

Given a handlebody $V_g$, a \emph{cut system} on $V_g$ is a collection $\alpha_1,\dots,\alpha_g$ of disjoint, non-isotopic meridians such that $\bd V_g -(\alpha_1\cup\cdots \cup \alpha_g)$ is connected. Equivalently, a cut system $\alpha_1,\dots,\alpha_g$ are the boundary curves of a collections of disks $D_1,\dots,D_g\subset V_g$ such that $V_g-(D_1\cup\cdots \cup D_g)$ is a single $3$-ball.

The following lemma demonstrates a way to construct a sequence of cut systems $\{Z_i\}$ such that each $Z_i$ has two fewer intersections with some (multi)meridian $\beta$ than $Z_{i-1}$. The version of this lemma listed below comes from \cite[Proposition 4.1]{hamHenDehn}, though this lemma is well-known and is true in higher genus. Versions of this lemma in higher genus cases can be found, for instance, in \cite[Lemma 1.1]{masur1986} and \cite[Lemma 5.2]{hamenstdt2011geometry}.

\begin{lemma}[{{\cite[Proposition 4.1]{hamHenDehn}}}]\label{lem:surgery}
    Let $Z=\{\alpha_1,\alpha_2\}$ be a cut system on $V_2$, and suppose $\beta$ is some (multi)meridian. Then either $\alpha_1\cup\alpha_2$ is disjoint from $\beta$, or there exists a subarc $b\subset \beta$ with the following properties.
    \begin{enumerate}
        \item The arc $b$ intersects $\alpha_1\cup\alpha_2$ only in its endpoints, and both endpoints lie on the same curve, say $\alpha_1$.
        
        \item The endpoints of $b$ approach $\alpha_1$ from the same side.
        
        \item Let $a,a'$ be the two components of $\alpha_1-b$. Then one of $a\cup b$ or $a'\cup b$ is a meridian, say $a\cup b$. Furthermore, $(a\cup b,\alpha_2)$ is a cut system that we call the \emph{surgery of $Z$ defined by $b$ in the direction of $\beta$}.
        
        \item The surgery defined by $b$ has two fewer intersections with $\beta$ than $Z$.
    \end{enumerate}
\end{lemma}

Given an initial cut system $Z$ and (multi)meridian $\beta$, Lemma \ref{lem:surgery} allows us to construct a sequence $\{Z_i\}_{i=1}^n$ of cut systems such that $Z_1=Z$, $Z_i$ is a surgery of $Z_{i-1}$ in the direction of $\beta$ for $i\in(1,n]$, and $Z_n$ is disjoint from $\beta$. Furthermore, consecutive cut systems in $\{Z_i\}_{i=1}^n$ have no transverse intersections. We call a sequence of cut systems constructed in this way a \emph{surgery sequence starting at $Z$ in the direction of $\beta$.}

%%%%%%%%%%%%%%%%%%%%%%%%%%%%%%%%%%%%%%%%%
\subsection{Geometry of \texorpdfstring{$\cat$}{CAT(0)} cube complexes}\label{subsec:catHHS}

An $n$-\emph{cube} for $0\leq n < \infty$ is a copy of the Euclidean cube $[-\frac{1}{2},\frac{1}{2}]^n$. A \emph{cube complex} is a cell complex in which the $n$-cells are $n$-cubes, and in which the attaching maps are isometries. A cube complex $\mathcal{X}$ is $\cat$ if every triangle in $\mathcal{X}$ is at least as thin as a comparison triangle in Euclidean space. For a more detailed definition and other properties of $\cat$ cube complexes, see for example \cite[Sections I.7 and II.1]{BH}. For the remainder of this section, let $\mathcal{X}$ be a $\cat$ cube complex.

A \emph{midcube} of a cube $c$ is a subspace obtained by restricting exactly one coordinate of $c$ to $0$. A \emph{hyperplane} $H$ is a connected union of midcubes of $\mathcal{X}$ such that for any finite dimensional cube $c$ of $\mathcal{X}$, either $H\cap c = \emptyset$ or $H \cap c$ is a midcube. Any hyperplane $H$ in $\mathcal{X}$ is \emph{separating}, meaning $\mathcal{X} - H$ has exactly two components called \emph{half-spaces} \cite{sageev}. The \emph{carrier} $N(H)$ of a hyperplane $H$ is the union of cubes in $\mathcal{X}$ which have non-empty intersection with $H$. There is a cubical isometric embedding $H\times [-\frac{1}{2},\frac{1}{2}] \simeq N(H) \hookrightarrow \mathcal{X}$, and we denote by $H^{\pm}$ the images of $H\times \{\pm \frac{1}{2}\}$. We call each of these $H^{\pm}$ \emph{combinatorial hyperplanes.}

A subcomplex $F$ of $\mathcal{X}$ is \emph{convex} if $F^{(1)}$ is metrically convex in $\mathcal{X}^{(1)}$, using the induced path metric, and if every cube whose $0$-skeleton is contained in $F$ is also contained in $F$. This notion of convexity agrees with the $\cat$-metric convexity for subcomplexes, though not for arbitrary subspaces of $\mathcal{X}$ \cite{Hagen_2020}.

We say that two convex subcomplexes $F_1$ and $F_2$ are \emph{parallel} if for every hyperplane $H$ in $\mathcal{X}$, $F_1\cap H \neq \emptyset$ if and only if $F_2 \cap H \neq \emptyset$. Note that parallelism is an equivalence relation, and we will denote the parallelism class of a convex subcomplex $F$ via $[F]$. Also note that combinatorial hyperplanes $H^+$ and $H^-$ are convex, and are always parallel to one another. We say that a convex subcomplex $F_1$ is \emph{parallel into} a convex subcomplex $F_2$ if for every hyperplane $H$ in $\mathcal{X}$, if $F_1\cap H \neq \emptyset$, then $F_2 \cap H \neq \emptyset.$ Occasionally it will be useful to talk about hyperplanes $H_1$ and $H_2$ being parallel (into). By this we mean that the associated combinatorial hyperplanes $H_1^{\pm}$ and $H_2^{\pm}$ are parallel (into). The following lemma provides a useful characterization of parallel subcomplexes.

\begin{lemma}[{{\cite[Lemma 2.4]{behrstock2017}}}]\label{lem:HHS-2.4}
 Let $F,F'\subset \mathcal{X}$ be convex subcomplexes. The following are equivalent:
 \begin{enumerate}
     \item $F$ and $F'$ are parallel.
     \item There is a cubical isometric embedding $F\times [0,a] \to \mathcal{X}$ whose restrictions to $F\times \{0\}$ and $F\times \{a\}$ factor as $F\times \{0\} \cong F\hookrightarrow \mathcal{X}$ and $F\times \{a\} \cong F'\hookrightarrow \mathcal{X}$, respectively, and for every vertex $x\in F$, $\{x\}\times[0,a]$ is a combinatorial geodesic segment crossing exactly those hyperplanes that separate $F$ from $F'$.
 \end{enumerate}
Hence, there exists a convex subcomplex $E_F$ such that there is a cubical embedding $F\times E_F \to \mathcal{X}$ with convex image such that for each $F'$ in the parallelism class of $F$, there exists a $0$-cube $e\in E_F$ such that $F\times \{e\}\to \mathcal{X}$ factors as $F\times \{e\} \xrightarrow{id} F' \hookrightarrow F$.
\end{lemma}

Given any subset $A\subset \mathcal{X}$ and a hyperplane $H$ in $\mathcal{X}$, we will say $H$ \emph{crosses} $A$ if $A\cap H \neq \emptyset$. Additionally, two hyperplanes $H_1$ and $H_2$ are said to \emph{osculate} if $H_1 \cap H_2 = \emptyset$ but $N(H_1) \cap N(H_2) \neq \emptyset.$

Given a convex subcomplex $F\subset \mathcal{X},$ we define the \emph{gate map} $\mathfrak{g}_F:\mathcal{X}^{(0)}\to F^{(0)}$ between $0$-skeleta to be the map such that $\mathfrak{g}_F(x)$ is the unique closest $0$-cube in $F^{(0)}$ to $x$. It is proven in \cite{behrstock2017} that this map extends to a cubical map $\mathfrak{g}_F:\mathcal{X}\to F$ such that an $n$-cube $c$ is collapsed to the unique $m$-cube whose $0$-cubes are the images of the $0$-cubes of $c$ under the gate map, where $0\leq m \leq n$,. We include here a lemma from \cite{Hagen_2020} and another lemma from \cite{behrstock2017} regarding the gate map that we will make use of throughout the paper.

\begin{lemma}[{{\cite[Lemma 1.5]{Hagen_2020}}}]\label{lem:Hagen-1.5}
    For any convex subcomplexes $F,F'\subset \mathcal{X}$, the hyperplanes crossing $\gate{F}{F'}$ are precisely the hyperplanes crossing both $F$ and $F'$.
\end{lemma}

\begin{lemma}[{{\cite[Lemma 2.6]{behrstock2017}}}]\label{lem:HHS-2.6}
If $F,F'\subset \mathcal{X}$ are convex subcomplexes, then $\gate{F}{F'}$ and $\gate{F'}{F}$ are parallel subcomplexes. Moreover, if $F\cap F' \neq \emptyset$, then $\gate{F}{F'} = \gate{F'}{F} = F\cap F'$.
\end{lemma}

Let $F\subset \mathcal{X}$ be a convex subcomplex; $F$ is a $\cat$ cube complex. The \emph{contact graph} $\contact{F}$ is a $\delta$-hyperbolic graph, (actually a quasi-tree), originally defined by Hagen \cite{Hagen_2013}. Each vertex in the contact graph corresponds to a hyperplane in $F$, and there is an edge between two vertices if the carriers of the corresponding hyperplanes have non-empty intersection, ie if the hyperplanes either cross or osculate. If $K\subset F$ is a convex subcomplex, then the hyperplanes of $K$ can be described as $K\cap H$ where $H$ is a hyperplane of $F$; this is because every hyperplane is determined by a single midcube contained in it. The definition of convexity implies that the inclusion $K\hookrightarrow F$ induces an injective graph homomorphism $\contact{K}\to \contact{F}$ sending a hyperplane $H\cap K$ of $K$ to the hyperplane $H$ of $F$. Furthermore, via the definition of parallel subcomplexes, if $K_1,K_2\subset F$ are parallel, convex subcomplexes, then $\contact{K_1}$ and $\contact{K_2}$ are the same subcomplexes of $\contact{F}$.

%%%%%%%%%%%%%%%%%%%%%%%%%%%%%%%%%%%%%%%%%
\subsection{\texorpdfstring{$\cat$}{CAT(0)} cube complexes as hierarchically hyperbolic spaces}\label{subsec:HHS}

Here we describe what is needed in order to show that a $\cat$ cube complex $\mathcal{X}$ is a hierarchically hyperbolic space. Because this paper is only concerned with a specific $\cat$ cube complex, we omit the complete definition of an HHS and refer the reader to \cite{behrstock2017} for full details.

For the remainder of this subsection, let $\mathcal{X}$ denote some arbitrary $\cat$ cube complex. In order to show that $\mathcal{X}$ is an HHS, one must demonstrate the existence of a \emph{factor system}.

\begin{definition}[{{\cite[Definition 8.1]{behrstock2017}}}]\label{def:factorsys}
    A \emph{factor system} $\fs$ is a collection of non-empty, convex subcomplexes of $\mathcal{X}$ satisfying the following properties:
\begin{enumerate}
    \item $\mathcal{X}\in \fs$.
    \item There exists a number $N \geq 1$ such that every $x\in\mathcal{X}^{(0)}$ is contained in at most $N$ subcomplexes in $\fs$. We refer to this as the \emph{finite multiplicity} property.
    \item If $F$ is a non-trivial subcomplex of $\mathcal{X}$ that is parallel to a combinatorial hyperplane of $\mathcal{X}$, then $F\in\fs$.
    \item There is some number $\xi \geq 0$ such that if $F_1,F_2\in\fs$ and $\operatorname{diam}(\mathfrak{g}_{F_1}(F_2)) \geq \xi$, then $\mathfrak{g}_{F_1}(F_2)\in\fs$.
\end{enumerate}
\end{definition}

Given a factor system $\fs$ for $\mathcal{X}$ and a subcomplex $F\in\fs,$ \cite[Definition 8.14]{behrstock2017} defines the \emph{factored contact graph} $\fcontact{F}$, which is constructed from the contact graph $\contact{F}$ in the following way. Let $F'\in \fs-F$ such that $F'\subset F$ and such that $\operatorname{diam}(F')\geq \xi$ or $F'$ is parallel to a combinatorial hyperplane of $\mathcal{X}$, (or both). Given a subcomplex $F'$ with these properties, we add one vertex $v_{[F']}$ to $\contact{F}$ corresponding to parallelism class of $F'$, and we connect $v_{[F']}$ by an edge to each vertex in $\contact{F'}\subset \contact{F}$. This means that if $F''$ is parallel to $F'$, then we only add one vertex $v_{[F']} = v_{[F'']}$ to $\contact{F}$. Note that the contact graph $\contact{F}$ is an induced subgraph of $\fcontact{F}$, ie $\contact{F}^{(0)} \subset \fcontact{F}^{(0)}$ and the edges of $\contact{F}$ consist of all edges from $\fcontact{F}$ whose endpoints are in $\contact{F}^{(0)}$.

If $\mathcal{X}$ contains a factor system $\fs$, then via \cite[Remark 13.2]{behrstock2017}, $\mathcal{X}$ is an HHS whose set of \emph{domains} $\is$, (sometimes called an \emph{index set}), is a subset of $\fs$ containing one representative $F\in\fs$ of each parallelism class in $\fs$, (except single points). The set $\is$ is equipped with a partial order $\sqsubseteq$ such that $F_1 \sqsubseteq F_2$ if and only if $F_1$ is parallel into $F_2$; the maximal element is the cube complex $\mathcal{X}$ itself. The $\delta$-hyperbolic space associated to a subcomplex $F$ is the factored contact graph $\fcontact{F}$. If $G$ is a group acting properly, cocompactly, and by isometries on $\mathcal{X}$, then $(G,\is)$ is a hierarchically hyperbolic group. Note that this is not the definition of an HHG, but rather a specific example of an HHG. For the full definition, see for instance \cite[Definition 1.21]{HHSII}.

One way to construct a factor system for $\mathcal{X}$ is via the hyperclosure.
\begin{definition}[{{\cite[Definition 1.14]{Hagen_2020}}}]\label{def:hyperclosure}
    The \emph{hyperclosure of $\mathcal{X}$} is the intersection $\fs$ of all sets $\mathfrak{G}$ of convex subcomplexes of $\mathcal{X}$ that satisfy the following properties:
\begin{enumerate}
    \item $\mathcal{X}\in\mathfrak{G}$.
    \item If $C$ is a combinatorial hyperplane of $\mathcal{X}$, then $C\in \mathfrak{G}$.
    \item If $F,F'\in\mathfrak{G}$, then $\gate{F}{F'}\in\mathfrak{G}$.
    \item If $F\in\mathfrak{G}$ and $F'$ is parallel to $F$, then $F'\in\mathfrak{G}$.
\end{enumerate}
\end{definition}
If the hyperclosure has the finite multiplicity property, then $\fs$ is a factor system in the sense of Definition \ref{def:factorsys}. This is clear because properties (1), (2), and (4) of Definition \ref{def:hyperclosure} satisfy properties (1) and (3) of Definition \ref{def:factorsys}, and property (3) of Definition \ref{def:hyperclosure} satisfies property (4) of Definition \ref{def:factorsys} with $\xi = 0$.

The following lemma will be useful in our analysis of the hyperclosure.
\begin{lemma}[{{\cite[Lemma 2.2]{Hagen_2020}}}]\label{lem:Hagen-2.2}
    Let $\fs$ be the hyperclosure of $\mathcal{X}$. Let $\fs_0 = \{\mathcal{X}\}$, and let $\fs_n$, for $n\geq1$, be the subset of $\fs$ consisting of subcomplexes that can be written in the form $\gate{C}{F}$ where $C$ is a combinatorial hyperplane of $\mathcal{X}$ and $F\in\fs_{n-1}$. Then $\fs=\cup_{n\geq 0} \fs_n.$
\end{lemma}
Notice that $\fs_1$ is equal to the set of all combinatorial hyperplanes of $\mathcal{X}.$

A set of convex subcomplexes closely related to the hyperclosure is the set $\mathfrak{M}_{\xi}$ that we will refer to as the \emph{closure of $\mathcal{X}$}. Here we define $\mathfrak{M}_{\xi}$ to be the closure of the set
\[ \mathcal{Y}=\{ \text{combinatorial hyperplanes of  }\mathcal{X}\}\cup\{\mathcal{X}\}\]
under projections with diameter $\geq \xi$; that is, $\mathfrak{M}_{\xi}$ is the smallest set containing $\mathcal{Y}$ such that for all $F_1, F_2\in \mathfrak{M}_{\xi}$, if $\operatorname{diam}(\gate{F_1}{F_2}) \geq \xi$, then $\gate{F_1}{F_2}\in\mathfrak{M}_{\xi}$.

As with the hyperclosure, $\mathfrak{M}_{\xi}$ will be a factor system when it has finitely multiplicity; it is clear by the definition of $\mathfrak{M}_{\xi}$ that it satisfies properties (1) and (4) of Definition \ref{def:factorsys}. Moreover, $\mathfrak{M}_{\xi}$ satisfies property (3) of Definition \ref{def:factorsys} because if $C$ is a combinatorial hyperplane and $C'\subset \mathcal{X}$ is a convex subcomplex parallel to $C$, then by \cite[Lemma 2.5]{behrstock2017}, $C$ is contained in a combinatorial hyperplane $H$, and $C' = \gate{H}{C} \in \mathfrak{M}_{\xi}$. When $\mathfrak{M}_{\xi}$ is a factor system, it is referred to in \cite{behrstock2017} as the \emph{minimal factor system}, and the authors note that any factor system for $\mathcal{X}$ with projections closed under the chosen $\xi$ must contain $\mathfrak{M}_{\xi}$. We also note that if $\fs$ is the hyperclosure of $\mathcal{X}$, then $\mathfrak{M}_{0}\subset \fs$. This is because $\mathfrak{M}_0$ satisfies properties (1)-(3) of Definition \ref{def:hyperclosure}, implying that any $\mathfrak{G}$ as in Definition \ref{def:hyperclosure} must contain $\mathfrak{M}_0$. We have equality when $\mathfrak{M}_0$ is closed under parallelism. For the remainder of this paper, we let $\mathfrak{M} = \mathfrak{M}_0$ so that $\mathfrak{M}\subset \fs$.

Associated to every HHS are sets known as \emph{standard product regions.} For the purposes of this paper, we will only need to understand what the product regions look like when our HHS is a $\cat$ cube complex; for details regarding standard product regions for a general HHS, see for instance \cite[Section 13.1]{behrstock2017}. For a $\cat$ cube complex, \cite[Remark 13.5]{behrstock2017} describes the standard product regions as subcomplexes of $\mathcal{X}$ that are of the form $F\times E_F$, where $F$ is any subcomplex in a factor system $\fs$, and $E_F$ is an associated subcomplex as described in Lemma \ref{lem:HHS-2.4}.

Suppose $(G,\is)$ is an HHG for which the HHG structure comes from an action on the $\cat$ cube complex $\mathcal{X}$. We say that $(G,\is)$ has \emph{unbounded products} if the following holds: for every $F\in\is-\{\mathcal{X}\}$, whenever $\operatorname{diam}(F) = \infty$, then $\operatorname{diam}(E_F)= \infty$. For the more general definition of unbounded products for any HHS, see \cite[Section 3.1]{abbott2017largest}.

%%%%%%%%%%%%%%%%%%%%%%%%%%%%%%%%%%%%%%%%%
\subsection{Coarse geometry and stability}\label{subsec:coarsestable}

Suppose $(X,d_X)$ and $(Y,d_Y)$ are metric spaces and that $f:X\to Y$ is a (not necessarily continuous) map. If there exists $K\geq 1$, $C\geq 0$ such that for every $a,b\in X$
\[
    \frac{1}{K}d_X(a,b) -C \leq d_Y(f(a),f(b))
 \leq Kd_X(a,b) +C,
\]
then we say that $f$ is a $(K,C)$-\emph{quasi-isometric embedding}.
If it is also true that there exists $D\geq 0$ such that $Y$ is contained in a $D$-neighborhood of $f(X)$, then we call $f$ a $(K,C)$-\emph{quasi-isometry}, and we say that $f(X)$ is a \emph{$D$-dense subset of $Y$.} When there exists a quasi-isometry $f:X\to Y$, we say that $X$ and $Y$ are \emph{quasi-isometric}. For any quasi-isometry $f:X\to Y$, there exists a quasi-isometry $g:Y\to X$ and a constant $k\in\N$ such that for all $x\in X$ and $y\in Y$, $d_X(gf(x),x)\leq k$ and $d_Y(fg(y),y) \leq k$; we call $g$ a \emph{quasi-inverse} for $f$. If $f:I\to Y$  is a quasi-isometric embedding with $I$ an interval of the real line, we call $f$ a $(K,C)$-\emph{quasi-geodesic.}

Let $f:X\to Y$ be a function on two $G$-spaces $X$ and $Y$. We say that the function $f$ is \emph{coarsely $G$-equivariant} if there exists $N\in\N$ such that for all $x\in X$ and $\gamma\in G$,
\[
    d_Y(\gamma\cdot f(x) ,f(\gamma\cdot x)) \leq N.
\]
In other words, the function $f$ fails to be $G$-equivariant by some bounded distance.

Let $G$ be a finitely generated group, and $H$ a finitely generated subgroup of $G$. We say that $H$ is \emph{undistorted} in $G$ if the inclusion $i:H\hookrightarrow G$ is a quasi-isometric embedding for some (any) word metrics on $H$ and $G$. If $H$ is undistorted in $G$ then it is a \emph{stable subgroup} of $G$ if for any finite generating set $S$ for $G$ with associated word metric $|\cdot|_S$, and for every $K\geq 1$, $C\geq 0$, there is some $D =D(S,K,C)$ such that any two $(K,C)$-quasi-geodesics in $(G,|\cdot|_S)$ with common endpoints in $H\subset (G,|\cdot|_S)$ remain in the $D$-neighborhoods of each other. Durham and Taylor show in \cite{durham2015} that stability of subgroups is a quasi-isometric invariant.

As mentioned in the introduction, we will be using using the characterization of stable subgroups of HHGs provided in \cite{abbott2017largest}. The authors of that paper present two characterizations of stable subgroups of HHGs. In \cite[Theorem B]{abbott2017largest}, the authors provide a characterization of stable subgroups for any HHG, but this characterization requires alterations to the HHS structure. Alternatively, they produce a characterization of stable subgroups of HHGs that does not require any alteration of the HHS structure, but that adds the additional requirement that product regions are unbounded. In this paper, we will utilize the latter characterization of stable subgroups of HHGs, recorded below, with the wording changed slightly to better fit our setting.
\begin{theorem}[{{\cite[Corollary 6.2]{abbott2017largest}}}]\label{thm:abbott}
Suppose $(G,\is)$ is a hierarchically hyperbolic group with unbounded products, and that $H<G$ is a finitely generated subgroup. Then the following are equivalent.
\begin{enumerate}
    \item $H$ is a stable subgroup of $G$.
    \item $H$ is undistorted in $G$ and has uniformly bounded projections.
    \item Any orbit map $H\to \contact{S}$ is a quasi-isometric embedding, where $S$ is maximal in $\is$.
\end{enumerate}
\end{theorem}
Note that here $\contact{S}$ refers to the $\delta$-hyperbolic space associated to the maximal domain $S\in \is$. Additionally, uniformly bounded projections refers to the following notion. Suppose that $(\mathcal{X},\is)$ is any HHS with associated $\delta$-hyperbolic spaces $\{\contact{F}:F\in\is\}$. Also associated to $\mathcal{X}$ space is a collection of projection maps $\{\pi_F:\mathcal{X}\to 2^{\contact{F}}: F\in \is\}$ sending points in $\mathcal{X}$ to sets of bounded diameter in $\contact{F}$. Suppose $\mathcal{Y}\subset \mathcal{X}$ is any subset, and suppose $S$ is the maximal element of $\is$. We say that $\mathcal{Y}$ has $D$-\emph{bounded projections} if there exists some $D>0$ such that, $\operatorname{diam}(\pi_F(\mathcal{Y})) < D$ for all $F\in\is-\{S\}$. If the constant $D$ does not matter, we say $\mathcal{Y}$ has \emph{uniformly bounded projections}. In the case that $\mathcal{X}$ is a $\cat$ cube complex, the maps $\pi_F:\mathcal{X}\to \fcontact{F}$ send each point $x\in\mathcal{X}^{(0)}$ to the clique of vertices corresponding to hyperplanes whose carriers containing $\gate{F}{x}$.

%%%%%%%%%%%%%%%%%%%%%%%%%%%%%%%%%%%%%%%%%%%%%%%%%%%%%%%%%%%%%%%%%%%%%%%%
\section{A Model for \texorpdfstring{$\hh$}{H2}}\label{sec:model}

In \cite{hamHenDehn}, Hamenst\"adt and Hensel construct a $\cat$ cube complex on which the handlebody group of genus two acts properly, comcompactly, and by isometries, and which we will refer to as $\M$. In this section, we will take a detailed look at this cube complex. In particular, we summarize their construction in Section \ref{subsec:model}, classify the hyperplanes of $\M$ in Section \ref{subsec:hyperplanes}, determine the parallelism classes of the combinatorial hyperplanes of $\M$ in Section \ref{subsec:parallelhyperplanes}, and discuss some properties of the contact graph $\contact{\M}$ in Section \ref{subsec:edgescontact}. Throughtout this section, let $V_2$ be a genus two handlebody.

%%%%%%%%%%%%%%%%%%%%%%%%%%%%%%%%%%%%%%%%%
\subsection{The model}\label{subsec:model} Let $X = \{\alpha_1,\alpha_2,\alpha_3\}$ be a pants decomposition on $\bd V_2$ consisting only of non-separating meridians. Hamenst\"adt and Hensel (\cite[Lemma 6.1]{hamHenDehn}) show that for each such $X$, one can construct a \emph{dual system} $\Delta = \{\delta_1, \delta_2, \delta_3\}$ of non-separating meridians satisfying the following properties:
\begin{enumerate}
    \item[(i)] $\delta_i$ is disjoint from $\alpha_j$ for $i\neq j$.
    \item[(ii)] $\delta_i$ intersects $\alpha_i$ exactly twice.
    \item[(iii)] A dual system $\Delta$ is uniquely defined up to Dehn twists about curves in $X$. In particular, if $\delta_i$ and $\delta_i'$ are two different dual curves to $\alpha_i$, then $\delta_i = T_{\alpha_i}^{n_i}(\delta_i')$ for some integer $n_i$.
\end{enumerate}
The $0$-skeleton of $\M$ comprises all pairs $(X,\Delta)$ as above.

There are two types of edges in the $1$-skeleton of $\M$. Two vertices $(X,\Delta)$ and $(X',\Delta')$ will be connected by a \emph{twist edge} if $X = X'$ and $\Delta' = T_{\alpha_i}(\Delta)$ for some $\alpha_i \in X$, ie the vertices share a pants decomposition and the dual systems differ by a Dehn twist about one of the pants curves.

To describe the second type of edges, suppose $X = \{\alpha_1,\alpha_2,\alpha_3\}$ and $\Delta = \{\delta_1,\delta_2,\delta_3\}$ are a pants decomposition and dual system pair. Switching some $\alpha_i$, say $\alpha_1$, with the corresponding dual curve $\delta_1$ produces another pants decomposition $X' = \{\delta_1,\alpha_2,\alpha_3\}.$ The collection $\{ \alpha_1,\delta_2,\delta_3\}$ will not be a dual system to $X'$ because $\delta_2$ and $\delta_3$ will both intersect $\delta_1$ twice, but by applying a canonical cleanup process $c$ to $\{\alpha_1,\delta_2,\delta_3\}$, (as described in \cite[Section 6]{hamHenDehn} and as illustrated in Figure \ref{figure:cleanup-example}), we can obtain a dual system $\Delta' = \{\alpha_1, c(\delta_2), c(\delta_3)\}$ for $X'$. We connect any two such  vertices $(X,\Delta)$ and $(X',\Delta')$ via an edge, called a \emph{switch edge}, and we say $(X',\Delta')$ is \emph{obtained from $(X,\Delta)$ by switching $\alpha_1$.} The canonical cleanup function $c$ commutes with Dehn twists (\cite[Lemma 6.3]{hamHenDehn}).
    
\begin{figure}[htb]
\centering
\labellist \small\hair 2pt
 \pinlabel {$\alpha_2^+$} [ ] at 160 700
 \pinlabel {$\alpha_3^+$} [ ] at 210 350
 \pinlabel {$\delta_1$} [ ] at 230 505
 \pinlabel {$\delta_2$} [ ] at 420 625
 \pinlabel {$\alpha_1$} [ ] at 760 505
 \pinlabel {$\alpha_3^-$} [ ] at 610 590
 \pinlabel {$\alpha_2^-$} [ ] at 610 270
 \pinlabel {$\alpha_2^+$} [ ] at 1035 700
 \pinlabel {$c(\delta_2)$} [ ] at 1265 625
 \pinlabel {$\alpha_3^+$} [ ] at 1085 350
 \pinlabel {$\alpha_2^-$} [ ] at 1485 270
 \pinlabel {$\alpha_3^-$} [ ] at 1485 590
 \pinlabel {$\delta_1$} [ ] at 1105 505
 \pinlabel {$\alpha_1$} [ ] at 1635 505
\endlabellist
\includegraphics[scale=.20]{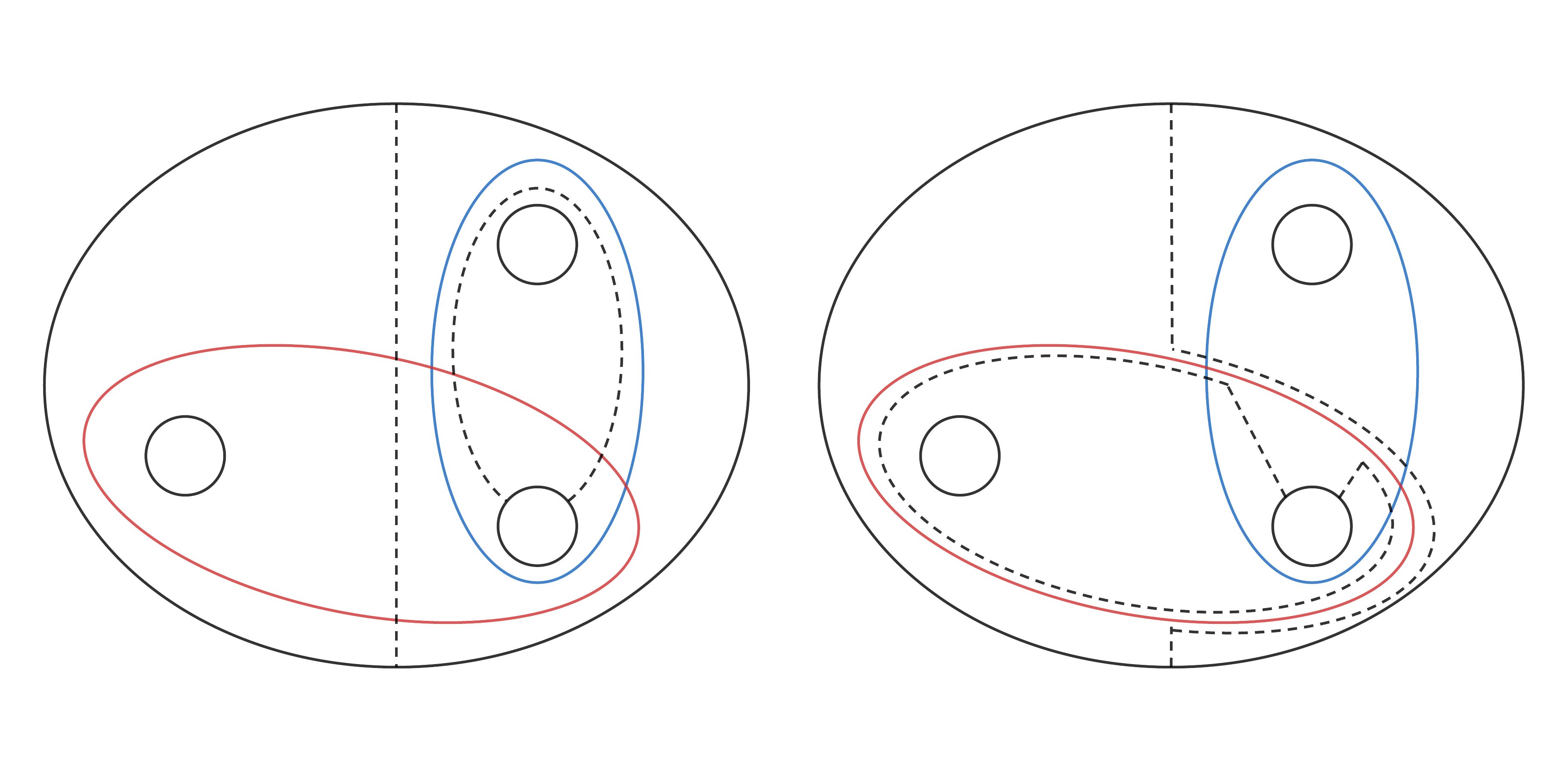}
\caption{This image shows $\bd V_2$ cut along $\alpha_2$ and $\alpha_3$, where $\alpha_i^{\pm}$ refers to the two sides of the curve after cutting. This figure illustrates the cleanup of $\delta_2$ after switching $\alpha_1$ and $\delta_1$, and is based off of \cite[Figure 8]{hamHenDehn}}.
\label{figure:cleanup-example}
\end{figure}

We then fill in higher dimensional cubes $[-\frac{1}{2},\frac{1}{2}]^n$ wherever we see their $1$-skeleta. The resulting cube complex is $3$-dimensional and contains two distinct types of $3$-cubes, which we now describe. First, fix some pants decomposition $X=\{\alpha_1,\alpha_2,\alpha_3\}$. Consider the subgraph of the $1$-skeleton of $\M$ containing only vertices whose pants decomposition is $X$. Because dual systems for a given pants decomposition are unique up to Dehn twists about the pants curves $\alpha_i$, this subgraph contains only twist edges, and is in fact isomorphic to the Cayley graph of $\Z^3$, with respect to a basis generating set. This subgraph forms the $1$-skeleton of a subcomplex isomorphic to $\R^3$, with the standard integral cube complex structure, and the $3$-cubes contained in this subcomplex are referred to as \emph{twist cubes}. We denote this subcomplex by $\M(X)$, and we call such subcomplexes \emph{twist flats.}

To describe the second type of $3$-cubes contained in $\M$, again fix a pants decomposition $X = \{\alpha_1,\alpha_2,\alpha_3\}$, as well as a dual system $\Delta = \{\delta_1,\delta_2,\delta_3\}$. Let $(X',\Delta')$ be obtained from $(X,\Delta)$ by switching $\alpha_2$. Notice that the following four vertices form the four vertices of a $2$-cube $k(X)$ whose edges are all twist edges:
    \[
        \{(X,\Delta), (X,T_{\alpha_1}(\Delta)), (X,T_{\alpha_3}(\Delta)), (X,T_{\alpha_1}T_{\alpha_3}(\Delta))\}.
    \]
Similarly, the following four vertices are contained in a $2$-cube $k(X')$ whose edges are all twist edges:
    \[
        \{(X',\Delta'), (X',T_{\alpha_1}(\Delta')), (X',T_{\alpha_3}(\Delta')), (X',T_{\alpha_1}T_{\alpha_3}(\Delta'))\}.
    \]
Furthermore, $k(X)$ and $k(X')$ are connected to one another via switch edges. In particular, there are switch edges connecting $(X,\Delta)$ to $(X',\Delta')$, $(X,T_{\alpha_1}(\Delta))$ to $(X',T_{\alpha_1}(\Delta'))$, $(X,T_{\alpha_3}(\Delta))$ to $(X',T_{\alpha_3}(\Delta'))$, and $(X,T_{\alpha_1}T_{\alpha_3}(\Delta))$ to $(X',T_{\alpha_1}T_{\alpha_3}(\Delta'))$. These switch edges, along with the twist edges in $k(X)$ and $k(X')$ form the $1$-skeleton of a $3$-cube, which we refer to as a \emph{switch cube}. The subcomplex containing all switch cubes connecting the twist flats $\M(X)$ and $\M(X')$ is isomorphic to $\R^2\times [-\frac{1}{2},\frac{1}{2}]$. We denote this subcomplex by $\M(X,X')$ and refer to such subcomplexes as \emph{switch bridges}. An illustration of how a switch bridge connects two twist flats can be seen in Figure \ref{figure:twistflats-switchbridge}.

\begin{figure}[htb]
\centering
\labellist \small\hair 2pt
\pinlabel {$T_{\alpha_2}$} [ ] at 450 300
 \pinlabel {$T_{\alpha_1}$} [ ] at 150 95
 \pinlabel {$T_{\alpha_3}$} [ ] at 343 805
 \pinlabel {$\M(X)$} [ ] at 203 560
 \pinlabel {$\M(X')$} [ ] at 1345 560
 \pinlabel {$\M(X,X')$} [ ] at 720 400
 \pinlabel {$T_{\alpha_1}$} [ ] at 935 95
 \pinlabel {$T_{\delta_2}$} [ ] at 1456 300
 \pinlabel {$T_{\alpha_3}$} [ ] at 1125 805
\endlabellist
\includegraphics[scale=.23]{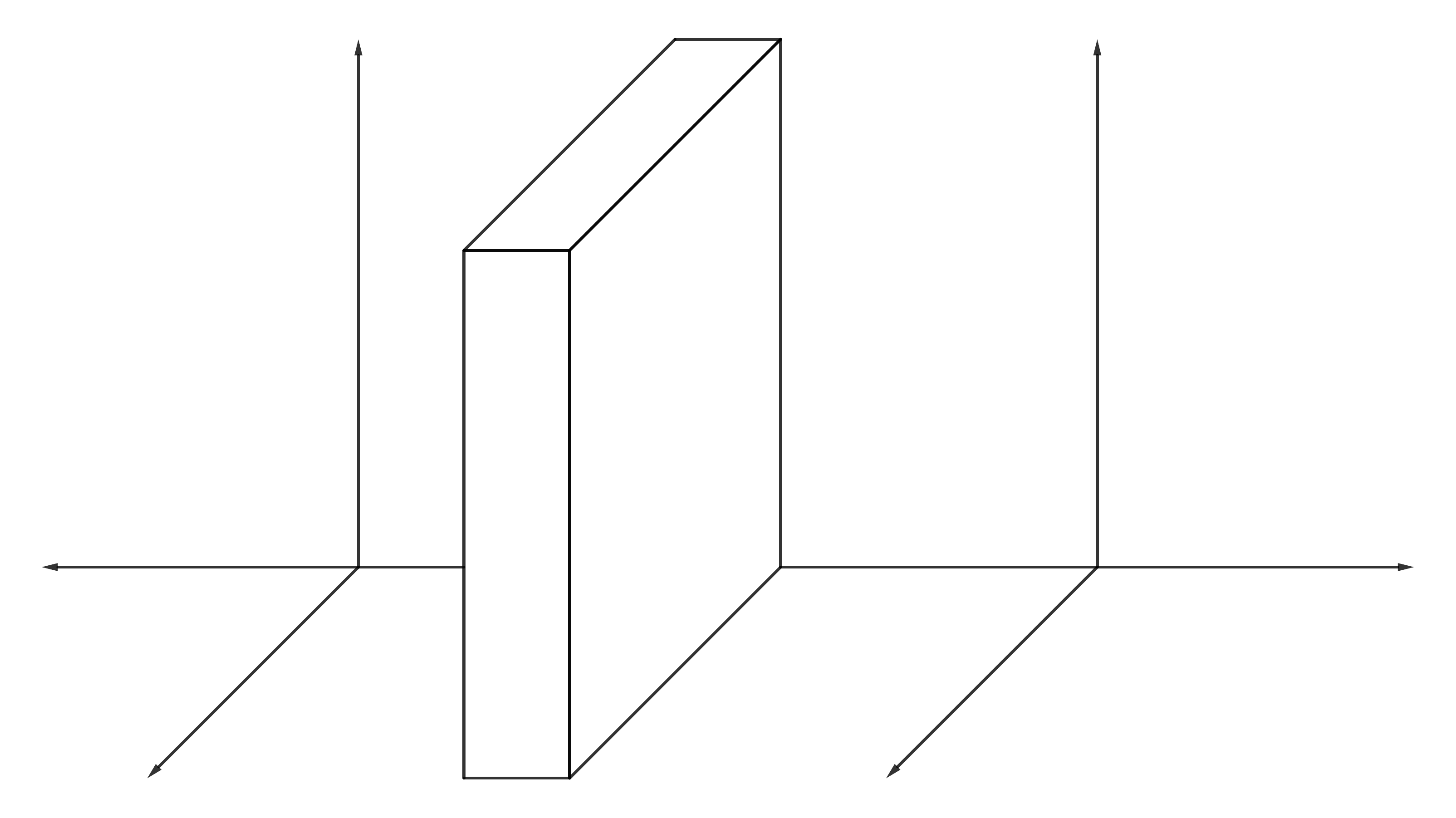}
\caption{Two twist flats $\M(X)$ and $\M(X')$, which are copies of $\R^3$, are glued to the switch bridge $\M(X,X')$, which is a copy of $\R^2\times [-\frac{1}{2},\frac{1}{2}]$, along copies of $\R^2$ contained in the $2$-skeletons of $\M(X)$ and $\M(X')$. Here $X = \{\alpha_1,\alpha_2,\alpha_3\}$ and $X' = \{\alpha_1,\delta_2,\alpha_3\}$.}
\label{figure:twistflats-switchbridge}
\end{figure}

Hamenst\"adt and Hensel show that there is a surjection $p$ from $\M$ onto a tree called the \emph{non-separating meridional pants graph}, which we refer to in this paper as $\Pnm$. The vertices of $\Pnm$ correspond to pants decompositions $X$ of non-separating meridians, and there is an edge between two pants decompositions if they intersect minimally, ie if they intersect twice. The map $p:\M \to \Pnm$ maps each twist flat $\M(X)$ to the vertex $X\in \Pnm^{(0)}$ and each switch bridge $\M(X,X')$ to the edge between $X$ and $X'$ in $\Pnm$.

Additionally, given a cut system $Z$ on $V_2$, Hamenst\"adt and Hensel define the induced subgraph $\Pnm(Z)$ of $\Pnm$ as the subgraph with vertices corresponding to pants decompositions containing the cut system $Z$. They show in \cite[Corollary 5.11]{hamHenDehn} that $\Pnm(Z)$ is a tree, and that for any distinct cut systems $Z\neq Z'$, the subtrees $\Pnm(Z)$ and $\Pnm(Z')$ intersect in at most a single point.

Let us define a similar subgraph $\Pnm(\alpha)$ of $\Pnm$ as the induced subgraph with vertices corresponding to pants decompositions that contain the non-separating meridian $\alpha$. The following lemma regarding $\Pnm(\alpha)$ will be useful in the discussion of parallelism classes of hyperplanes in $\M$. It should be noted that a proof of this lemma is also sketched in the proof of \cite[Corollary 5.18]{hamHenDehn}.

\begin{lemma}\label{lem:commoncurvepath}
    Let $\alpha$ a non-separating meridian on $V_2$. Then $\Pnm(\alpha)$ is a subtree of $\Pnm$.
\end{lemma}
\begin{proof}
    Since $\Pnm$ is a tree, it suffices to show that $\Pnm(\alpha)$ is connected. Suppose $X = \{\alpha,\beta_1,\beta_2\}$ and $X' = \{\alpha,\delta_1,\delta_2\}$ are two distinct pants decompositions of non-separating meridians. Since the pants decompositions are distinct but contain a common curve $\alpha$, one of the $\beta_i$ must intersect some $\delta_i$. Say $\beta_1 \cap (\delta_1\cup\delta_2) \neq \emptyset$. Let $Z = \{\alpha,\beta_1\}$, which is a cut system. There is a surgery sequence $(Z_i)^n_{i=1}$ starting from $Z_1 = Z$ in the direction $X'$. The final cut system $Z_n$ is disjoint from $X'$, and since $X'$ is a pants decomposition, it must be that $Z_n \subset X'$.
    
    Since $\alpha\cap \beta_i = \emptyset$ and $\alpha\cap \delta_i = \emptyset$ for each $i$, the meridian surgeries performed to attain $(Z_i)^n_{i=1}$ will never be performed on $\alpha$. In particular, every cut system $Z_i$ must contain the meridian $\alpha$. For a fixed $i\in[2,n-1]$, the unions $Z_i\cup Z_{i-1}$ and $Z_i\cup Z_{i+1}$ are vertices in $\Pnm(Z_i)$, since consecutive cut systems have no transverse intersections. Because $\Pnm(Z_i)$ is a tree and because $\alpha\in Z_i$, there is a path $\gamma_i\subset \Pnm(Z_i)$ connecting $Z_i\cup Z_{i-1}$ to $Z_i\cup Z_{i+1}$ such that every vertex $\gamma_i(j)$ contains $\alpha$. 
    
    Furthermore, both $X$ and $Z_1\cup Z_2$ are vertices in $\Pnm(Z_1)$, so there is a path $\gamma_1\subset\Pnm(Z_1)$ connecting $X$ to $Z_1\cup Z_2$ such that each vertex $\gamma_1(j)$ contains $\alpha$. Similarly, since $Z_n\subset X'$, it follows that $Z_{n-1}\cup Z_n$ and $X'$ are vertices in $\Pnm(Z_n)$, and hence there is a path $\gamma_n\subset \Pnm(Z_n)$ connecting $Z_{n-1}\cup Z_n$ to $X'$ such that every vertex $\gamma_n(j)$ contains $\alpha$.
    
    The path $\gamma$ constructed by concatenating the paths $\gamma_i$ for $i\in[1,n]$ is a path from $X$ to $X'$ contained entirely in $\Pnm(\alpha)$. This implies that $\Pnm(\alpha)$ is a connected subgraph of the tree $\Pnm$, and hence $\Pnm(\alpha)$ is a subtree.
\end{proof}

%%%%%%%%%%%%%%%%%%%%%%%%%%%%%%%%%%%%%%%%%
\subsection{Hyperplanes}\label{subsec:hyperplanes}

In this subsection, we examine the variants of hyperplanes  found in $\M$, and their associated combinatorial hyperplanes. Understanding the hyperplanes will prove useful in the discussion of the contact graph $\mathcal{C}\M$, and understanding the combinatorial hyperplanes is necessary to prove that $\M$ is a hierarchically hyperbolic space.

\begin{lemma}\label{lem:hypetype}
    Let $H$ be a hyperplane in $\M$. Then $H$ fits into exactly one of the following types.
    \begin{enumerate}
        \item All of the midcubes comprising $H$ are contained in switch cubes. Consequently, $H$ is entirely contained in a switch bridge and is isometric to $\R^2.$
        
        \item $H$ comprises midcubes contained in switch cubes and midcubes contained in twist cubes. Moreover, for each twist flat $\M(X)$ that has non-empty intersection with $H$, $\M(X)\cap H$ is isometric to $\R^2$ and is parallel to one of the coordinate planes in $\M(X)$.
    \end{enumerate}
\end{lemma}

\begin{proof}
    Let $H$ be a hyperplane of $\M$. We recall that any midcube is contained in a unique hyperplane, and hence a hyperplane in $\M$ is determined by its intersection with a single $3$-cube. Specifically, we can distinguish between the two proposed types of hyperplanes by examining how they intersect a single switch cube. We note that every hyperplane must intersect at least one switch cube. This is because if it did not, it would intersect only twist cubes. However, each twist cube has a switch cube glued to each of it's 6 faces and hence any hyperplane that intersects that twist cube must also intersect four of the adjoining switch cubes.
    
    Let $s$ be a switch cube that has non-empty intersection with $H$. Recall that every switch cube $s$ contains both twist and switch edges. Further, the twist edges in $s$ form the $1$-skeletons of two opposite faces of $s$, and these two faces are connected to one another via four parallel switch edges. The switch cube $s$ is contained inside of a unique switch bridge, say $\M(X,X')$, which is isometric to $\R^2\times [-\frac{1}{2},\frac{1}{2}]$. We will show that the two cases in the Lemma come from whether or not $H$ intersects the switch edges of $s$. Note that if $H$ crosses one switch edge of $s$, it must cross all the switch edges of $s$ and in fact all of the switch edges in $\M(X,X')$.
    
    Suppose first that $H$ intersects $s$ so that it only crosses the switch edges of $s$. With this being the case, it follows that $H$ must be contained entirely within $\M(X,X')$ and in fact $\M(X,X')$ is the carrier of $H$. Thus, in this case, $H$ is isometric to $\R^2$.
    
    For the second case, suppose $H$ does not intersect any of the switch edges of $\M(X,X')$. Then $H$ must cross the two faces of $s$ which contain only twist edges. These two faces are glued to twist cubes inside of the twist flats $\M(X)$ and $\M(X')$. Since twist flats are isometric to the standard cubulation of $\R^3$, when $H$ has non-empty intersection with a twist flat, that intersection must be a plane that is parallel to one of the coordinate planes, and hence the intersection is isometric to $\R^2$.
\end{proof}

\begin{definition}
    If a hyperplane is of the first type described in Lemma \ref{lem:hypetype}, then we will call it a \emph{switch hyperplane} and we will call its associated combinatorial hyperplanes \emph{combinatorial switch hyperplanes.} If a hyperplane is of the second type described in Lemma \ref{lem:hypetype}, then we will call it a \emph{twist hyperplane}, and we will refer to the associated combinatorial hyperplanes as \emph{combinatorial twist hyperplanes.}
\end{definition}

We will denote the switch hyperplane contained in a switch bridge $\M(X,X')$ by $H(X,X')$. Now suppose that $X = \{\alpha_1,\alpha_2,\alpha_3\}$ and that $\Delta = \{\delta_1,\delta_2,\delta_3\}$ is a dual ststen to $X$ such that $(X',\Delta')$ is obtained from $(X,\Delta)$ by switching $\alpha_2$. Then the two combinatorial switch hyperplanes associated to $H(X,X')$ will be denoted $C(X,\delta_2)$ and $C(X',\alpha_2)$. Here, $C(X,\delta_2)$ is the combinatorial hyperplane in $\M(X)$ that contains all vertices of the form $(X,\{T_{\alpha_1}^{n_1}(\delta_1), \delta_2, T_{\alpha_3}^{n_3}(\delta_3)\}).$ Similarly, $C(X',\alpha_2)$ is the combinatorial hyperplane in $\M(X')$ that contains all vertices of the form $(X',\{c(T_{\alpha_1}^{n_1}(\delta_1)), \alpha_2, c(T_{\alpha_3}^{n_3}(\delta_3))\}).$

To describe twist hyperplanes and the notation we use, we make two more definitions. Suppose that $\alpha$ and $\delta$ are non-separating meridians; we say that the pair \emph{$(\alpha,\delta)$ is a part of} a vertex $(X,\Delta)\in \M^{(0)}$ if $\alpha\in X$, $\delta\in \Delta$, and $\delta$ is the dual to $\alpha.$ Now suppose there is some vertex $(X,\Delta)\in \M^{(0)}$ such that $(\alpha,\delta)$ is a part of $(X,\Delta)$. We say that a non-separating meridian \emph{$\delta'$ is equivalent to $\delta$} if the following holds:
there exists a sequence of non-separating meridians $\delta = \delta_0, \delta_1, \dots, \delta_n = \delta'$ and a sequence of vertices $(X,\Delta) = (X_0,\Delta_0), (X_1,\Delta_1), \dots, (X_n,\Delta_n)$ such that
\begin{enumerate}
    \item $(\alpha,\delta_i)$ is a part of $(X_i,\Delta_i)$ for all $i$, and
    \item $(X_i,\Delta_i)$ is obtained from $(X_{i-1},\Delta_{i-1})$ by switching one of the curves in $X_{i-1}-\{\alpha\}$, or $(X_i,\Delta_i) = T^j_{\beta}((X_{i-1},\Delta_{i-1}))$ where $\beta\in X_{i-1}-\{\alpha\}$ and $j\in \Z$.
\end{enumerate}
It is clear that this is an equivalence relation, and we denote the equivalence class of $\delta$ as $[\delta]_{\alpha}.$ We will often drop the subscript and simply write $[\delta] = [\delta]_{\alpha}$ to avoid cluttering notation when it is clear from context what the corresponding pants curve should be. We also note that $T_{\alpha}([\delta]_{\alpha}) = [T_{\alpha}(\delta)]_{\alpha}$.

The following lemma shows that a (combinatorial) twist hyperplane is determined by a non-separating curve $\alpha$ and an associated equivalence class $[\delta]_{\alpha}$.

\begin{lemma}\label{lem:twisthypcurves}
    Suppose $H$ is a twist hyperplane. Then there exists non-separating meridians $\alpha_1$ and $\delta_1$ such that for every vertex $(X,\Delta)\in N(H)^{(0)}$, there is some $\delta'$ in $[\delta_1]$ or in $[T_{\alpha_1}(\delta_1)]$ such that $(\alpha_1,\delta')$ is a part of $(X,\Delta)$. Furthermore, $H$ is the unique hyperplane with this property. Moreover $p(H) = p(H^{\pm})= \Pnm(\alpha_1)$.
\end{lemma}
\begin{proof}
    Suppose $H$ is a twist hyperplane. Let $X = \{\alpha_1, \alpha_2, \alpha_3\}$ be a pants decomposition of non-separating meridians such that $H$ has non-empty intersection with $\M(X)$. As described in Lemma \ref{lem:hypetype} (2), the intersection of $H$ with $\M(X)$ is isometric to $\R^2$, and specifically is parallel to one of the coordinate planes of $\M(X).$ Consequently, for one of the curves in $X$, say $\alpha_1$, there is some dual curve $\delta_1$ to $\alpha_1$ such that all of the vertices of $N(H) \cap \M(X)$ are of one of the following forms:
    \[
        (X, T^{n_2}_{\alpha_2}T^{n_3}_{\alpha_3}(\Delta)) \text{ or } (X, T^{n_2}_{\alpha_2}T^{n_3}_{\alpha_3}(T_{\alpha_1}(\Delta))), 
    \]
    where $\Delta$ is a dual system to $X$ which contains $\delta_1$, and $n_2,n_3\in\Z.$ (Essentially, $N(H) \cap \M(X)$ is parallel to the $T_{\alpha_2}T_{\alpha_3}$-coordinate plane). Notice that each of these vertices contains $\alpha_1$ as a pants curve, and either $\delta_1$ or $T_{\alpha_1}(\delta_1)$ as a dual to $\alpha_1$. In fact, every vertex in $N(H)$ must contain $\alpha_1$ because the only switch bridges that can be crossed by $H$ are those that correspond to switching in the direction of $\alpha_2$ or $\alpha_3$. It then follows that the only dual curves that can be found for $\alpha_1$ in $N(H)$ must be obtained from the cleanup procedure associated to switching pants curves other than $\alpha_1$. The collection of dual curves that can be obtained in this way is exactly described by the equivalence classes $[\delta_1]$ and $[T_{\alpha_1}(\delta_1)]$.
    
    In fact, $H$ must be the unique twist hyperplane for which every vertex in its carrier contains $\alpha_1$ as a pants curve and for which the duals to $\alpha_1$ in the carrier are exactly those non-separating meridians in $[\delta_1]$ and $[T_{\alpha_1}(\delta_1)]$. To see this, let $K$ be some other twist hyperplane with these same properties. Notice that by the definition of dual curves, $\bd V_2 - (\alpha_1\cup \delta_1)$ is a disjoint union of two annuli $A_1\cup A_2$, so if $(\alpha_1,\delta_1)$ is the part of some vertex $(X,\Delta)$, then the other pants curves in $X- \{\alpha_1\}$ must be contained in $A_1\cup A_2$. However, each of $A_1$ and $A_2$ contains a unique meridian, so in fact if we know $(\alpha_1,\delta_1)$ is a part of $(X,\Delta)$, then the other two pants curves in $X$ are uniquely determined by the pair $(\alpha_1,\delta_1)$. By a similar argument, if $(\alpha_1,T_{\alpha_1}(\delta_1))$ is a part of a vertex $(X',\Delta')$, then in fact $X' = X$. Because $(\alpha_1,\delta_1)$ and $(\alpha_1,T_{\alpha_1}(\delta_1))$ are parts of vertices in $N(K)$, it then follows that $N(K) \cap \M(X) \neq \emptyset$. Moreover, 
    \[ C(X,\delta_1) \cup C(X,T_{\alpha_1}(\delta_1)) = (K^- \cup K^+) \cap \M(X) = (H^- \cup H^+) \cap \M(X).\]
    Since hyperplanes in $\M$ are determined by their intersections with any single $3$-cube, it follows that $H = K$.
    
    It remains to show that $p(H) = p(H^{\pm}) = \Pnm(\alpha_1)$. Because we know that every vertex in $H$ must contain $\alpha_1$ as a pants curve, it follows that $p(H) \subset \Pnm(\alpha_1)$. We can conclude that $p(H) = \Pnm(\alpha_1)$, (and hence $p(H^{\pm}) = \Pnm(\alpha_1)$), because otherwise there would be some pants decomposition $X'$ containing $\alpha_1$ which cannot be reached from $X$ via a series of switches not involving $\alpha_1$; if this was the case, then $\Pnm(\alpha_1)$ could not be a subtree of $\Pnm$, contradicting Lemma \ref{lem:commoncurvepath}.
\end{proof}

Via Lemma \ref{lem:twisthypcurves}, we are then justified in the notation $H(\alpha_1,[\delta_1])$ for the twist hyperplanes determined by $\alpha_1$ and $[\delta_1]$. We will denote the two combinatorial twist hyperplanes associated to a twist hyperplane $H(\alpha_1,[\delta_1])$ by $C(\alpha_1,[\delta_1])$ and $C(\alpha_1,[T_{\alpha_1}(\delta_1)])$. Here, $C(\alpha_1,[\delta_1])$ is the combinatorial hyperplane associated to $H(\alpha_1,[\delta_1])$ such that the duals to $\alpha_1$ are contained in $[\delta_1]$, and $C(\alpha_1,[T_{\alpha_1}(\delta_1)])$ is the combinatorial hyperplane such that the duals to $\alpha_1$ are contained in $[T_{\alpha_1}(\delta_1)]$.

In Figure \ref{figure:small-twisthyperplanes}, one can see two illustrations of the local structure of a twist hyperplane. Figure \ref{figure:twisthyperplanes} illustrates the non-empty intersection of a twist hyperplane with several twist flats.

\begin{figure}[htb]
\centering
\labellist
\small\hair 2pt
\pinlabel {$T_{\alpha_2}$} [ ] at 455 310
 \pinlabel {$T_{\alpha_1}$} [ ] at 150 110
 \pinlabel {$T_{\alpha_3}$} [ ] at 343 815
 \pinlabel {$\M(X)$} [ ] at 203 570
 \pinlabel {$\M(X')$} [ ] at 1345 570
 \pinlabel {$T_{\alpha_1}$} [ ] at 935 110
 \pinlabel {$T_{\delta_2}$} [ ] at 1456 310
 \pinlabel {$T_{\alpha_3}$} [ ] at 1125 825
 \pinlabel {$\M(X,X')$} [ ] at 960 460
 \pinlabel {$H(X,X')$} [ ] at 960 605
\endlabellist
\begin{tikzpicture}
\draw (0, 0) node[inner sep=0]
{\includegraphics[scale=.23]{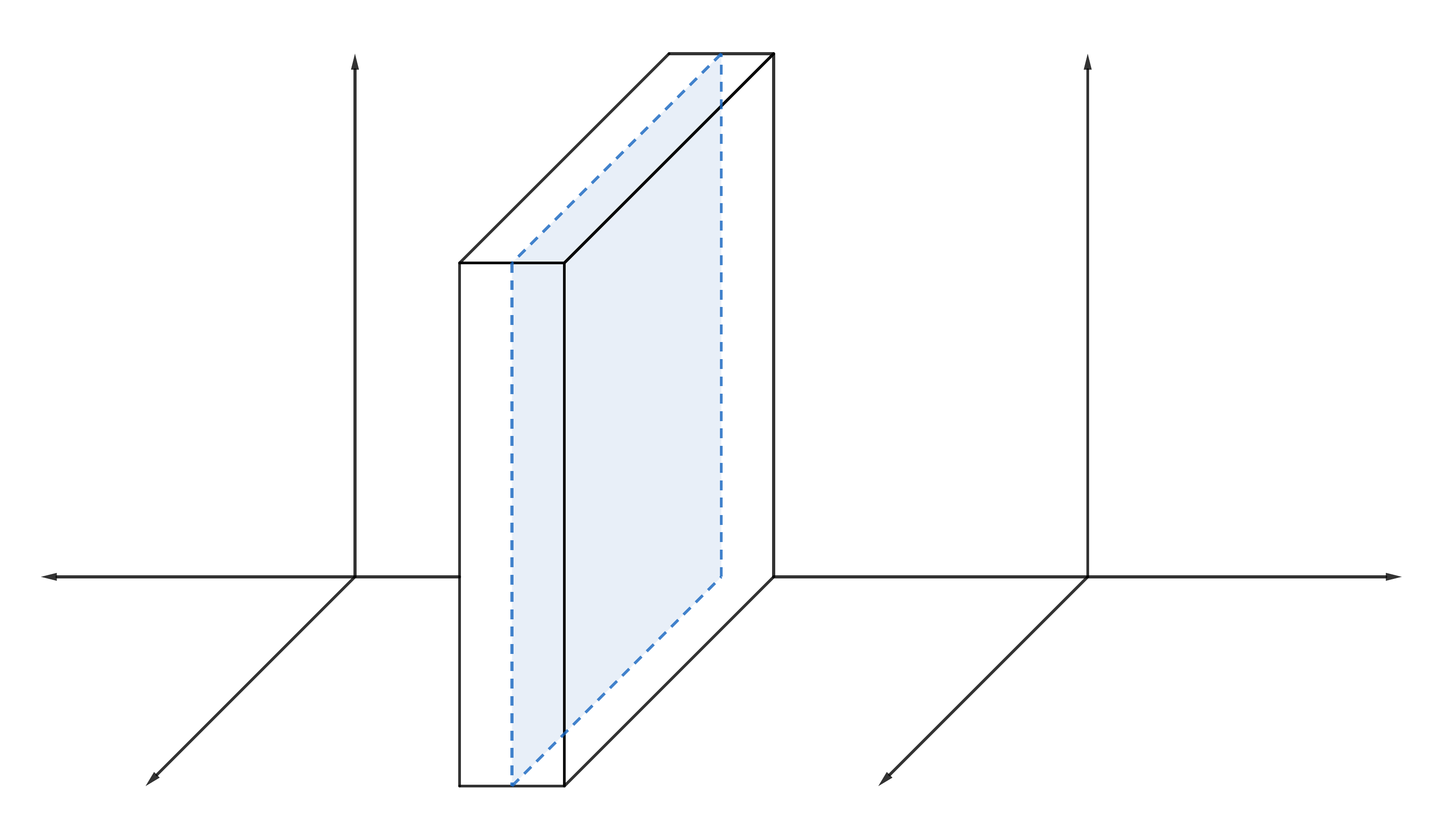}};
\draw [->] (.9,1.45) to [out=150,in=30] (-0.6,1.45);
\draw [->] (.9,-.1) to [out=-150,in=-30] (.15,-.1);
\end{tikzpicture}
\caption{A switch hyperplane $H(X,X')$ is entirely contained in the switch bridge $\M(X,X')$. Here $X = \{\alpha_1,\alpha_2,\alpha_3\}$ and $X' = \{\alpha_1,\delta_2,\alpha_3\}$. }
\label{figure:switchhyperplanes}
\end{figure}

\begin{figure}[htb]
\centering
\begin{subfigure}{.5\textwidth}
    \centering
    \labellist
    \small\hair 2pt
     \pinlabel {twist} [ ] at 290 830
     \pinlabel {twist} [ ] at 580 830
     \pinlabel \rotatebox{45}{switch} [ ] at 675 140
    \endlabellist
    
    \includegraphics[scale=.18]{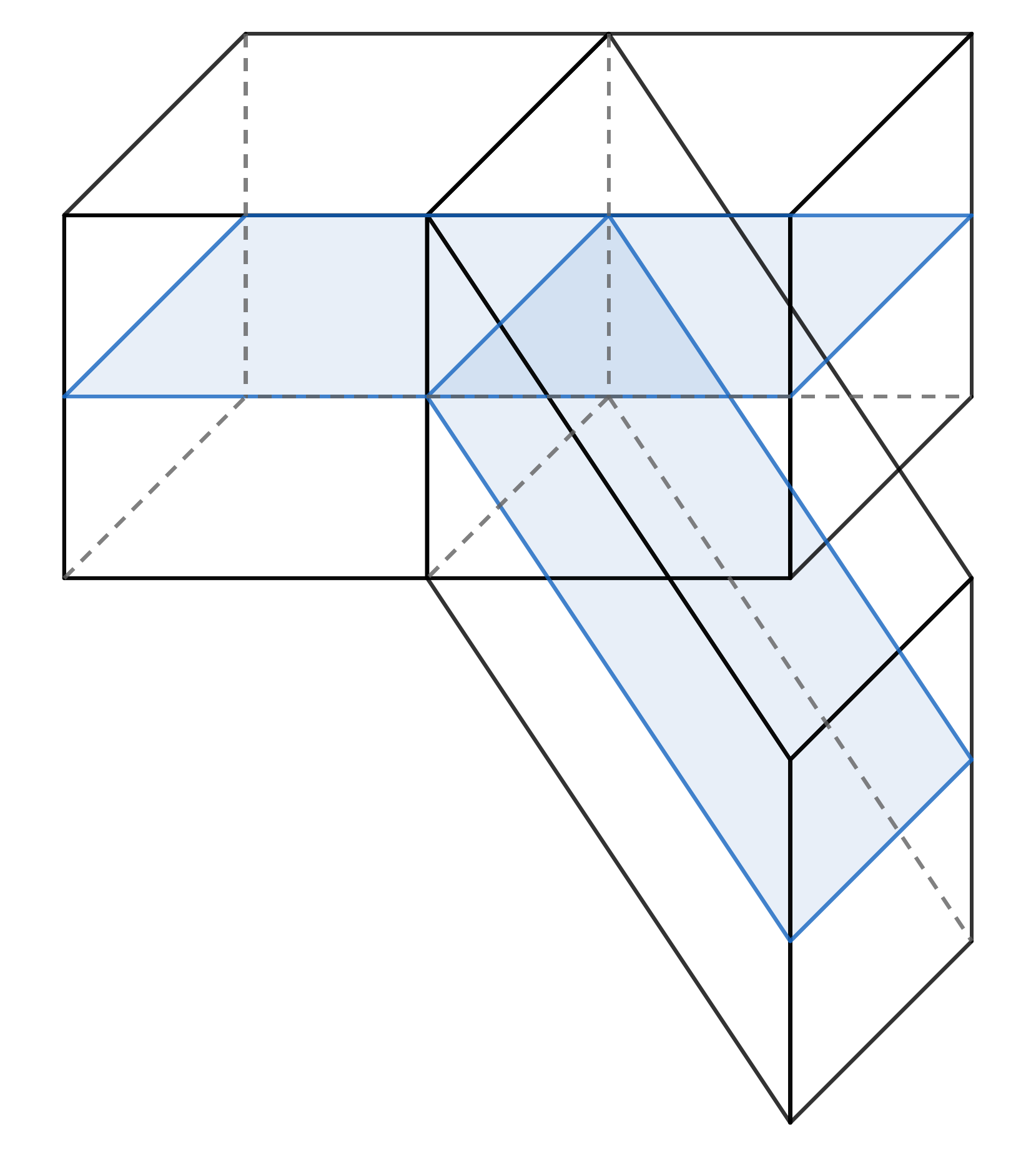} \hspace{1cm}
\end{subfigure}%
\begin{subfigure}{.5\textwidth}
\centering
    \labellist
    \small\hair 2pt
     \pinlabel {$m$} [ ] at 550 410
     \pinlabel {twist} [ ] at 1000 705
     \pinlabel {twist} [ ] at 30 165
     \pinlabel {switch} [ ] at 300 750
     \pinlabel {switch} [ ] at 790 190
    \endlabellist
    
    \begin{tikzpicture}
    \draw (0, 0) node[inner sep=0] {\includegraphics[scale=.18]{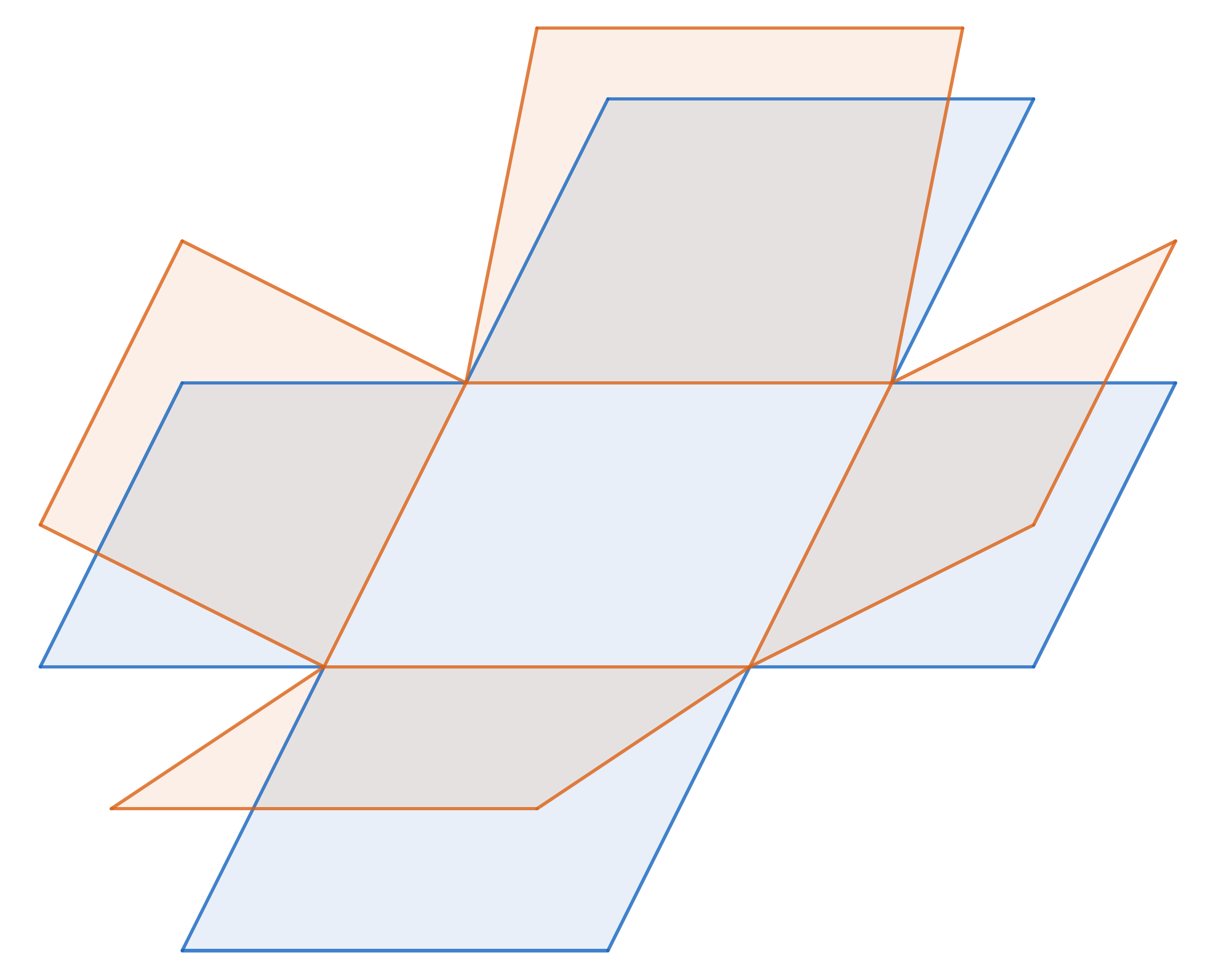}};
    \draw [->] (-1.5,1.6) to [out=240, in=120] (-1.6,.8);
    \draw [->] (-1.1,1.9) to [out=30, in=150] (-.3,1.9);
    \draw [->] (2.35,1.7) to [out=150, in=30]  (1.85,1.7);
    \draw [->] (3.05,1.55) to [out=-30, in=60] (2.9,.45);
    \draw [->] (1.4,-1.35) to [out=120,in=240] (1.4,-.4);
    \draw [->] (.9,-1.5) to [out=150,in=30] (-.6,-1.5);
    \draw [->] (-2.9,-1.85) to [out=-30, in=200]  (-1.7,-1.9);
    \draw [->] (-3.1,-1.5) to [out=60, in=210]  (-2.4,-.8);
    \end{tikzpicture}
\end{subfigure}

\caption{Pictured here are two representations of the local structure of a twist hyperplane. Any $2$-cube that is crossed by a twist hyperplane is contained in three $3$-cubes: two twist cubes contained in a single twist flat, and one switch cube, (pictured left). On the right we see that if $m$ is a midcube of a twist hyperplane that is contained in a twist cube, then it is connected to eight other midcubes: four midcubes contained in disjoint switch bridges, (orange), and four midcubes contained in twist cubes in the same twist flat as $m$, (blue).}
\label{figure:small-twisthyperplanes}
\end{figure}

\begin{figure}[htb]
\centering
\labellist \small\hair 2pt
 \pinlabel {$T_{\alpha_3}$} [ ] at 830 718
 \pinlabel {$T_{\alpha_1}$} [ ] at 665 148
 \pinlabel {$T_{\alpha_2}$} [ ] at 920 260
 \pinlabel {$H(\alpha_3,[\delta_3])$} [ ] at 249 690
 \pinlabel {$T_{\alpha_3}(\delta_3)$} [ ] at 775 561
 \pinlabel {$\delta_3$} [ ] at 820 471
 
  \pinlabel {$\M(X')$} [ ] at 188 250
 \pinlabel {$\M(X'')$} [ ] at 1320 250
 \pinlabel {$\M(X)$} [ ] at 620 220
 \pinlabel {$\M(X''')$} [ ] at 980 833
 
 \pinlabel {$\M(X,X')$} [ ] at 358 60
 \pinlabel {$\M(X,X'')$} [ ] at 994 60
 \pinlabel {$\M(X,X''')$} [ ] at 675 833
\endlabellist
\begin{tikzpicture}
\draw (0, 0) node[inner sep=0]
{\includegraphics[scale=.25]{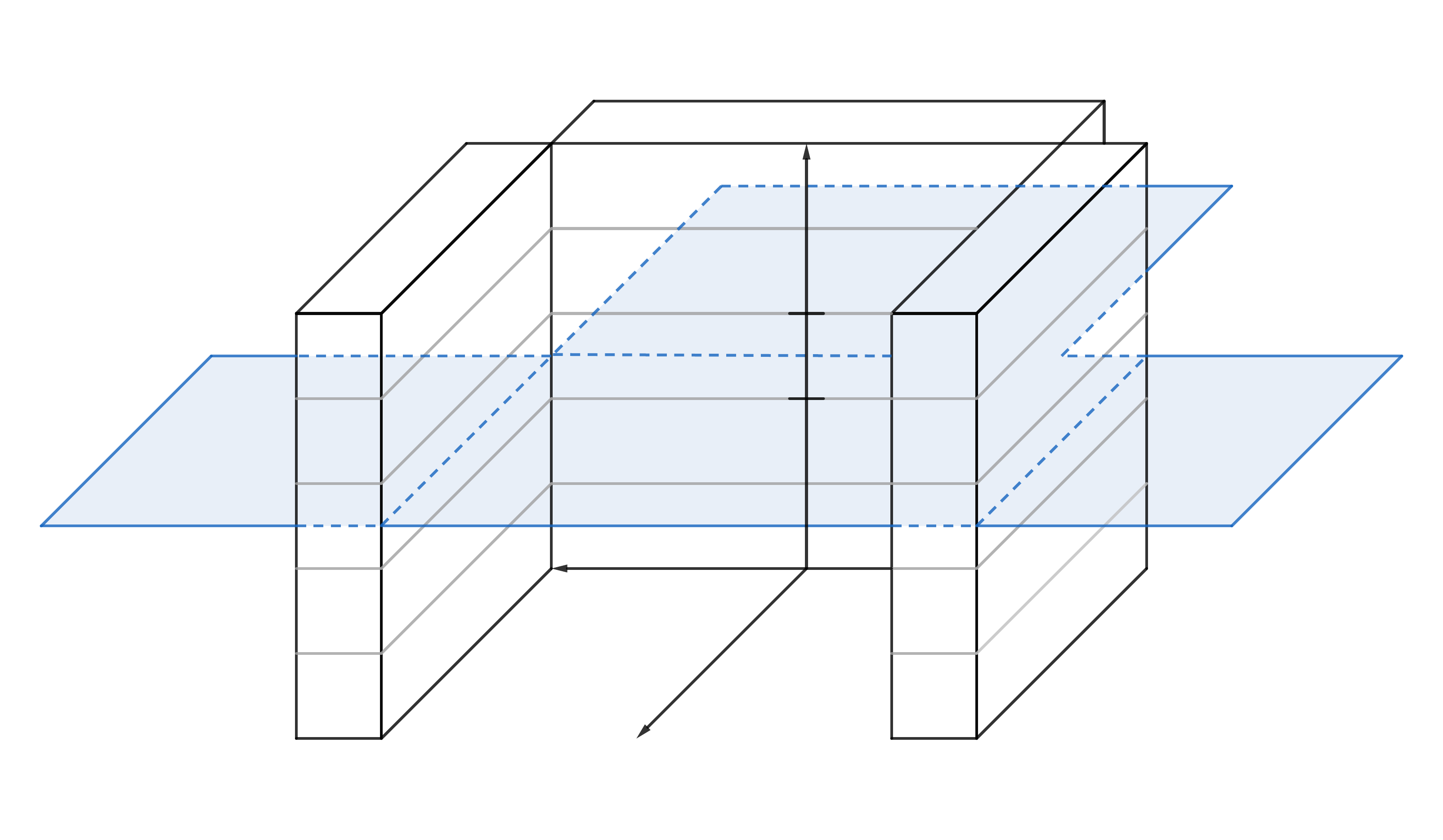}};
\draw [->] (-4.5,1.9) to  (-4.5,0);
\draw [->] (-1.4,3.2) to [out=260, in=150] (-1,2.7);
\draw [->] (-3.6,-3.2) to  (-3.6,-2.8);
\draw [->] (2,-3.2) to  (2,-2.8);
\end{tikzpicture}

\caption{A twist hyperplane $H(\alpha_3,[\delta_3])$ crossing four twist flats $\M(X)$, $\M(X')$, $\M(X'')$ , and $\M(X''')$, and three switch bridges $\M(X,X')$, $\M(X,X'')$, and $\M(X,X''')$. Notice also that $N(H(\alpha_3,[\delta_3]))$ contains vertices with dual curves $\delta_3$ and $T_{\alpha_3}(\delta_3)$.}
\label{figure:twisthyperplanes}
\end{figure}

%%%%%%%%%%%%%%%%%%%%%%%%%%%%%%%%%%%%%%%%%
\subsection{Parallelism classes of combinatorial hyperplanes}\label{subsec:parallelhyperplanes}

In this section, we determine the parallelism classes for the two types of combinatorial hyperplanes in $\M$. Recall that the parallelism class of a combinatorial hyperplane $C$ is the equivalence class of all convex subcomplexes that are parallel to $C$. In a general $\cat$ cube complex, it is not necessarily true that all convex subcomplexes that are parallel to a combinatorial hyperplane are themselves combinatorial hyperplanes, but we will see that this is the case in $\M$.

\begin{lemma}\label{lem:twistcombparallel}
The parallelism class for a combinatorial twist hyperplane $C(\alpha_1,[\delta_1])$ is the collection $[C(\alpha_1,[\delta_1])] = \bigcup_{k\in\Z} \{ T_{\alpha_1}^k(C(\alpha_1,[\delta_1]))\}$. In other words, $[C(\alpha_1,[\delta_1])]$ consists of all combinatorial twist hyperplanes whose image under the projection $p:\M\to\Pnm$ is $\Pnm(\alpha_1)$.
\end{lemma}
\begin{proof}
    Let $C(\alpha_1,[\delta_1])$ be any combinatorial twist hyperplane, and fix any $F\in [C(\alpha_1,[\delta_1])]$. Because $p(C(\alpha_1,[\delta_1])) = \Pnm(\alpha_1)$, (via Lemma \ref{lem:twisthypcurves}), we know the switch bridges crossed by $C(\alpha_1,[\delta_1])$ correspond exactly to the edges of $\Pnm(\alpha_1)$.
    Because $F$ is parallel to $C(\alpha_1,[\delta_1])$, $F$ must also cross every switch bridge corresponding to an edge of $\Pnm(\alpha_1)$. Moreover, because $F$ intersects the same switch hyperplanes as $C(\alpha_1,[\delta_1])$, it must also intersect all of the same twist flats. Therefore, $p(F) = \Pnm(\alpha_1)$.
    
    We also know that $C(\alpha_1,[\delta_1])$ does not cross the twist hyperplanes $H(\alpha_1,[T^{n_1}_{\alpha_1}(\delta_1)])$, where $n_1\in \Z$. This means that $F$ must be contained in some combinatorial twist hyperplane $C(\alpha_1,[T^{n_1}_{\alpha_1}(\delta_1)])$. Via the cubical isometric embedding described in Lemma \ref{lem:HHS-2.4}, $F$ must be an entire twist hyperplane $C(\alpha_1,[T^{n_1}_{\alpha_1}(\delta_1)])$. Therefore, $F$ is a combinatorial twist hyperplane with $p(F) = \Pnm(\alpha_1)$. The combinatorial twist hyperplanes whose projections to $\Pnm$ is $\Pnm(\alpha_1)$ are exactly the combinatorial hyperplanes $\bigcup_{k\in\Z} \{ T_{\alpha_1}^k(C(\alpha_1,[\delta_1]))\}$. Hence, the parallelism class of $C(\alpha_1,[\delta_1])$ corresponds exactly to the combinatorial twist hyperplanes in $\bigcup_{k\in\Z} \{ T_{\alpha_1}^k(C(\alpha_1,[\delta_1]))\}$.
\end{proof}

\begin{cor}\label{cor:twistsimplytrans}
    The action of $\langle T_{\alpha_1} \rangle$ on $[C(\alpha_1,[\delta_1])]$ is simply transitive, (ie is transitive and free). Consequently, the set $[C(\alpha_1,[\delta_1])]$ has infinite cardinality.
\end{cor}
\begin{proof}
    By Lemma \ref{lem:twistcombparallel}, if $C(\alpha_1,[\delta_1])$ and  $C(\alpha_1,[\delta_1'])$ are parallel twist hyperplanes, then we know that $C(\alpha_1,[\delta_1']) = C(\alpha_1,[T_{\alpha_1}^{n_1}(\delta_1)])$ for exactly one $n_1\in\Z$. Hence, the action of $\langle T_{\alpha_1}\rangle$ on $[C(\alpha_1,[\delta_1])]$ is simply transitive.
    
    Since $\langle T_{\alpha_1}\rangle$ has infinite order and acts simply transitively on $[C(\alpha_1,[\delta_1])]$, it follows that $[C(\alpha_1,[\delta_1])]$ has infinite cardinality.
\end{proof}

\begin{lemma}\label{lem:switchcombparallel}
    Let $X = \{\alpha_1,\alpha_2,\alpha_3\}$ be a pants decomposition of non-separating meridians, and let $\delta_1$ be a dual to $\alpha_1$. The parallelism class for the switch combinatorial hyperplane $C(X,\delta_1)$ consists of all combinatorial switch hyperplanes that correspond to edges in $\Pnm(\{\alpha_2,\alpha_3\})$.
\end{lemma}
\begin{proof}
    Because every switch hyperplane is contained in its own switch bridge and because combinatorial switch hyperplanes do not cross switch bridges, $C(X,\delta_1)$ does not cross any switch hyperplanes. This means that if $F\in [C(X,\delta_1)]$, it must be contained in a single twist flat. Additionally, $F$ does not cross any twist hyperplanes of the form $H(\alpha_1,[\delta_1'])$, where $\delta_1'$ is any dual to $\alpha_1$. We can see this via the definition of parallelism and the fact that $C(X,\delta_1)$ does not cross any hyperplanes of the form $H(\alpha_1,[\delta_1'])$. Since $F$ must be contained in a single twist flat and does not cross any twist hyperplane of the form $H(\alpha_1,[\delta_1'])$, it follows that $F$ must be contained in a combinatorial switch hyperplane of the form $C(X',\delta_1')$ where $X'\in\Pnm(\{\alpha_2,\alpha_3\})^{(0)}$ and $\delta_1'$ is a dual to $X'-\{\alpha_1,\alpha_2\}$. Via the cubical isometric embedding described in Lemma \ref{lem:HHS-2.4}, $F$ must actually be an entire combinatorial hyperplane $C(X',\delta_1')$, ie a combinatorial switch hyperplane corresponding to an edge in $\Pnm(\{\alpha_2,\alpha_3\})$.
    
    The two families of hyperplanes that have non-empty intersection with $C(X,\delta_1)$ are twist hyperplanes of the form $H(\alpha_2,[\delta_2])$ and $H(\alpha_3,[\delta_3])$, where $\delta_2$ and $\delta_3$ are any duals to $\alpha_2$ and $\alpha_3$ respectively. So in fact the parallelism class of $C(X,\delta_1)$ contains all switch combintorial hyperplanes corresponding to edges in $p(H(\alpha_2,[\delta_2]) \cap H(\alpha_3,[\delta_3])) =\Pnm(\{\alpha_2,\alpha_3\})$.
\end{proof}

\begin{cor}\label{cor:switchinfcard}
    Let $X$ and $\delta_1$ be as in Lemma \ref{lem:switchcombparallel}. The parallelism class $[C(X,\delta_1)]$ has infinite cardinality.
\end{cor}
\begin{proof}
    By Lemma \ref{lem:switchcombparallel}, we know that elements in $[C(X,\delta_1)]$ correspond exactly to edges in $\Pnm(\alpha_2,\alpha_3).$ We know that $\Pnm(\alpha_2,\alpha_3)$ must have infinite valence because there are infinitely many non-separating meridians disjoint from $\alpha_2 \cup \alpha_3.$ Since $\Pnm(\alpha_2,\alpha_3)$ has infinite valence, it must also have infinitely many edges. Henec, the cardinality of $[C(X,\delta_1)]$ is infinite.
\end{proof}

%%%%%%%%%%%%%%%%%%%%%%%%%%%%%%%%%%%%%%%%%
\subsection{Edges in the contact graph}\label{subsec:edgescontact}

Recall that if there is an edge between two hyperplanes $H_1$ and $H_2$ in a contact graph. then either $H_1$ and $H_2$ cross or they osculate. In this section we characterize all of the possible ways that we could have an edge between two hyperplanes in the contact graph.

\begin{lemma}\label{lem:edgescontact}
    Let $H_1$ and $H_2$ be two hyperplanes in $\M$. If there is an edge between $H_1$ and $H_2$ in $\contact{\M}$, then exactly one of the following holds.
    \begin{enumerate}
        \item $H_1$ and $H_2$ are two switch hyperplanes that osculate. In this case, they must be adjacent to a common twist flat and are not parallel.
        
        \item $H_1$ and $H_2$ are twist hyperplanes that osculate. In this case, $H_1$ and $H_2$ must actually be parallel to one another.
        
        \item $H_1$ and $H_2$ are twist hyperplanes that cross. In this case, $H_1 = H(\alpha_1, [\delta_1])$ and $H_2 = H(\alpha_2,[\delta_2])$ where $\alpha_1$ and $\alpha_2$ are distinct, disjoint non-separating meridians.
        
        \item $H_1$ is a twist hyperplane and $H_2$ is a switch hyperplane, and the two hyperplanes osculate. In this case, $H_2$ is parallel into $H_1$.
        
        \item $H_1$ is a twist hyperplane and $H_2$ is a switch hyperplane, and the two hyperplanes cross.
    \end{enumerate}
\end{lemma}
\begin{proof}
    Cases (1)-(5) above account for all possible combinations for the types of $H_1$ and $H_2$ and how they contact each other except for the case of two switch hyperplanes that cross. In fact, this case is not possible because every switch hyperplane is entirely contained in a distinct switch bridge. For the remainder of the proof, we show that each of the other combinations are possible, and provide characterizations of the hyperplanes in these cases.
    
    For case (1), suppose that $H_1$ and $H_2$ are two switch hyperplanes that osculate. We will show that they must be adjacent to a common twist flat and are not parallel, ie  we will see that $H_1 = H(X,X')$ and $H_2 = H(X,X'')$ with $X'\neq X''$. Here, when we say adjacent, we mean that $N(H_1)$ and $N(H_2)$ have non-empty intersection with a common twist flat. If $H_1$ and $H_2$ were not adjacent to a common twist flat, then $H_1^{\pm}$ and $H_2^{\pm}$ would all be contained in distinct twist flats, so $N(H_1) \cap N(H_2) = \emptyset.$ If $H_1$ and $H_2$ were adjacent to a common twist flat but were parallel, then there must be at least a copy of $\R^2\times \{-\frac{1}{2},\frac{1}{2}\}$ consisting entirely of twist cubes separating the switch bridges containing the two hyperplanes, so they could not osculate. Thus, the only way for two switch hyperplanes to osculate is if the are adjacent to a common twist flat and and are not parallel. See Figure \ref{figure:switch-switch} for an illustration of this case.
    
    \begin{figure}[htb]
        \centering
        \labellist
        \small\hair 2pt
         \pinlabel {$\M(X)$} [ ] at 203 570
         \pinlabel {$\M(X')$} [ ] at 1345 570
         \pinlabel {$\M(X'')$} [ ] at 935 -5
         \pinlabel {$H(X',X'')$} [ ] at 960 460
         \pinlabel {$H(X,X')$} [ ] at 960 600
        \endlabellist
        \begin{tikzpicture}
        \draw (0, 0) node[inner sep=0]
        {\includegraphics[scale=.23]{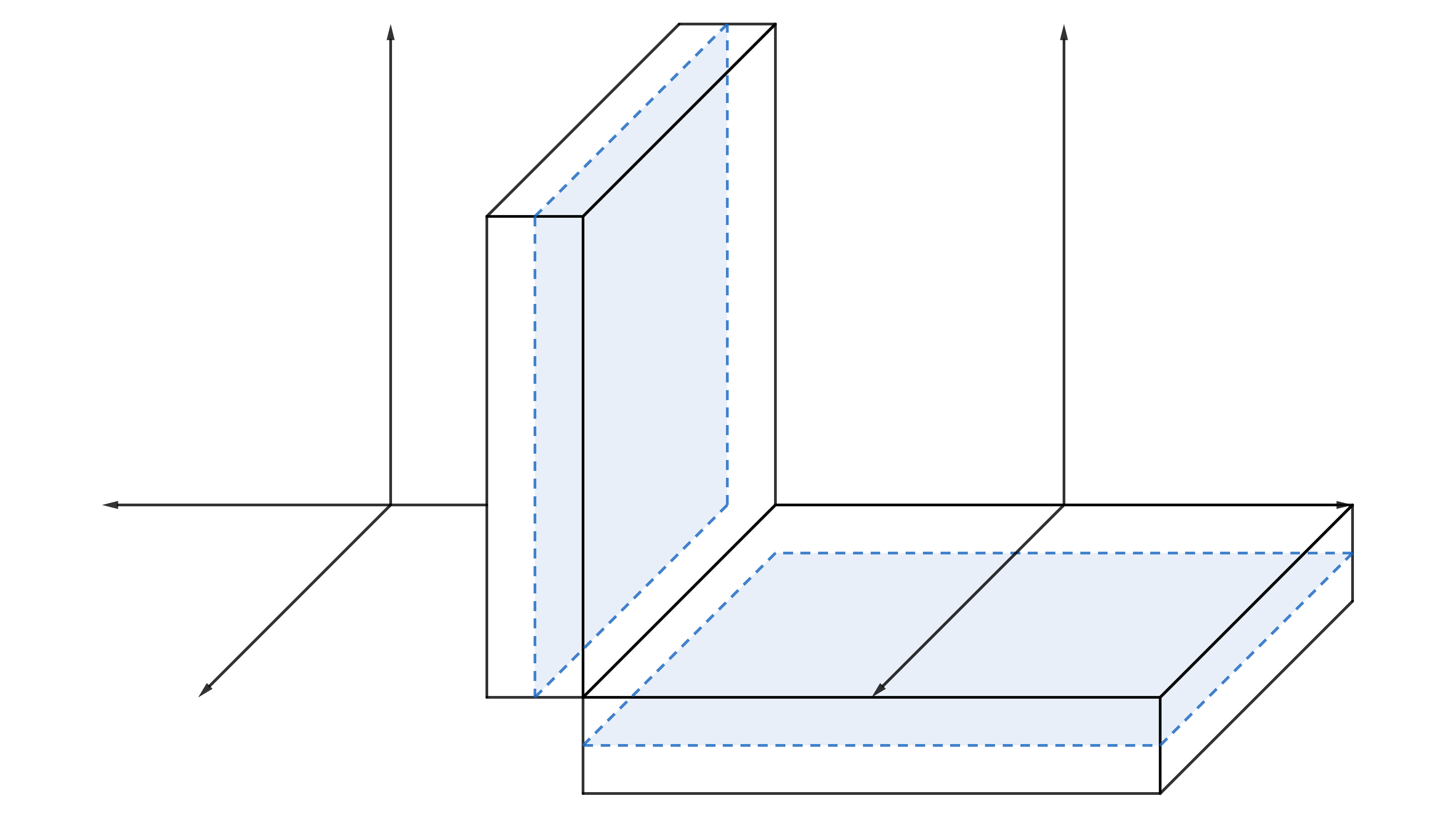}};
        \draw [->] (.9,1.55) to [out=150,in=30] (-0.6,1.55);
        \draw [->] (1.5,-.1) to (1.5,-1.5);
        \end{tikzpicture}
        \caption{An example of two switch hyperplanes that osculate, and are thus connected by an edge in $\contact{\M}$.}
        \label{figure:switch-switch}
    \end{figure}
    
    For case (2), suppose that $H_1$ and $H_2$ are twist hyperplanes that osculate. The only way for $H_1$ and $H_2$ to osculate is if there are non-separating meridians $\alpha$ and $\delta$ such that $H_1 = H(\alpha,[\delta])$ and $H_2= H(\alpha,[T_{\alpha}(\delta)]$, which are parallel hyperplanes by Lemma \ref{lem:twistcombparallel}. To see this, let $H_1 = H(\alpha,[\delta])$ and observe that in order for two twist hyerplanes to osculate, they must have non-empty intersection with a common twist flat, say $\M(X)$; otherwise, there would be a switch bridge separating the two hyperplanes and their carriers. Lemma \ref{lem:hypetype} (2) tells us that $H_1\cap \M(X)$ and $H_2\cap \M(X)$ are both isometric to $\R^2$ and are parallel to some coordinate planes in $\M(X)$; the only way for two such planes not to cross would be if they were parallel in $\M(X)$. Then because 
    \[(H_1^- \cup H_1^+) \cap M(X) = C(X,\delta') \cup C(X,T_{\alpha}(\delta'))\]
    for some $\delta'\in[\delta]$, and because $H_2 \cap \M(X)$ must be parallel to $H_1\cap \M(X)$, it follows that $N(H_1) \cap N(H_2)\cap \M(X)$ is equal to either $C(X,\delta')$ or $C(X,T_{\alpha}(\delta'))$; without loss of generality, suppose it is equal to $C(X,\delta')$. Then because hyperplanes in $\M$ are uniquely determined by a single $3$-cube, it must be the case that $H_2 = H(\alpha,[T^{-1}_{\alpha}(\delta)])$, which is parallel to $H_1$ by Lemma \ref{lem:twistcombparallel}. See Figure \ref{figure:twist-twist-osculate} for an illustration.
    
        \begin{figure}
        \centering
        \labellist \small\hair 2pt
         \pinlabel {$T_{\alpha_3}$} [ ] at 340 815
         \pinlabel {$T_{\alpha_3}$} [ ] at 1125 815
         \pinlabel {$T_{\delta_2}$} [ ] at 1460 308
         \pinlabel {$T_{\alpha_2}$} [ ] at 450 308
         \pinlabel {$T_{\alpha_1}$} [ ] at 145 100
         \pinlabel {$T_{\alpha_1}$} [ ] at 930 100
         \pinlabel {$H(\alpha_3,[T_{\alpha_3}(\delta_3)])$} [ ] at 1000 725
         \pinlabel {$H(\alpha_3,[\delta_3])$} [ ] at 980 243
        \endlabellist
        \begin{tikzpicture}
        \draw (0, 0) node[inner sep=0]
        {\includegraphics[scale=.22]{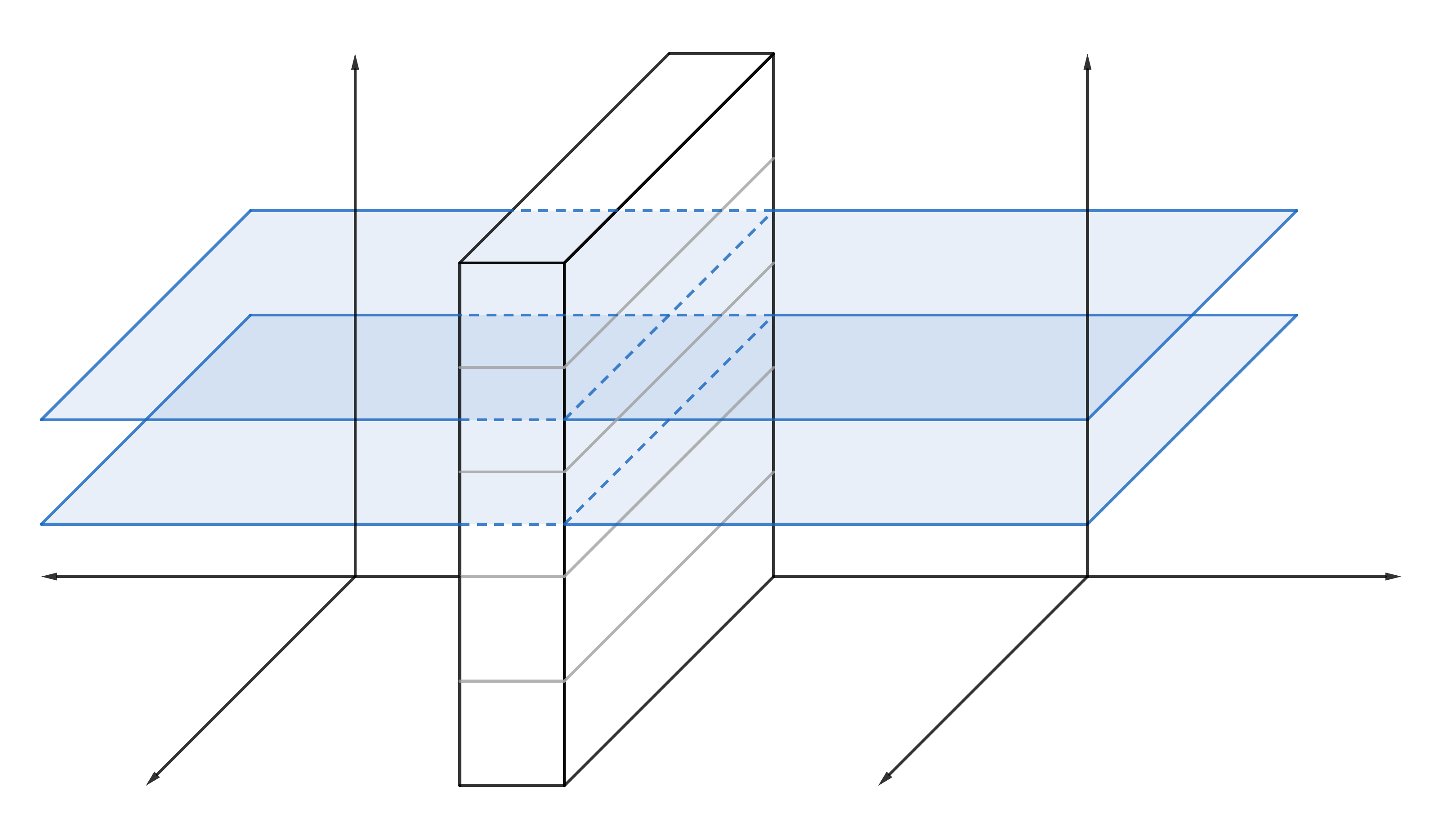}};
        \draw [->] (1.5,1.95) to  (1.5,1.4);
        \draw [->] (1.5,-1.45) to  (1.5,-.6);
        \end{tikzpicture}
        \caption{Two parallel twist hyperplanes $H(\alpha_3,[\delta_3])$ and $H(\alpha_3,[T_{\alpha_3}(\delta_3)])$ will osculate.}
        \label{figure:twist-twist-osculate}
    \end{figure}
    
    For case (3), suppose $H_1 = H(\alpha_1,[\delta_1])$ and $H_2 = H(\alpha_2,[\delta_2])$ are two twist hyperplanes that cross. If this is the case, then there is some pants decomposition $X$ such that $H_1 \cap H_2 \cap \M(X)$ is non-empty, implying that $\alpha_1,\alpha_2\in X$. Because $\alpha_1$ and $\alpha_2$ are in a pants decomposition together, they must be disjoint. Furthermore, Because $H_1$ and $H_2$ are not parallel, $\alpha_1$ and $\alpha_2$ must be distinct pants curves in $X$. See Figure \ref{figure:twist-twist-cross} for an illustration.
    
    \begin{figure}
        \centering
        \labellist
        \small\hair 2pt
         \pinlabel {$T_{\alpha_3}$} [ ] at 357 740
         \pinlabel {$T_{\alpha_3}$} [ ] at 1190 740
         \pinlabel {$T_{\delta_2}$} [ ] at 1415 265
         \pinlabel {$T_{\alpha_1}$} [ ] at 1030 60
         \pinlabel {$T_{\alpha_1}$} [ ] at 165 60
         \pinlabel {$T_{\alpha_2}$} [ ] at 470 265
         \pinlabel {$H(\alpha_3,[\delta_3])$} [ ] at 178 652
         \pinlabel {$H(\alpha_2,[\delta_2])$} [ ] at 1320 633
        \endlabellist
        \begin{tikzpicture}
        \draw (0, 0) node[inner sep=0]
        {\includegraphics[scale=.22]{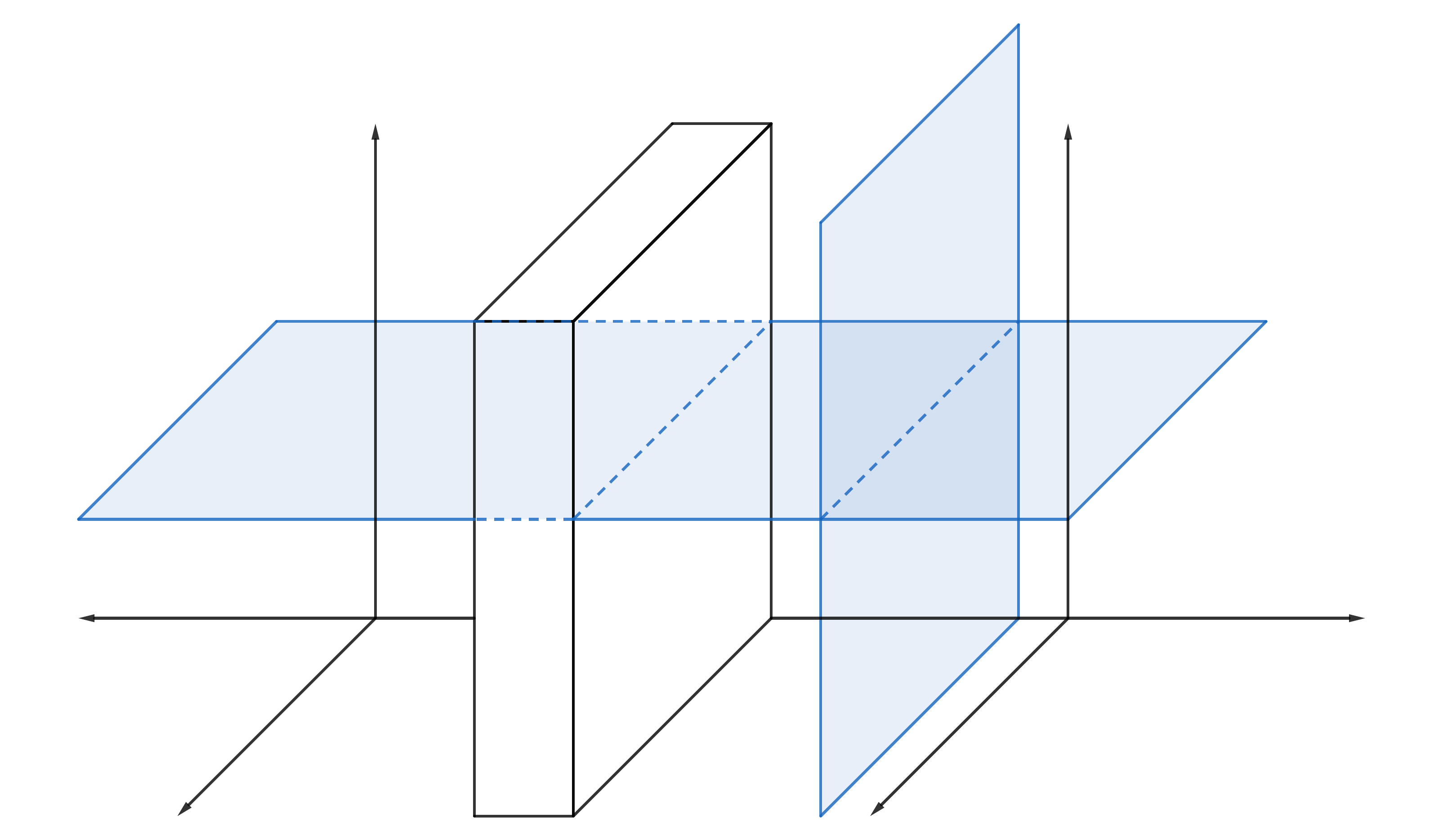}};
        \draw [->] (-4.5,1.4) to [out=-90, in=90] (-3.7,.4);
        \draw [->] (3.4,1.5) to  (2.2,1.5);
        \end{tikzpicture}
        \caption{Two twist hyperplanes $H(\alpha_2,[\delta_2])$ and $H(\alpha_3,[\delta_3])$ cross when $\alpha_2$ and $\alpha_3$ are disjoint, distinct non-separating meridians.}
        \label{figure:twist-twist-cross}
    \end{figure}
    
    For case (4), suppose that $H_1$ is a twist hyperplane and $H_2$ is a switch hyperplane, and that the two hyperplanes osculate. We will show that $H_2$ must be parallel into $H_1$. To see this, suppose $H_2 = H(X,X')$, where $X = \{\alpha_1, \alpha_2,\alpha_3\}$ and $X'= \{\delta_1, \alpha_2, \alpha_3\}$. Because $H_2 \subset \M(X,X')$, $H_1$ must have non-empty intersection with $\M(X)$ or $\M(X')$; without loss of generality, suppose $\M(X) \cap H_1 \neq \emptyset$. It then follows that $H_1 = H(\alpha_1,[\delta_1])$. From here it is clear that $H_2$ is parallel into $H_1$ because one of the combinatorial hyperplanes associated to $H_2$ is $C(X,\delta_1)$, which is contained in one of the combinatorial hyperplanes associated to $H_1$, (namely $C(\alpha_1,[\delta_1])$). See Figure \ref{figure:twist-switch-osculate} for an illustration.

    \begin{figure}
        \centering
        \labellist \small\hair 2pt
         \pinlabel {$T_{\alpha_1}$} [ ] at 90 210
         \pinlabel {$T_{\alpha_2}$} [ ] at 310 350
         \pinlabel {$T_{\alpha_3}$} [ ] at 235 690
         \pinlabel {$T_{\alpha_3}$} [ ] at 770 690
         \pinlabel {$T_{\delta_2}$} [ ] at 987 350
         \pinlabel {$T_{\alpha_1}$} [ ] at 635 210
         \pinlabel {$H(\delta_2,[\alpha_2])$} [ ] at 715 515
         \pinlabel {$H(X,X')$} [ ] at 350 690
         \pinlabel {$\M(X')$} [ ] at 900 515
         \pinlabel {$\M(X)$} [ ] at 150 515
        \endlabellist
        \begin{tikzpicture}
        \draw (0, 0) node[inner sep=0]
        {\includegraphics[scale=.3]{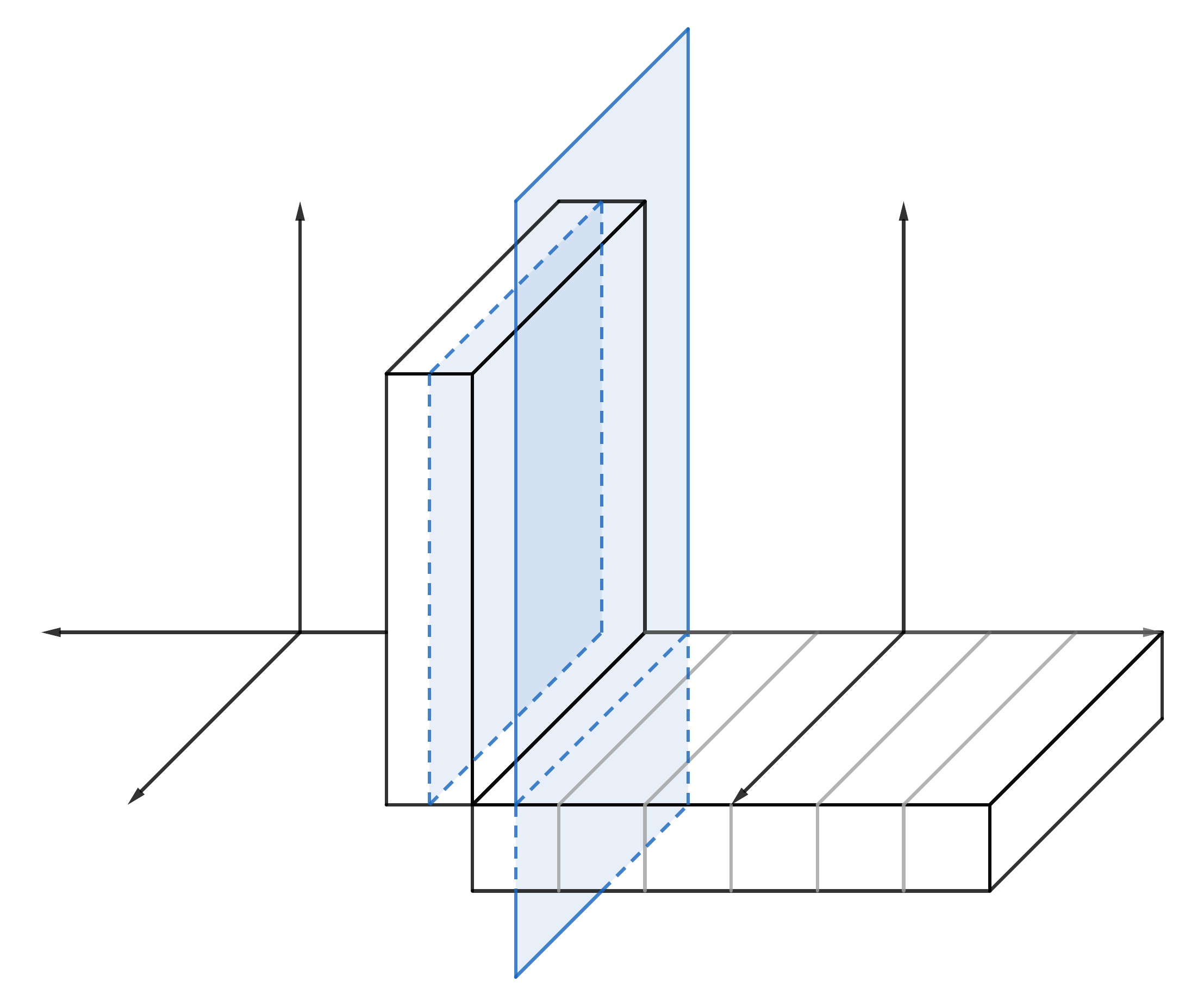}};
        \draw [->] (1.1,.8) to  (.6,.8);
        \draw [->] (-1.8,2.4) to [out=-90, in=90] (-1.2,1.4);
        \end{tikzpicture}
        \caption{The switch hyperplane $H(X,X')$ pictured here osculates with the twist hyperplane $H(\delta_2,[\alpha_2])$.}
        \label{figure:twist-switch-osculate}
    \end{figure}
    
    For case (5), let $(X,\Delta)\in\M(X)^{(0)}$ such that $(\alpha,\delta)$ is a part $(X,\Delta)$. The hyperplanes $H_1 = H(\alpha,[\delta])$ and $H_2 = H(X,\delta')$ where $\delta'$ is a dual to a curve in $X-\{\alpha\}$, will cross one another. See Figure \ref{figure:twist-switch-cross} for an illustration.
    
    \begin{figure}
        \centering
        \labellist
        \small\hair 2pt
        \pinlabel {$T_{\alpha_2}$} [ ] at 450 307
         \pinlabel {$T_{\alpha_1}$} [ ] at 150 107
         \pinlabel {$T_{\alpha_3}$} [ ] at 343 817
         \pinlabel {$\M(X)$} [ ] at 203 667
         \pinlabel {$\M(X')$} [ ] at 1345 667
         \pinlabel {$T_{\alpha_1}$} [ ] at 935 107
         \pinlabel {$T_{\delta_2}$} [ ] at 1456 307
         \pinlabel {$T_{\alpha_3}$} [ ] at 1125 817
         \pinlabel {$H(X,X')$} [ ] at 965 800
         \pinlabel {$H(\alpha_3,[\delta_3])$} [ ] at 965 662
        \endlabellist
        \begin{tikzpicture}
        \draw (0, 0) node[inner sep=0]
        {\includegraphics[scale=.23]{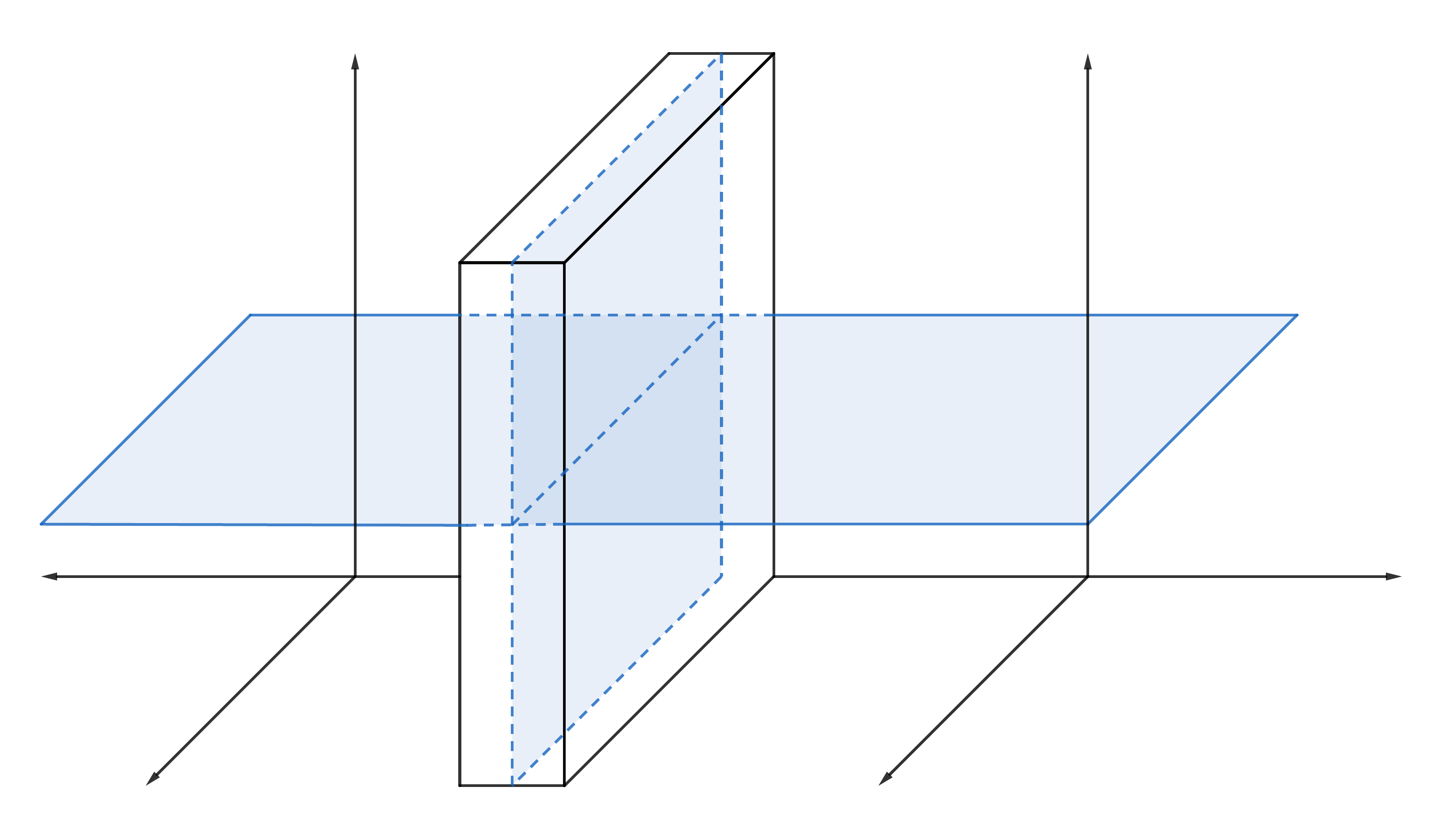}};
        \draw [->] (.75,2.55) to [out=-150,in=10] (-0.4,1.5);
        \draw [->] (1.55,1.4) to (1.55,.5);
        \end{tikzpicture}
        \caption{The twist hyperplane $H(\alpha_3,[\delta_3])$ crosses the switch hyperplane $H(X,X')$.}
        \label{figure:twist-switch-cross}
    \end{figure}
    \end{proof}
    
    The converse of case (3) in the previous lemma is also true.
    
    \begin{lemma}\label{lem:disjoint-implies-cross}
        Let $H_1 = H(\alpha_1,[\delta_1])$ and $H_2 = H(\alpha_2,[\delta_2])$ be two twist hyperplanes. If $\alpha_1$ and $\alpha_2$ are distinct, disjoint meridians, then $H_1$ and $H_2$ must cross.
    \end{lemma}
    \begin{proof}
        Suppose that $\alpha_1,\alpha_2$ are distinct, disjoint non-separating meridians. Then there is some pants decomposition $X$ containing both $\alpha_1$ and $\alpha_2$. To see this, notice that the complement $V_2-\{\alpha_1\cup\alpha_2\}$ is a genus $0$ handlebody with four spots, two spots corresponding to $\alpha_1$, and two spots corresponding to $\alpha_2$. Then any meridian $\eta$ on $V_2-\{\alpha_1\cup\alpha_2\}$ separating the two spots corresponding to $\alpha_1$ and separating the two spots corresponding to $\alpha_2$ will be a non-separating meridian on $V_2$ that is disjoint from both $\alpha_1$ and $\alpha_2$. Together, $X= \{\alpha,\beta,\eta\}$ is a non-separating pants decomposition on $V_2$.
        
        For both $i=1,2$, the hyperplane $H_i$ must have non-empty intersection with $\M(X)$ because $X\in\Pnm(\alpha_i)^{(0)}$ and $p(H_i) = \Pnm(\alpha_i)$, (by Lemma \ref{lem:twisthypcurves}). Additionally, because $\alpha_1$ and $\alpha_2$ are distinct, $H_1\cap \M(X)$ and $H_2\cap\M(X)$ must be planes that are parallel to distinct coordinate planes in $\M(X)$. Hence, $H_1\cap H_2 \cap \M(X)$ is non-empty and thus $H_1$ and $H_2$ must cross.
    \end{proof}

%%%%%%%%%%%%%%%%%%%%%%%%%%%%%%%%%%%%%%%%%%%%%%%%%%%%%%%%%%%%%%%%%%%%%%%%
\section{Constructing a Factor System with Unbounded Products}\label{sec:fsandup}

In this section, we show that $\M$ contains a factor system, and that consequently $\hh$ is an HHG. Furthermore, we show that the factor system has unbounded products.

%%%%%%%%%%%%%%%%%%%%%%%%%%%%%%%%%%%%%%%%%
\subsection{Characterizing the hyperclosure}\label{subsec:candidate}

Let $\fs$ be the hyperclosure of $\M$. In this section, we will use Lemma \ref{lem:Hagen-2.2} to determine exactly what subcomplexes are contained in $\fs$.

\begin{prop}\label{prop:contentsoffs}
$\fs$ is equal to the set of subcomplexes of $\M$ of the following types:
\begin{enumerate}
    \item the whole space $\M$;
    \item subcomplexes $F_1\cap F_2$ such that $F_1$ and $F_2$ are (not necessarily distinct) combinatorial hyperplanes of $\M$ with non-empty intersection, and
    \item $0$-cubes of $\M$.
\end{enumerate}
\end{prop}

Let $\fs'$ denote the set containing all subcomplexes of the types listed in Proposition \ref{prop:contentsoffs}.

\begin{lemma}\label{lem:characterizefs'}
If $F\in\fs'$, then $F$ falls into one of the following categories:
\begin{enumerate}
    \item $F=\M$;
    %\item $F=C(X,\delta)$ is a combinatorial switch hyperplane,
    \item $F$ is a combinatorial switch hyperplane;
    %\item $F=C(\alpha,[\delta])$ is a combinatorial twist hyperplane,
    \item $F$ is a combinatorial twist hyperplane;
    %\item $F=l(X,\delta,\delta')$ is the intersection of two combinatorial switch hyperplanes, 
    \item $F = l(X,\delta,\delta')$ is the non-empty intersection of two combinatorial switch hyperplanes $C(X,\delta)$ and $C(X,\delta')$, and $F$ is isometric to $\R$;
    %\item $F = t(\alpha,\alpha',[\delta],[\delta'])$ is the intersection of two combinatorial twist hyperplanes, or
    \item $F=t(\alpha,\alpha',[\delta],[\delta'])$ is the non-empty intersection of two combinatorial twist hyperplanes $C(\alpha,[\delta])$ and $C(\alpha',[\delta'])$, and $F$ is isometric to a tree that maps onto $\Pnm(\{\alpha,\alpha'\})$ via the map $p:\M\to\Pnm$; or
    \item $F$ is a $0$-cube of $\M$.
\end{enumerate}
Furthermore, $\fs'\subset \fs$.
\end{lemma}
\begin{proof}
It is clear by definition that $\fs'$ contains $\M$ and the $0$-cubes of $\M$. Additionally, $\fs'$ contains all combinatorial hyperplanes because any combinatorial hyperplane $C$ can be described as the intersection $C\cap C$. It remains to describe the intersections of two distinct combinatorial hyperplanes. To that end, suppose $F_1$ and $F_2$ are two distinct combinatorial hyperplanes that have non-empty intersection. Based on the types of the hyperplanes $F_1$ and $F_2$, we get three different cases.
\begin{enumerate}
    \item[(i)] Let $F_1=C(X,\delta)$ and $F_2=C(X,\delta')$ be two distinct combinatorial switch hyperplanes. Recall that every vertex in $C(X,\delta)$ must contain $\delta$ as a dual curve, and every vertex in $C(X,\delta')$ must contain $\delta'$ as a dual curve. Because $F_1$ and $F_2$ are both isometric to $\R^2$, (by Lemma \ref{lem:hypetype}), and are parallel to distinct coordinate planes in $\M(X)$, the intersection $F_1\cap F_2$ is then an isometrically embedded copy of $\R$ contained in $\M(X)^{(1)}$ such that every vertex contains both $\delta$ and $\delta'$ as dual curves. We will denote such a line by $l(X,\delta,\delta')$.
    
    \item[(ii)] Let $F_1=C(\alpha,[\delta])$ and $F_2=C(\alpha',[\delta'])$ be two distinct combinatorial twist hyperplanes. By the description of twist hyperplanes in Lemma \ref{lem:hypetype}, the intersection $F_1\cap F_2$ will be a tree $t(\alpha,\alpha',[\delta],[\delta'])$ such that every vertex contains both $\alpha$ and $\alpha'$ as pants curves, and such that the dual curves to $\alpha$ and $\alpha'$ will be in $[\delta]$ and $[\delta']$, respectively. Any non-empty intersection of $t(\alpha,\alpha',[\delta],[\delta'])$ with a twist flat will be a copy of $\R$, (in particular a line of the type described in (1)), and these copies of $\R$ will be connected across switch bridges to other copies of $\R$ via intervals $[-\frac{1}{2},\frac{1}{2}]$. In terms of the map $p:\M \to \Pnm$, $p|_{t(\alpha,\alpha',[\delta],[\delta'])}$ maps onto $\Pnm(\{\alpha,\alpha'\})$ and the fiber of a vertex $X\in \Pnm(\{\alpha,\alpha'\})^{(0)}$ is the line $l(X,\delta,\delta').$ In this way we can think of $t(\alpha,\alpha',[\delta],[\delta'])$ as a kind of ``blow-up" of the tree $\Pnm(\{\alpha,\alpha'\})$.
    
    \item[(iii)] Let $F_1= C(\alpha,[\delta])$ be a combinatorial twist hyperplane and let $F_2=C(X,\delta')$ be a combinatorial switch hyperplane. It is possible that $C(X,\delta')\subset C(\alpha,[\delta])$, in which case the intersection is just $C(X,\delta')$. Otherwise, the intersection of the two will be an isometrically embedded copy of $\R$ as in case (i) above, (since $C(X,\delta')\subset \M(X)$ and the intersection of $C(\alpha,[\delta])$ with $\M(X)$ is a combinatorial switch hyperplane).
\end{enumerate}

For the containment $\fs'\subset \fs$, notice that subcomplexes of types (1)-(3) are in $\fs$ because they satisfy properties (1) and (2) of Definition \ref{def:hyperclosure}, (so they are contained in every set $\mathfrak{G}$ as in Definition \ref{def:hyperclosure}). Additionally, by Lemma \ref{lem:HHS-2.6}, subcomplexes of types (4) and (5) are projections of combinatorial hyperplanes, so they are contained in $\fs$ via properties (2) and (3) of Definition \ref{def:hyperclosure}. Lastly, any $0$-cube $(X,\Delta)$, where $\Delta =\{\delta_1,\delta_2,\delta_3\}$, can be viewed as the intersection of the three combinatorial hyperplanes $\{C(X,\delta_i)\}_{i=1}^3$. Hence, by Lemma \ref{lem:HHS-2.6},
\[ (X,\Delta) = C(X,\delta_1) \cap C(X,\delta_2) \cap C(X,\delta_3) = \gate{C(X,\delta_1)}{\gate{C(X,\delta_2)}{C(X,\delta_3)}} \in \fs.\]
Therefore, $\fs'\subset \fs.$
\end{proof}

To continue towards proving Proposition \ref{prop:contentsoffs}, we prove the following lemma, which allows us to factor projections of subcomplexes of $\M$ through projections to combinatorial switch hyperplanes. Note that this is a version of Lemma 2.1 from \cite{Hagen_2020} that is specific to our context.

\begin{lemma}\label{lem:factorprojections} Suppose $\{C_i\}_{i=1}^{n}$ is a collection of combinatorial switch hyperplanes separating two subcomplexes $F_1,F_2\subset \M$. Suppose the $C_i$ are ordered by distance to $F_1$, with $C_n$ being closest to $F_1$ and $C_{1}$ being furthest from $F_1$. Then $\gate{F_1}{F_2}$ is parallel to $\gate{F_1}{\gate{C_n}{\cdots \gate{C_{1}}{F_2}\cdots}}$.
\end{lemma}
\begin{proof}
We prove this inductively, and use Lemma \ref{lem:Hagen-1.5}, which says that the hyperplanes crossing $\gate{F_1}{F_2}$ are exactly those crossing both $F_1$ and $F_2$.

Suppose for the base case that there is a combinatorial switch hyperplane $C$ separating $F_1$ and $F_2$. If $H$ is a hyperplane crossing both $F_1$ and $F_2$, (and therefore also $\gate{F_1}{F_2}$ by Lemma \ref{lem:Hagen-1.5}), then $H$ must cross $C$ as well; this follows from the fact that hyperplanes separate $\cat$ cube complexes into disjoint half-spaces. By Lemma \ref{lem:Hagen-1.5}, it follows that $H$ crosses $\gate{C}{F_2}$, and one more application of this lemma to $F_1$ and $\gate{C}{F_2}$ implies $H$ crosses $\gate{F_1}{{\gate{C}{F_2}}}$. So if a hyperplane $H$ crosses $\gate{F_1}{F_2}$, it must also cross $\gate{F_1}{{\gate{C}{F_2}}}$.

Now suppose instead that $H$ is a hyperplane that crosses $\gate{F_1}{{\gate{C}{F_2}}}$. By applying Lemma \ref{lem:Hagen-1.5} several times, we get that $H$ must cross $F_1$, $C$, and $F_2$. Since $H$ crosses both $F_1$ and $F_2$, applying Lemma \ref{lem:Hagen-1.5} one last time tells us that $H$ crosses $\gate{F_1}{F_2}$. We have now shown that any hyperplane $H$ crosses $\gate{F_1}{F_2}$ if and only if $H$ crosses $\gate{F_1}{{\gate{C}{F_2}}}$, and hence the two subcomplexes are parallel.

For the inductive step, assume that when $F_1$ and $F_2$ are separated by $n-1$ ordered combinatorial switch hyperplanes $\{C_i\}_{i=1}^{n-1}$, then $\gate{F_1}{F_2}$ is parallel to $\gate{F_1}{\gate{C_{n-1}}{\cdots \gate{C_{1}}{F_2}\cdots}}$. Now suppose $F_1$ and $F_2$ are separated by at least $n$ hyperplanes, and that $\{C_i\}_{i=1}^{n}$ is a collection of combinatorial switch hyperplanes separating the two subcomplexes, ordered so that $C_n$ is closest to $F_1$ and $C_{1}$ is closest to $F_2$. Then $\gate{C_{1}}{F_2}$ is a subcomplex that is separated from $F_1$ by $n-1$ hyperplanes. Applying the inductive hypothesis tells us that $\gate{F_1}{\gate{C_n}{\cdots \gate{C_{1}}{F_2}\cdots}}$ is parallel to $\gate{F_1}{\gate{C_{1}}{F_2}}$. By the same reasoning as the base case, $\gate{F_1}{\gate{C_{1}}{F_2}}$ is parallel to $\gate{F_1}{F_2}$. Since parallelism is an equivalence relation, it follows that $\gate{F_1}{F_2}$ is parallel to $\gate{F_1}{\gate{C_n}{\cdots \gate{C_{1}}{F_2}\cdots}}$.
\end{proof}

We also show that $\fs'$ is closed under parallelism.

\begin{lemma}\label{lem:closedparallel}
If $F\in\fs'$ and $F'\subset \M$ is a convex subcomplex that is parallel to $F$, then $F'\in\fs'$.
\end{lemma}
\begin{proof}
    We will prove this lemma in four cases.

    \textbf{Case 1:} $\M$ is contained in its own parallelism class since it is the only convex subcomplex of $\M$ that intersects every hyperplane. Additionally, all $0$-cubes are parallel to one another since they are the only convex subcomplexes that do not intersect any hyperplane. Hence, if $F=\M$ or $F$ is a $0$-cube, then $F'\in\fs'$.

    \textbf{Case 2:} Suppose $F$ is a combinatorial hyperplane. By Lemmas \ref{lem:twistcombparallel} and \ref{lem:switchcombparallel}, $F'$ must also be a combinatorial hyperplane. Hence, $F'\in\fs'$.
    
    \textbf{Case 3:} Suppose $F = l(X,\delta_1,\delta_2)$, where $X = \{\alpha_1,\alpha_2,\alpha_3\}$ and $\delta_1,\delta_2$ are duals to $\alpha_1,\alpha_2$, respectively. Because $F$ is contained entirely within a single twist flat, it does not intersect any switch hyperplanes. This means that $F'$ also does not intersect any switch hyperplanes and must therefore be contained in a single twist flat.
    The only hyperplanes that $F$ does intersect are twist hyperplanes of the form $H(\alpha_3,[T_{\alpha_3}^{n_3}(\delta_3)])$, where $\delta_3$ is a dual to $\alpha_3$ in $\M(X)$ and $n_3\in\Z$. Additionally, any twist hyperplane that crosses a hyperplane $H(\alpha_3,[T_{\alpha_3}^{n_3}(\delta_3)])$ does not cross $F$. In particular, any twist hyperplanes $H(\alpha_1',[\delta_1'])$ and $H(\alpha_2',[\delta_2'])$ such that $\{\alpha_1',\alpha_2',\alpha_3\}\in\Pnm(\alpha_3)^{(0)}$ do not cross $F$. Lemma \ref{lem:HHS-2.4} tells us that there is a cubical isometry from $F'$ to $F$. Combining the above facts implies that $F'$ is a line $l(X',\delta_1',\delta_2')$, where $X'\in\Pnm(\alpha_3)^{(0)}$ and $\delta_1',\delta_2'$ are duals to the curves in $X'-\{\alpha_3\}$. Hence, $F'\in \fs'$.
    
    \textbf{Case 4:} Suppose $F=t(\alpha_1,\alpha_2,[\delta_1],[\delta_2])$, and let $X = \{\alpha_1,\alpha_2,\alpha_3\}$, (so $F$ has non-empty intersection with $\M(X)$). First, we know that $F$ intersects switch hyperplanes that correspond to edges in $p(H(\alpha_1,[\delta_1])\cap H(\alpha_2,[\delta_2])) = \Pnm(\{\alpha_1,\alpha_2\})$, meaning that $F'$ must also intersect each of these switch hyperplanes. In particular, this means $p(F') = \Pnm(\{\alpha_1,\alpha_2\}).$ Additionally, $F$ does not intersect the twist hyperplanes $H(\alpha_1,[T_{\alpha_1}^{n_1}(\delta_1)])$ and $H(\alpha_2,[T_{\alpha_2}^{n_2}(\delta_2)])$ where $n_1,n_2\in\Z$. Lemma \ref{lem:HHS-2.4} tells us that there is a cubical isometry from $F'$ to $F$. Combining the above facts implies that $F'$ must be a tree $t(\alpha_1,\alpha_2,[T_{\alpha_1}^{n_1}(\delta_1)],[T_{\alpha_2}^{n_2}(\delta_2)]).$ Thus, $F'\in\fs'$.
\end{proof}

We are now equipped to prove Proposition \ref{prop:contentsoffs}.

\begin{proof}[Proof of Proposition \ref{prop:contentsoffs}]
    First, recall via Lemma \ref{lem:characterizefs'} that $\fs'\subset \fs.$ Thus, it remains to show that $\fs\subset \fs'$. Throughout this proof, we will use the notation $\mathfrak{C}(\mathcal{M})$ to denote the set of combinatorial hyperplanes of $\mathcal{M}.$
    
    To show that $\fs \subset \fs'$, we will explicitly determine the convex subcomplexes contained in $\fs$, and will see that they are indeed elements of $\fs'$. To this end, we will use the characterization of the hyperclosure given by Lemma \ref{lem:Hagen-2.2}, ie $\fs = \cup_{n\geq 0}\fs_{n}$, where $\fs_0=\{\M\}$ and $\fs_n$ for $n\geq 1$ is the set of convex subcomplexes of $\M$ that can be written in the form $\gate{C}{F}$ for some $C\in \mathfrak{C}(\mathcal{M})$ and $F\in\fs_{n-1}$.
    
    We will start by showing that projections $\gate{C}{F}$, where $C$ and $F$ are not separated by any combinatorial switch hyperplanes, are all contained in $\fs'$. Then we will show that for any $n$, any projection of a subcomplex $F\in \fs_n$ onto a combinatorial hyperplane can be decomposed into projections between subcomplexes that are not separated by any combinatorial switch hyperplanes.
    
    \textbf{Case 1:}
    Let $\fs'_0=\fs_0$, and let $\fs'_n$ be the set of convex subcomplexes that can be written as $\gate{C}{F_{n-1}}$ such that $C\in\mathfrak{C}(\mathcal{M})$, $F_{n-1}\in\fs'_{n-1}$, and $C$ and $F_{n-1}$ are not separated by a switch hyperplane. We will show first that $\cup_{n\geq 0}\fs'_{n}\subset\fs'$. Specifically, we will show that the sets $\fs_k'$ stabilize after $k=3$.
    
    For $n=1$, we have $\fs_1 = \fs'_1 = \mathfrak{C}(\mathcal{M})$ since no combinatorial hyperplane is separated from $\M$ by any other combinatorial hyperplane.
    
    Suppose $F_2\in\fs'_2$, meaning that $F_2 = \gate{C}{F_1}$ for some $C\in \mathfrak{C}(\mathcal{M})$ and some $F_1\in\fs'_1 = \mathfrak{C}(\mathcal{M})$, where $F_1$ and $C$ are not separated by any combinatorial switch hyperplanes. Because $F_1$ and $C$ are not separated by a switch hyperplane, both $F_1$ and $C$ have non-empty intersection with some twist flat $\M(X)$.
    Futher, because $F_1$ and $C$ are combinatorial hyperplanes that intersect a common twist flat, one of the following must hold:
    \begin{enumerate}
        \item[(i)] $F_1$ and $C$ are parallel;
        \item[(ii)] one of $F_1,C$ is parallel into the other; or
        \item[(iii)] $F_1$ and $C$ have non-empty intersection.
    \end{enumerate}
    These are in fact the only cases because if $F_1$ and $C$ are disjoint, they must be parallel or one must be parallel into the other. This follows from a similar argument to cases (2) and (4) in Lemma \ref{lem:edgescontact}; to summarize, $F_1\cap \M(X)$ and $C\cap \M(X)$ are isometric to $\R^2$ and are parallel in $\M(X)$, and by the fact that hyperplane in $\M$ are determined by the intersection with a single $3$-cube in $\M$, $F_1$ and $C$ must be parallel, or one must be parallel into the other. Thus, $F_2$ is determined by one of the following cases.
    \begin{enumerate}
        \item[(a)] If $F_1$ and $C$ are parallel, then by Lemmas \ref{lem:HHS-2.4}, \ref{lem:twistcombparallel}, and \ref{lem:switchcombparallel}, $F_2 = \gate{C}{F_1}$ is a combinatorial hyperplane.
        
        \item[(b)] If $F_1$ is parallel into $C$, then $\gate{C}{F_1}$ is the subcomplex of $C$ that is parallel to $F_1$. Again by Lemmas \ref{lem:HHS-2.4}, \ref{lem:twistcombparallel}, and \ref{lem:switchcombparallel}, $F_2 = \gate{C}{F_1}$ is a combinatorial hyperplane.
        
        \item[(c)] If $C$ is parallel into $F_1$, we know that $\gate{F_1}{C}$ is the subcomplex of $F_1$ that is parallel to $C$. Then Lemma \ref{lem:HHS-2.6} implies $\gate{C}{F_1}$ is also parallel to $\gate{F_1}{C}$. Again by  Lemmas \ref{lem:HHS-2.4}, \ref{lem:twistcombparallel}, and \ref{lem:switchcombparallel}, $F_2 = \gate{C}{F_1}$ is a combinatorial hyperplane.
        
        \item[(d)] If $F_1$ and $C$ have non-empty intersection, then Lemma \ref{lem:HHS-2.6} tells us that $\gate{C}{F_1} = F_1\cap C$, which is the intersection of two combinatorial hyperplanes. Specifically, by Lemma \ref{lem:characterizefs'}, $F_2 = \gate{C}{F_1}$ is either a combinatorial switch hyperplane, a line $l(X,\delta_1,\delta_2)$, or a tree $t(\alpha_1,\alpha_2,[\delta_1],[\delta_2])$.
    \end{enumerate}
    In each of these cases, $F_2$ is a convex subcomplex found in $\fs'$. Thus, $\fs'_2 \subset \fs'.$
    
    Now suppose $F_3\in\fs'_3$, meaning that $F_3 = \gate{C}{F_2}$ for some $C\in \mathfrak{C}(\mathcal{M})$ and some $F_2\in\fs'_2$ where $F_2$ and $C$ are not separated by any combinatorial switch hyperplanes. Because $F_2\in\fs'_2,$ we know that $F_2$ is either a combinatorial hyperplane, a tree $t(\alpha_1,\alpha_2,[\delta_1],[\delta_2])$, or a line $l(X,\delta_1,\delta_2)$. This means $F_3$ is determined by one of the following cases.
    \begin{enumerate}
        \item[(a)] If $F_2$ is a combinatorial hyperplane, then $F_2\in\fs'_1$ and $F_3=\gate{C}{F_2}\in\fs'_2$. We already know that $F_3$ will be a combinatorial hyperplane, a tree, or a line.
        
        \item[(b)] Suppose $F_2=l(X,\delta_1,\delta_2)$. Since we have assumed that $C$ and $F_2$ are not separated by a switch hyperplane, they must both intersect the twist flat $\M(X)$. This means that either $F_2$ is parallel into $C$, or $F_2$ intersects $C$ in a single point. In the first case, $F_3 = \gate{C}{F_2}$ will be the line contained in $C$ that is parallel to $F_2$, and in the second case $F_3 = \gate{C}{F_2}$ is the single $0$-cube of intersection.
        
        \item[(c)] Suppose $F_2=t(\alpha_1,\alpha_2,[\delta_1],[\delta_2])$. Then the classification of $\gate{C}{F_2}$ depends on whether $C$ is a combinatorial switch or twist hyperplane.
        \begin{enumerate}
            \item[(i)] Suppose $C$ is a combinatorial switch hyperplane, and say it is contained in the twist flat $\M(X)$, where $\alpha_1,\alpha_2\in X$. Then $\gate{C}{F_2}\subset C\subset \M(X)$. The intersection of $F_2$ with $\M(X)$ is the line $l(X,\delta_1,\delta_2)$. If $C$ and $F_2$ have non-empty intersection, then $\gate{C}{F_2} =C\cap F_2$ will be a single $0$-cube or the line $l(X,\delta_1,\delta_2)$. If instead $C$ and $F_2$ are disjoint, then $\gate{C}{F_2}$ will be the line contained in $C$ that is parallel to $F_2$. This can be seen via Lemma \ref{lem:Hagen-1.5} and the fact that the only hyperplanes intersecting both $C$ and $F_2$  will be the collection of twist hyperplanes $H(\alpha_3,[\delta_3])$, where $\alpha_3\in X - \{\alpha_1,\alpha_2\}$.
            
            \item[(ii)] Suppose $C$ is a combinatorial twist hyperplane. To determine the projection $\gate{C}{F_2}$ we must understand how the two combinatorial twist hyperplanes $C(\alpha_1,[\delta_1])$ and $C(\alpha_2,[\delta_2])$ relate to $C$. Because $F_2$ and $C$ intersect a common twist flat, it must be that $C$ intersects one or both of $C(\alpha_1,[\delta_1])$ and $C(\alpha_2,[\delta_2])$. If $C$ intersects both $C(\alpha_1,[\delta_1])$ and $C(\alpha_2,[\delta_2])$, then either their intersection is a single $0$-cube, (meaning $\gate{C}{F_2}$ is that $0$-cube), or the intersection is the entire tree, meaning $\gate{C}{F_2}$ is the entire tree. If $C$ intersects only one of $C(\alpha_1,[\delta_1])$ and $C(\alpha_2,[\delta_2])$, say to $C(\alpha_1,[\delta_1])$, then $C$ must be parallel to $C(\alpha_1,[\delta_1])$, and $F_2$ must be parallel into $C$. This means that $\gate{C}{F_2}$ will be the tree contained in $C$ that is parallel to $F_2$.
        \end{enumerate}
    \end{enumerate}
    Thus, the convex subcomplexes contained in $\fs'_3$ are either combinatorial hyperplanes, trees, lines, or single $0$-cubes, meaning $\fs'_3\subset \fs'$. Additionally, we see that $\fs'_3 = \fs'_2 \cup \{0-\text{cubes}\}$. By this description, we can see that projections of elements of $\fs'_3$ to combinatorial hyperplanes are still elements of $\fs'_3$. This means that for $n\geq 3$, $\fs'_n = \fs'_3\subset \fs'$. Thus, $\cup_{n\geq 0} \fs'_n \subset \fs'.$
    
    \textbf{Case 2:} 
    For the general case, we will show that if $F\in\fs_n$, then $F\in\cup_{n\geq 0} \fs'_n\subset \fs'.$ We proceed by induction.
    
    Recall that $\fs_0 = \fs_0'$ and $\fs_1 = \fs_1'$. Clearly, for $n=0$ or $1$, $\fs_n \subset \fs'$.
    
    For the inductive step, suppose that for $k\geq 2$, if $F_{k-1}\in\fs_{k-1}$, then $F_{k-1}\in\cup_{n\geq 0} \fs'_n.$ Now suppose that $F_k\in\fs_k.$ This means $F_k=\gate{C}{F_{k-1}}$ for some combinatorial hyperplane $C$ and some $F_{k-1}\in\fs_{k-1}$. Suppose that $\{C_i\}_{i=1}^m$ are \emph{all} of the combinatorial switch hyperplanes separating $C$ from $F_{k-1}$, where the $C_i$ are ordered such that $C_1$ is closest to $F_{k-1}$ and $C_m$ is furthest from $F_{k-1}$. Consider the subcomplex $\gate{C}{\gate{C_m}{\cdots \gate{C_1}{F_{k-1}}\cdots}}$. By assumption, the subcomplexes $F_{k-1}$ and $C_1$ are not separated by any combinatorial switch hyperplanes. The inductive hypothesis then tells us that $\gate{C_1}{F_{k-1}}\in\cup_{n\geq 0} \fs'_n\subset \fs'$. Furthermore, because $C_i$ and $C_{i+1}$ are not separated by any combinatorial switch hyperplanes, $\gate{C}{\gate{C_m}{\cdots \gate{C_1}{F_{k-1}}\cdots}}\in \fs'_{k+m}.$ By Lemma \ref{lem:factorprojections}, we know that $\gate{C}{\gate{C_m}{\cdots \gate{C_1}{F_{k-1}}\cdots}}$ is parallel to $F_k = \gate{C}{F_{k-1}}$, so Lemma \ref{lem:closedparallel} implies $F_k\in \cup_{n\geq 0} \fs'_n\subset \fs'$. This proves that $\fs \subset \fs'.$
    
    Thus, we have shown that $\fs = \fs'.$
\end{proof}

\begin{cor}\label{cor:closureishyperclosure}
    Let $\mathfrak{M}$ be the closure of $\M$. Then $\mathfrak{M} = \mathfrak{F}.$
\end{cor}
\begin{proof}
    Clearly $\fs'\subset \mathfrak{M} \subset \fs$. Then Proposition \ref{prop:contentsoffs} implies $\fs'=\mathfrak{M}=\fs$.
\end{proof}

%%%%%%%%%%%%%%%%%%%%%%%%%%%%%%%%%%%%%%%%%
\subsection{\texorpdfstring{$\hh$}{H2} is an HHG with unbounded products}\label{subsec:proofoffsandup}

With the classification of the subcomplexes of $\fs$ in hand, we can now prove that $\hh$ is an HHG with unbounded products. We start by using our characterization of $\fs$ given by Proposition \ref{prop:contentsoffs} to prove that $\fs$ is a factor system.

\begin{lemma}\label{lem:factorsys}
$\fs$ is a factor system for $\M$.
\end{lemma}
\begin{proof}
    By the definition of the hyperclosure, all that must be proven is that property (2) of Definition \ref{def:factorsys} is satisfied, ie that $\fs$ has finite multiplicity. We will use the classification $\fs=\fs'$ afforded by Proposition \ref{prop:contentsoffs} to show that the finite multiplicity property holds with $N=14$.
    
    Let $(X,\Delta)$ be a vertex in $\M$, where $X = \{\alpha_1, \alpha_2, \alpha_3\}$ and $\Delta = \{\delta_1, \delta_2, \delta_3\}$. Of course $(X,\Delta)\in \M$ and $(X,\Delta)\in(X,\Delta)$. Additionally, $(X,\Delta)$ is in three combinatorial switch hyperplanes corresponding to each of the three dual curves: $C(X,\delta_1)$, $C(X,\delta_2)$, and $C(X,\delta_3)$. Similarly, $(X,\Delta)$ will be contained in three combinatorial twist hyperplanes corresponding to the three pants curves and their duals: $C(\alpha_1,[\delta_1])$, $C(\alpha_2,[\delta_2])$, and $C(\alpha_3,[\delta_3])$. The vertex is also contained in several lines and trees. Particularly, $(X,\Delta)$ is contained in the lines $l(X, \delta_1,\delta_2)$, $l(X,\delta_1,\delta_3)$, and $l(X,\delta_2,\delta_3)$, as well as the trees $t(\alpha_1,\alpha_2,[\delta_1],[\delta_2])$, $t(\alpha_1,\alpha_3,[\delta_1],[\delta_3])$, and $t(\alpha_2,\alpha_3,[\delta_2],[\delta_3])$. These are all the types of subcomplexes in $\fs$ that contain $(X,\Delta)$, for a total of $14$ subcomplexes. Thus, property (2) is satisfied, and we have shown that $\fs$ is a factor system for $\M$.
\end{proof}

The existence of a factor system for $\M$ leads to the following corollary.

\begin{cor}\label{cor:HHG}
$(\hh,\is)$ is a hierarchically hyperbolic group, where $\is$ is a subset of $\fs$ containing a single element from each parallelism class in $\fs$, (excluding single $0$-cubes), and our associated set of $\delta$-hyperbolic spaces is the set of factored contact graphs $\{ \fcontact{F} : F\in \is\}$.
\end{cor}
\begin{proof}
Via \cite[Remark 13.2]{behrstock2017} and Lemma \ref{lem:factorsys}, we can conclude that $\M$ is hierarchically hyperbolic, with domains $\is$ and $\delta$-hyperbolic spaces as described in the corollary. We exclude $0$-cubes so that nesting and orthogonality are mutually exclusive.

In \cite[Proposition 6.7]{hamHenDehn}, Hamenst\"adt and Hensel prove that $\hh$ acts properly, cocompactly, and by isometries on $\M$. Since $\M$ is an HHS, it follows that $\hh$ is an HHG with the same domains and associated $\delta$-hyperbolic spaces.
\end{proof}

Lastly, in order to be able to apply Theorem \ref{thm:abbott} to $(\hh,\is)$, we prove that $(\hh,\is)$ has unbounded products.

\begin{lemma}\label{lem:unboundedproducts}
$(\hh,\is)$ has unbounded products.
\end{lemma}
\begin{proof}
Every subcomplex $F\in \is-\{\M\}$ has infinite diameter, so we must show that the corresponding factors $E_F$, as defined in Lemma \ref{lem:HHS-2.4}, also have infinite diameter.

If $F$ is a combinatorial hyperplane, then by Corollaries \ref{cor:twistsimplytrans} and \ref{cor:switchinfcard}, there are infinitely many elements in the parallelism class of $F$. Each unique subcomplex in $[F]$ intersects $E_F$ in a unique $0$-cube, (by Lemma \ref{lem:HHS-2.4}), and hence, $E_F$ must have infinite diameter.

Suppose instead that $F\in\is-\{\M\}$ is arbitrary. By the characterization of $\fs\subset\is$ given in Proposition \ref{prop:contentsoffs}, we know that $F$ is contained in some combinatorial hyperplane $C$. Consider the cubical isometric embedding $C\times E_C \to \mathcal{X}$ given by Lemma \ref{lem:HHS-2.4}. If $F'\in[F]$ such that there is some $C'\in[C]$ with $F'\subset C'$, then there is a $0$-cube $e\in E_C$ such that $F\times \{e\}\to\mathcal{X}$ factors as $F\times \{e\}\xrightarrow{id}F'\hookrightarrow F.$ Thus, $E_C\subset E_F$. Since $\operatorname{diam}(E_C) =\infty,$ it follows that  $\operatorname{diam}(E_F) = \infty.$

Thus, $(\M,\is)$ has unbounded products.
\end{proof}

%%%%%%%%%%%%%%%%%%%%%%%%%%%%%%%%%%%%%%%%%%%%%%%%%%%%%%%%%%%%%%%%%%%%%%%%
\section{The Factored Contact Graph is Quasi-isometric to the disk graph}\label{sec:contactdisk}

The last piece necessary to prove the main theorems is to prove that the disk graph $\mathcal{D}(V_2)$ is coarsely $\hh$-equivariantly quasi-isometric to the factored contact graph $\fcontact{\M}$. In this section we prove this claim, and then finally prove the main theorems.

Let $\mathcal{ND}(V_g)$ be the \emph{non-separating disk graph}, ie the induced subgraph of $\mathcal{D}(V_g)$ whose vertices correspond to the non-separating meridians on $\bd V_g$. Since the vertices in the model $\M$ include only non-separating meridians, it will be easier to work with $\mathcal{ND}(V_2)$ rather than $\mathcal{D}(V_2)$ when constructing a quasi-isometry to $\fcontact{\M}$. The following proposition allows us to make this simplification.

\begin{prop}\label{prop:disk-nonsepdisk}
For $g\geq 2$, the non-separating disk graph $\mathcal{ND}(V_g)$ isometrically embeds as a $\frac{3}{2}$-dense subgraph of the disk graph $\mathcal{D}(V_g)$.
\end{prop}

The following argument is analogous to the argument that the non-separating curve graph is quasi-isometric to the curve graph.

\begin{proof}[Proof of Proposition \ref{prop:disk-nonsepdisk}]
Because $\mathcal{ND}(V_g)$ is a subgraph of $\mathcal{D}(V_g)$, the inclusion is $1$-Lipschitz.

Suppose now that $\gamma:[0,n]\to \mathcal{D}(V_g)$ is a geodesic in the the disk graph such that $\gamma(0)$ and $\gamma(n)$ are non-separating meridians. Suppose for some $0<i<n$ that $\gamma(i)$ is a separating meridian. The complement $V_g - \gamma(i)$ consists of two spotted handlebodies $Y_1$ and $Y_2$, each with genus at least $1$. The meridians $\gamma(i-1)$ and $\gamma(i+1)$ must intersect since they are distance two apart, but both must be disjoint from $\gamma(i)$. This means that $\gamma(i-1) \cup \gamma(i+1)$ is contained in say $Y_1$. Since $Y_2$ is a spotted handlebody of genus at least one, it must contain at least one non-separating meridian $\delta$. The meridian $\delta$ is disjoint from $\gamma(i-1) \cup \gamma(i+1)$, so we can replace $\gamma(i)$ with $\delta$. In this way, we can replace each separating meridian in $\gamma$ with a non-separating meridian, and thus the distance between $\gamma(0)$ and $\gamma(n)$ in the non-separating disk graph is at most the distance between them in the disk graph. This gives us the lower bound.

Lastly, by the above we see that any separating meridian is always disjoint from at least one non-separating meridian, so every vertex of $\mathcal{D}(V_g)$ is distance $1$ from a vertex in $\mathcal{ND}(V_g)$. Then any point on any edge in $\mathcal{D}(V_g)$ is at most distance $\frac{3}{2}$ from some point in $\mathcal{ND}(V_g)$. Thus, $\mathcal{ND}(V_g)$ is a $\frac{3}{2}$-dense subgraph of $\mathcal{D}(V_g)$.
\end{proof}

\begin{prop}\label{prop:nonsepdisk-contact}
The non-separating disk graph $\mathcal{ND}(V_2)$ isometrically embeds as a $\frac{3}{2}$-dense subgraph of the contact graph $\contact{\M}$.
\end{prop}
\begin{proof}
We define first a map $\iota:\mathcal{ND}(V_2)^{(0)}\to \contact{\M}^{(0)}$ as $\iota(\alpha) = H(\alpha,[\delta])$, where $H(\alpha,[\delta])$ is any twist hyperplane such that every vertex in $N(H(\alpha,[\delta]))$ contains $\alpha$ as a pants curve. This map is injective because given any two non-separating meridians $\alpha$ and $\beta$, if $H(\alpha,[\delta]) = H(\beta,[\delta'])$, then indeed $\alpha = \beta$.

Next we assume $\alpha$ and $\beta$ are two distinct, non-separating meridians connected by an edge in $\mathcal{ND}(V_2)$, and show that $\iota(\alpha)$ and $\iota(\beta)$ will also be connected by an edge. We know that $\iota(\alpha)=H(\alpha,[\delta])$ and $\iota(\beta) = H(\beta,[\delta'])$ for some $\delta$ and $\delta'$ that are duals to $\alpha$ and $\alpha'$. By Lemma \ref{lem:disjoint-implies-cross}, $\iota(\alpha)$ and $\iota(\alpha')$ must cross, so they are connected by an edge in $\contact{\M}$.

Because $\iota$ is injective and sends disjoint meridians to hyperplanes that cross one another, (ie sends edges to edges), it follows that the map $\iota$ extends to a simplicial embedding $\iota: \mathcal{ND}(V_2)\to \contact{\M}$, which is thus $1$-Lipschitz.

Suppose now that $\gamma:[0,n]\to \contact{\M}$ is a geodesic parametrized by arc length such that $\gamma(0)$ and $\gamma(n)$ are in the image of $\iota$. We will show that we can use $\gamma$ to produce a new geodesic consisting entirely of twist hyperplanes in the image of $\iota$. The first step is to show that starting from one end of $\gamma$, we can replace any switch hyperplane in $\gamma$ with a twist hyperplane. Then we must show that we can choose the twist hyperplanes to be in the image of $\iota$. If $n=0$ or $1$, then we are already done, so assume $n\geq 2$. 

Fix $i$ such that $0< i < n$. Suppose $\gamma(i)$ corresponds to a switch hyperplane and $\gamma(i-1)$ corresponds to a twist hyperplane. Then either
\begin{enumerate}
    \item $\gamma(i+1)$ is a twist hyperplane, or
    \item $\gamma(i+1)$ is a switch hyperplane.
\end{enumerate}
In either case, both $N(\gamma(i-1))$ and $N(\gamma(i+1))$ must have non-empty intersection with $N(\gamma(i))$, but must be disjoint from one another. Recall also that each edge in $\contact{\M}$ corresponds to the two hyperplanes either crossing or osculating.

For case (1), suppose $\gamma(i+1)$ is a twist hyperplane. It cannot be the case that one of $\gamma(i-1)$ or $\gamma(i+1)$ osculates with $\gamma(i)$ and the other crosses it because then $\gamma(i-1)$ and $\gamma(i+1)$ would intersect one another (see Figure \ref{figure:case1-caveat}). More specifically, if $\gamma(i-1)$ osculates with $\gamma(i)$ and $\gamma(i+1)$ crosses $\gamma(i)$, then by Lemma \ref{lem:edgescontact} (4), $\gamma(i)$ will be parallel into $\gamma(i-1)$, and by the definition of parallel into, $\gamma(i+1)$ must cross $\gamma(i-1)$.

\begin{figure}[htb]
\centering
\labellist \small\hair 2pt
 \pinlabel {$\gamma(i)$} [ ] at 500 720
 \pinlabel {$\gamma(i-1)$} [ ] at 835 720
 \pinlabel {$\gamma(i+1)$} [ ] at 835 604
\endlabellist
\begin{tikzpicture}
\draw (0, 0) node[inner sep=0]
{\includegraphics[scale=.24]{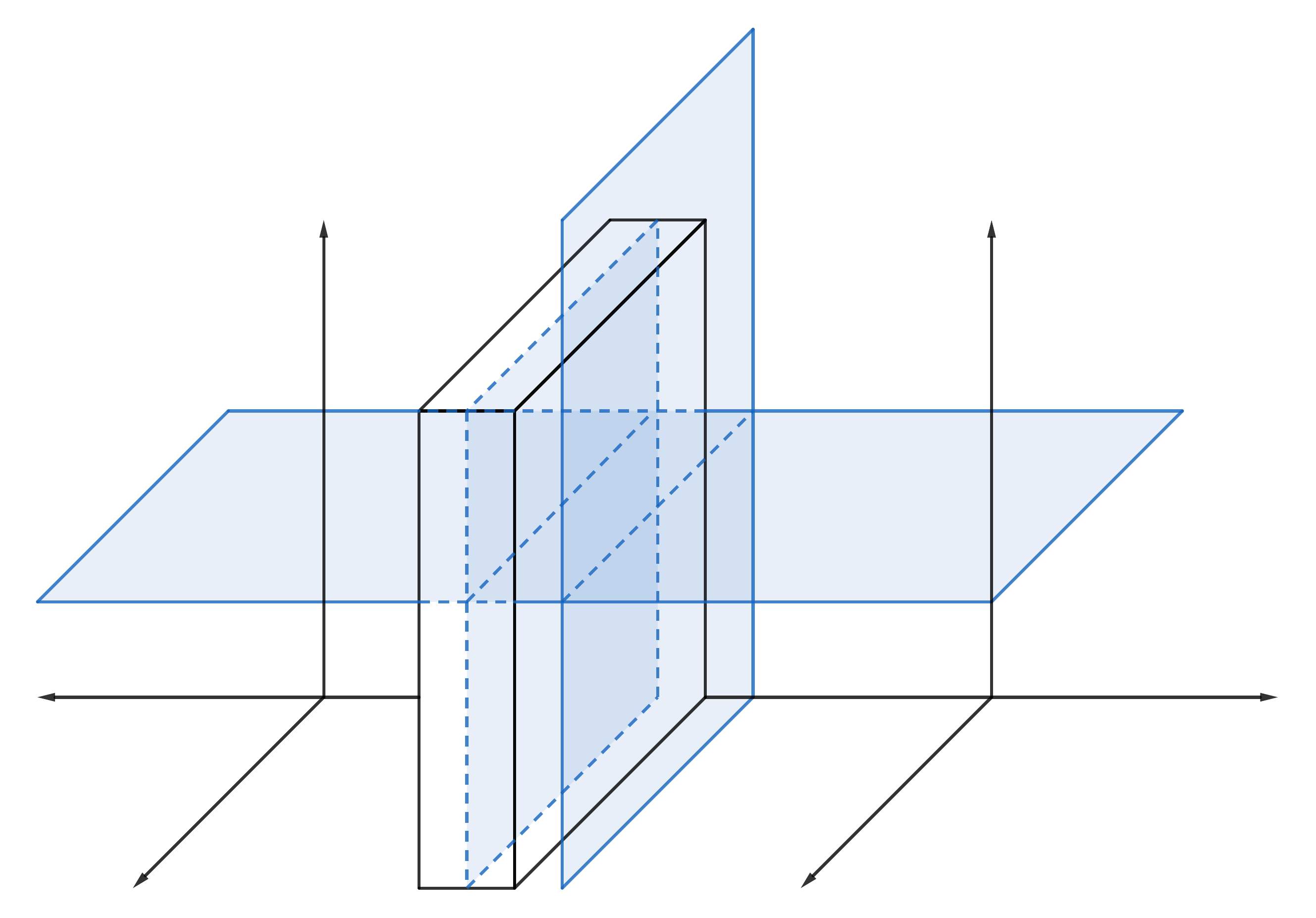}};
\draw [->] (-1.2,2.1) to  (-1.2,.5);
\draw [->] (1.1,2.5) to [out=150, in=30] (.3,2.5);
\draw [->] (1.6,1.2) to  (1.6,.1);
\end{tikzpicture}
\caption{Ruling out a subcase of case (1).}
\label{figure:case1-caveat}
\end{figure}

This means we have only two subcases:
\begin{enumerate}
    \item[(1a)] both $\gamma(i-1)$ and $\gamma(i+1)$ osculate with $\gamma(i)$, or
    
    \item[(1b)] both $\gamma(i-1)$ and $\gamma(i+1)$ cross $\gamma(i)$. Notice that since $N(\gamma(i+1))$ and $N(\gamma(i-1))$ must be disjoint, $\gamma(i-1)$ and $\gamma(i+1)$ must actually be parallel to one another.
\end{enumerate}
These two subcases are illustrated in Figure \ref{figure:case1ab}.
In case (1a), we can replace $\gamma(i)$ with any twist hyperplane that crosses $\gamma(i)$, as this hyperplane must also cross $\gamma(i-1)$ and $\gamma(i+1)$. This follows from the fact that $\gamma(i)$ must be parallel into both $\gamma(i-1)$ and $\gamma(i+1)$, (by Lemma \ref{lem:edgescontact} (4)). In case (1b), we can replace $\gamma(i)$ with any twist hyperplane that osculates with $\gamma(i)$, as this hyperplane must cross $\gamma(i-1)$ and $\gamma(i+1)$. Again, this follows from the fact that $\gamma(i)$ is parallel into any twist hyperplane with which it osculates (Lemma \ref{lem:edgescontact} (4)).

\begin{figure}[htb]
\centering
\begin{subfigure}{\textwidth}
    \centering
    \labellist \small\hair 2pt
     \pinlabel {$\gamma(i)$} [ ] at 750 380
     \pinlabel {$\gamma(i+1)$} [ ] at 750 530
     \pinlabel {$\gamma(i-1)$} [ ] at 150 455
    \endlabellist
    
    \begin{tikzpicture}[scale=.80]
    \draw (0, 0) node[inner sep=0]
    {\includegraphics[scale=.232]{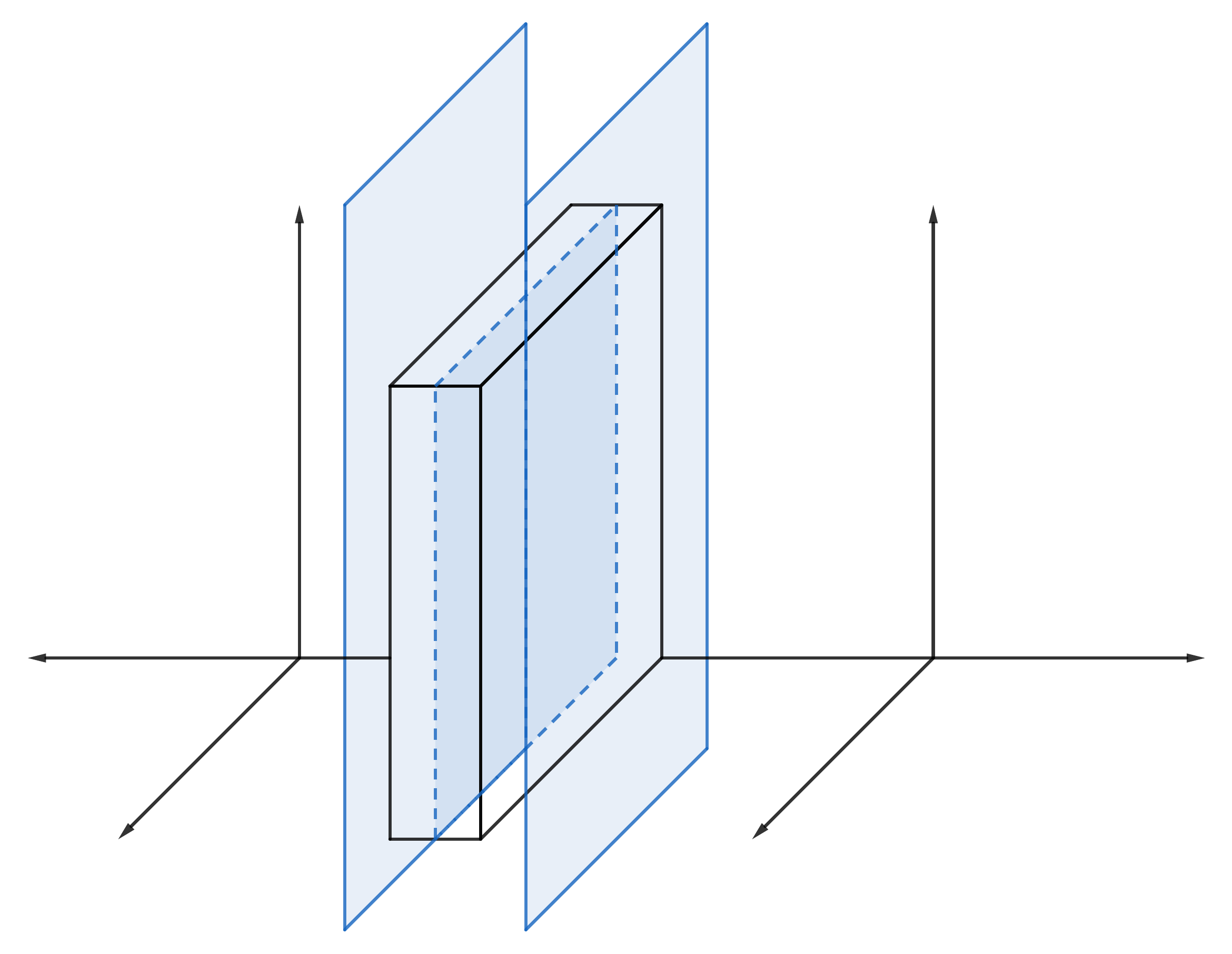}};
    \draw [->] (-3.5,.15) to (-2.3,.15);
    \draw [->] (1.1,1.1) to [out=150, in=30] (.25,1.1);
    \draw [->] (1.4,-.7) to [out=-150, in=-30] (-1.2,-.7);
    \end{tikzpicture}
\end{subfigure}
\begin{subfigure}{\textwidth}
\centering
\labellist
\small\hair 2pt
 \pinlabel {$\gamma(i)$} [ ] at 1000 780
 \pinlabel {$\gamma(i-1)$} [ ] at 1355 780
 \pinlabel {$\gamma(i+1)$} [ ] at 1000 200
\endlabellist
\begin{tikzpicture}
\draw (0, 0) node[inner sep=0]
{\includegraphics[scale=.18]{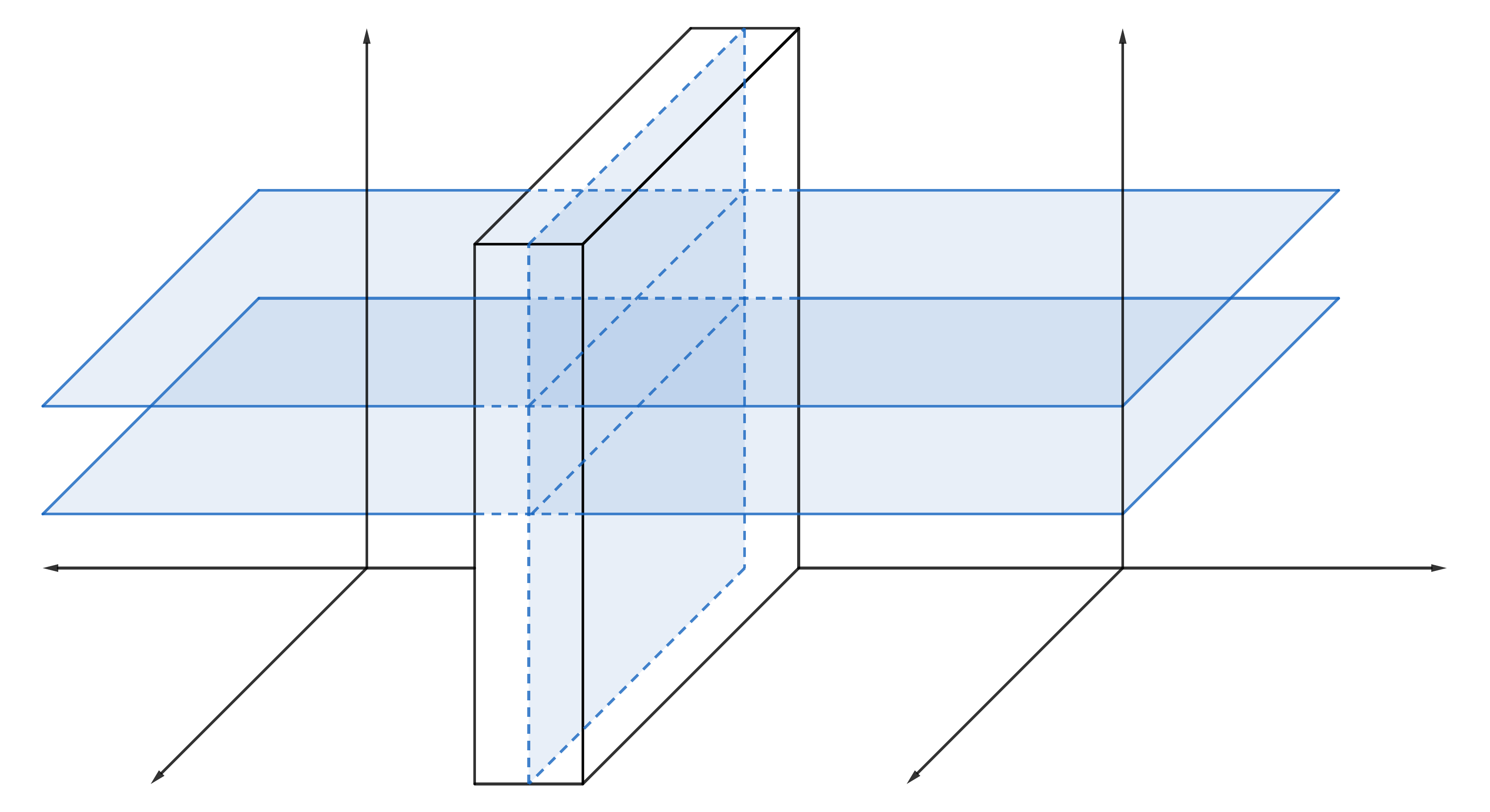}};
\draw [->] (3.2,1.9) to  (3.2,1.1);
\draw [->] (.7,1.9) to  (-0.2,1.6);
\draw [->] (1.1,-1.3) to (1.1,-.4);
\end{tikzpicture}
\end{subfigure}

\caption{Cases (1a) and (1b), respectively.}
\label{figure:case1ab}
\end{figure}

In case (2), where $\gamma(i+1)$ is a switch hyperplane, $\gamma(i)$ and $\gamma(i+1)$ must osculate because no two switch hyperplanes can cross, (again by Lemma \ref{lem:edgescontact}). This presents us with two subcases:
\begin{enumerate}
    \item[(2a)] $\gamma(i-1)$ osculates with $\gamma(i)$, or
    
    \item[(2b)] $\gamma(i-1)$ crosses $\gamma(i)$.
\end{enumerate}
Figure \ref{figure:case2ab} illustrates these two cases.
 In case (2a), we can replace $\gamma(i)$ with a twist hyperplane that crosses both $\gamma(i)$ and $\gamma(i+1)$, as this hyperplane must also cross $\gamma(i-1)$. Again we are using the fact that a switch hyperplane is parallel into any twist hyperplane with which it osculates (Lemma \ref{lem:edgescontact} (4)). In case (2b), we can also replace $\gamma(i)$ with a twist hyperplane that crosses both $\gamma(i)$ and $\gamma(i+1)$ as such a hyperplane must also cross $\gamma(i-1)$. This is because $\gamma(i+1)$ must be parallel into $\gamma(i-1)$.
 
\begin{figure}[htb]
\centering
\begin{subfigure}{\textwidth}
    \centering
    \labellist
    \small\hair 2pt
     \pinlabel {$\gamma(i-1)$} [ ] at 700 785
     \pinlabel {$\gamma(i)$} [ ] at 720 590
     \pinlabel {$\gamma(i+1)$} [ ] at 720 395
    \endlabellist
    \begin{tikzpicture}
    \draw (0, 0) node[inner sep=0]
    {\includegraphics[scale=.23]{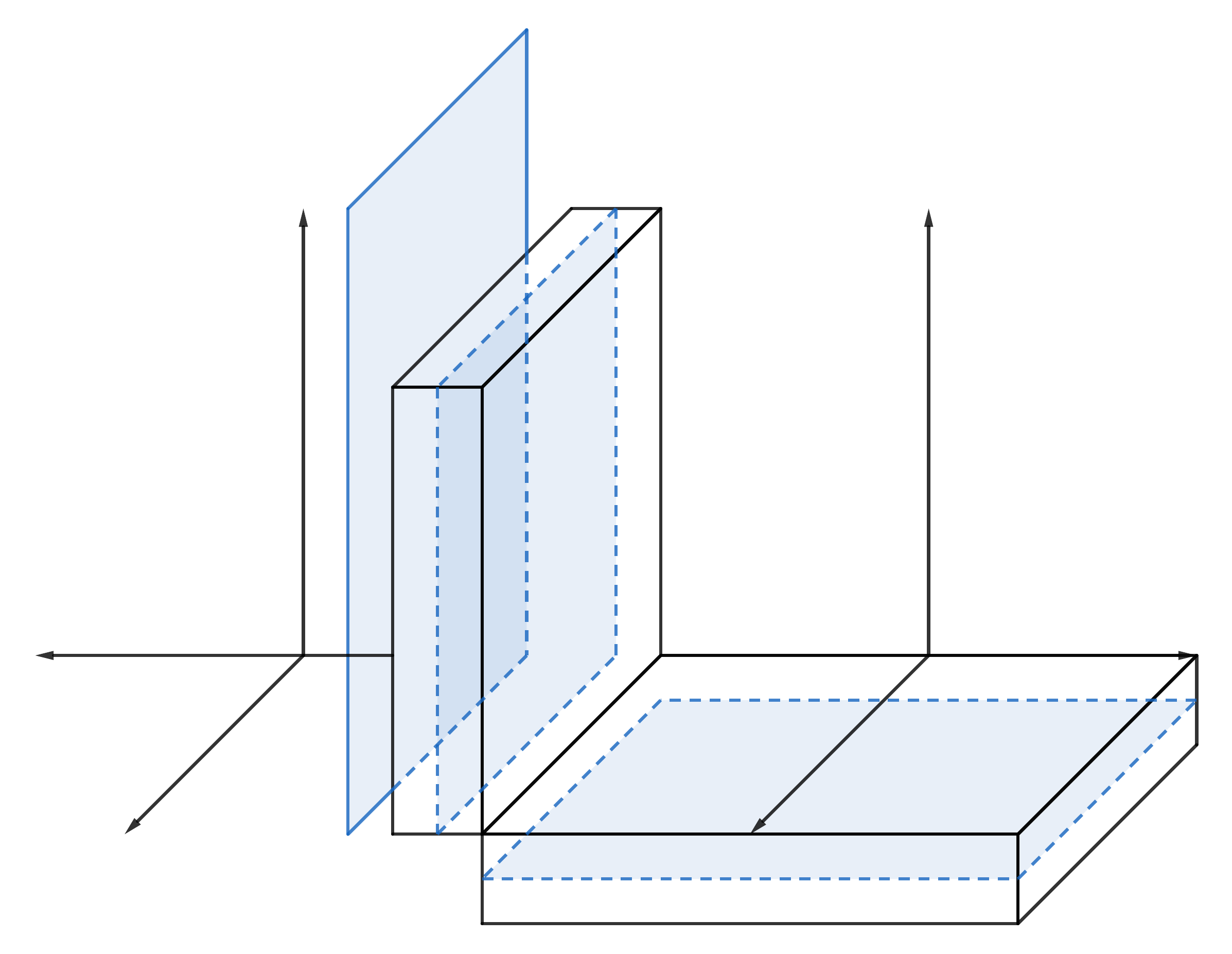}};
    \draw [->] (1.2,-.6) to  (1.2,-1.8);
    \draw [->] (.9,.9) to [out=-150, in=-30] (-.2,.9);
    \draw [->] (.3,2.63) to (-.8,2.63);
    \end{tikzpicture}
\end{subfigure}
\begin{subfigure}{\textwidth}
    \centering
    \labellist
    \small\hair 2pt
     \pinlabel {$\gamma(i-1)$} [ ] at 179 675
     \pinlabel {$\gamma(i)$} [ ] at 893 715
     \pinlabel {$\gamma(i+1)$} [ ] at 1213 415
    \endlabellist
    \begin{tikzpicture}
    \draw (0, 0) node[inner sep=0]
    {\includegraphics[scale=.18]{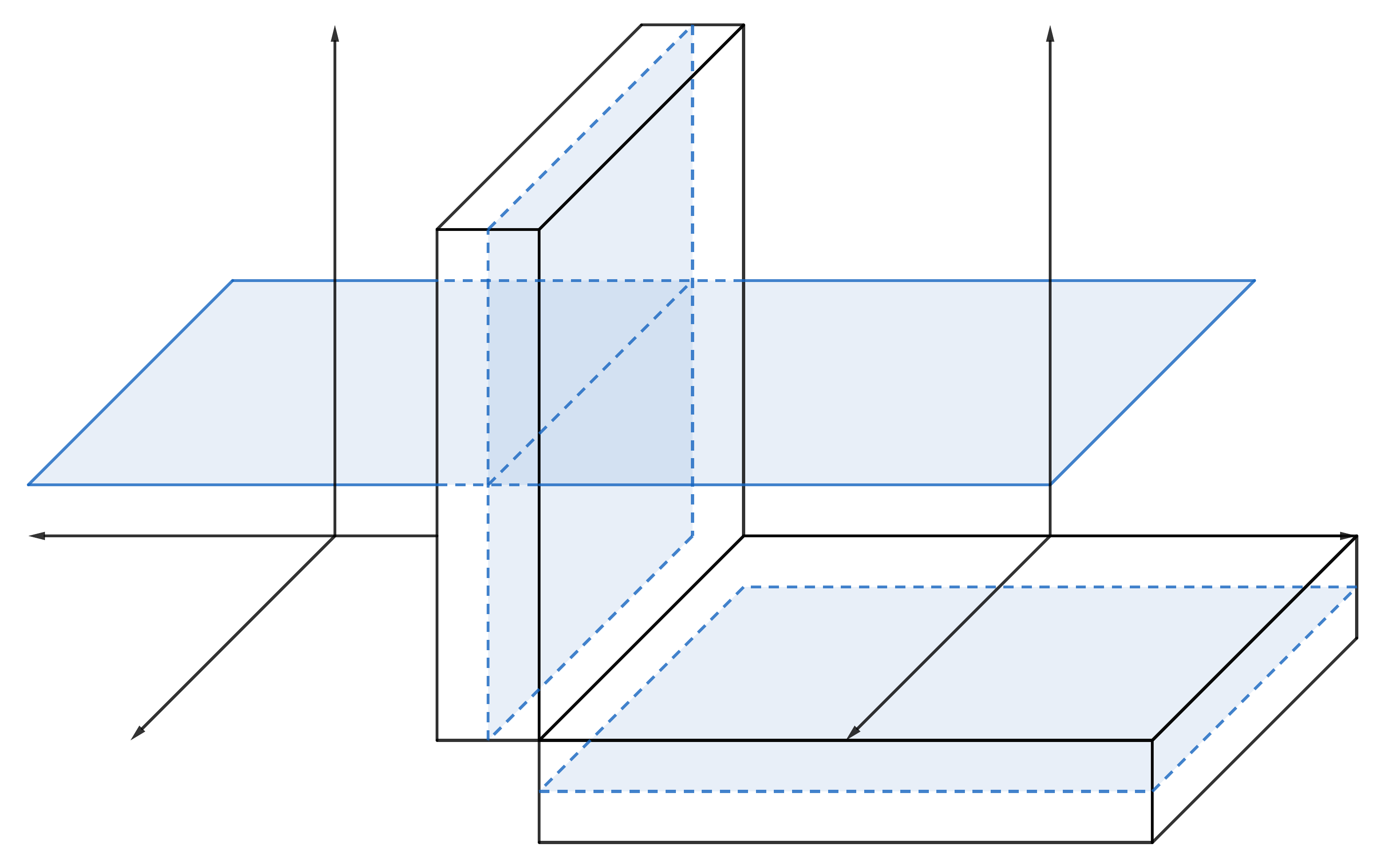}};
    \draw [->] (.9,1.4) to [out=-150, in=-30] (-.2,1.4);
    \draw [->] (-3,1.2) to (-2.5,.5);
    \draw [->] (3.1,-.4) to (3.1,-1.1);
    \end{tikzpicture}
\end{subfigure}

\caption{Cases (2a) and (2b), respectively.}
\label{figure:case2ab}
\end{figure}

By starting with one end of $\gamma$, say $\gamma(0)$, we can use the above arguments to replace any vertices in $\gamma$ corresponding to switch hyperplanes with vertices corresponding to twist hyperplanes. Let $\gamma'$ be the altered geodesic. Next we show that we can $\gamma'$ so that every vertex corresponds to a twist hyperplane in the image of $\iota$.

Suppose $\gamma'(i)=H(\alpha,[\delta])$ is a twist hyperplane that is not in the image of $\iota$. If $\gamma'(i-1)$ and $\gamma'(i+1)$ both cross $\gamma'(i)$, then because $\iota(\alpha)$ is parallel to $\gamma(i)$, (by Lemma \ref{lem:twistcombparallel}), it must also cross $\gamma'(i-1)$ and $\gamma'(i+1)$; we can replace $\gamma'(i)$ with $\iota(\alpha)$. If instead $\gamma'(i-1)$ and $\gamma'(i+1)$ both osculate with $\gamma'(i)$, then by Lemma \ref{lem:edgescontact} (2), all three hyperplanes are parallel, and there must be some twist hyperplane $H(\beta,[\eta])$ that crosses all three of them; we can replace $\gamma'(i)$ with $\iota(\beta)$. Lastly, it cannot be the case that $\gamma'(i-1)$ osculates with $\gamma'(i)$ and that $\gamma'(i+1)$ crosses $\gamma'(i)$ because $\gamma'(i+1)$ would then also cross $\gamma'(i-1)$.

We can thus replace $\gamma'$ with a geodesic consisting entirely of twist hyperplanes in the image of $\iota$. This means that the length of a geodesic in $\mathcal{ND}(V_2)$ connecting $\gamma(0)$ and $\gamma(n)$ is at most as long as $\gamma'$, (and so at most as long as $\gamma$), and hence we attain our lower bound.

It remains to show that $\iota$ is $\frac{3}{2}$-dense. Let $H(\alpha,[\delta])$ be a twist hyperplane not in the image of $\iota$. Let $H(\alpha',[\delta'])$ be any twist hyperplane that crosses $H(\alpha,[\delta])$. By Lemma \ref{lem:twistcombparallel}, $\iota(\alpha')$ is parallel to $H(\alpha',[\delta'])$, so because $H(\alpha,[\delta])$ crosses $H(\alpha',[\delta'])$, it must also cross $\iota(\alpha')$. Thus, $H(\alpha,[\delta])$ is a distance $1$ away from $\iota(\alpha')$, and hence a distance $1$ from the image of $\iota$.

Consider now a switch hyperplane $H(X,X')$. This hyperplane crosses all twist hyperplanes that have non-empty intersection with $\M(X)$ and $\M(X')$, and at least one of these twist hyperplanes will be in the image of $\iota$. Hence, $H(X,X')$ will also be a distance $1$ from the image of $\iota$. Thus, any vertex in $\contact{\M}$ will be a distance $1$ from some point in $\iota(\mathcal{ND}(V_2))$, and any point on an edge in $\contact{\M}$ will be at most distance $\frac{3}{2}$ from a point in $\iota(\mathcal{ND}(V_2))$. Thus, $\mathcal{ND}(V_2)$ is a $\frac{3}{2}$-dense subgraph of $\contact{\M}$.
\end{proof}

Propositions \ref{prop:disk-nonsepdisk} and \ref{prop:nonsepdisk-contact} gives us the following.

\begin{cor}\label{cor:contact-disk}
The factored contact graph $\fcontact{\M}$ is coarsely $\hh$-equivariantly quasi-isometric to the disk graph $\mathcal{D}(V_2)$.
\end{cor}
\begin{proof}
Recall from Corollary \ref{cor:closureishyperclosure} that our factor system for $\M$ is exactly the minimal factor system as described in \cite{behrstock2017}. \cite[Remark 8.18]{behrstock2017} tells us that for minimal factor systems, the factored contact graph $\fcontact{\M}$ is quasi-isometric to the contact graph $\contact{\M}$ via the inclusion. Proposition \ref{prop:disk-nonsepdisk} tells us that the inclusion $\mathcal{ND}(V_2) \hookrightarrow \mathcal{D}(V_2)$ is a quasi-isometry. By constructing a quasi-inverse $r$ for this inclusion, and composing this $r$ with $\iota$ from Proposition \ref{prop:nonsepdisk-contact} and the inclusion $\contact{\M}\hookrightarrow \fcontact{\M}$, we have a quasi-isometry $\phi: \mathcal{D}(V_2)\to\fcontact{\M}$.

To see that $\phi$ is coarsely $\hh$-equivariant, first note that clearly the inclusions $\mathcal{ND}(V_2)\hookrightarrow \mathcal{D}(V_2)$ and $\contact{\M}\hookrightarrow\fcontact{\M}$ are $\hh$-equivariant, and so the quasi-inverse $r$ will be coarsely $\hh$-equivariant. Furthermore, $\iota$ is coarsely $\hh$-equivariant because for any $\alpha\in\mathcal{ND}(V_2)^{(0)}$ and $g\in\hh$, the twist hyperplanes $g\cdot\iota(\alpha)$ and $\iota(g\cdot\alpha)$ will be parallel (by Lemma \ref{lem:twistcombparallel}), meaning their distance in $\contact{\M}$ will be at most two.
\end{proof}

It is now straightforward to prove Theorems \ref{thm:hhg-diskgraph} and \ref{thm:stable-orbit}.

\begin{proof}[Proof of Theorem \ref{thm:hhg-diskgraph}]
Corollary \ref{cor:HHG} tells us that $(\hh,\is)$ is an HHG with maximal $\delta$-hyperbolic space $\fcontact{\M}$. Corollary \ref{cor:contact-disk} tells us that $\fcontact{\M}$ is coarsely equivariantly quasi-isometric to $\mathcal{D}(V_2)$.
\end{proof}

\begin{proof}[Proof of Theorem \ref{thm:stable-orbit}.]
Theorem \ref{thm:hhg-diskgraph} tells us that $(\hh,\is)$ is an HHG with maximal $\delta$-hyperbolic space coarsely $\hh$-equivariantly quasi-isometric to $\mathcal{D}(V_2)$. Lemma \ref{lem:unboundedproducts} tells us that $(
\hh,\is)$ has unbounded products, which allows us to apply Theorem \ref{thm:abbott}. Because $\mathcal{D}(V_2)$ is coarsely $\hh$-equivariantly quasi-isometric to $\fcontact{\M}$, (the maximal $\delta$-hyperbolic space in our HHS), then if the orbit map of a subgroup into $\fcontact{\M}$ is a quasi-isometric embedding, so too is the orbit map into $\mathcal{D}(V_2)$.
\end{proof}

%%%%%%%%%%%%%%%%%%%%%%%%%%%%%%%%%%%%%%%%%%%%%%%%%%%%%%%%%%%%%%%%%%%%%%%%
\section{Applications}\label{sec:applications}

In this section we discuss several applications of the main theorems. Specifically, we discuss some properties of the disk graph and the stable subgroups of $\hh$ that follow from the fact that factored contact graphs are quasi-trees \cite[Proposition 8.5]{behrstock2017}. We then provide a topological characterization of the Morse boundary of $\hh$.

%%%%%%%%%%%%%%%%%%%%%%%%%%%%%%%%%%%%%%%%%
\subsection{Quasi-trees}\label{subsec:quasi-trees}

First, since the factored contact graph $\fcontact{\M}$ is a quasi-tree, Corollary \ref{cor:contact-disk} immediately implies the following.
\begin{cor}\label{cor:disk-quasitree}
    The disk graph of genus two $\mathcal{D}(V_2)$ is a quasi-tree.
\end{cor}

Theorem \ref{thm:stable-orbit} and Corollary \ref{cor:disk-quasitree} then imply the following.
\begin{cor}
    Stable subgroups of $\hh$ are virtually free. In particular, any stable subgroup $H$ of the mapping class group $MCG(\bd V_2)$ that is also a subgroup of $\hh$ must be virtually free.
\end{cor}
\begin{proof}
    Because any stable subgroup $G\leq \hh$ quasi-isometrically embeds in $\mathcal{D}(V_2)$ and $\mathcal{D}(V_2)$ is quasi-isometric to a tree, it follows that $G$ quasi-isometrically embeds in a tree. Thus, $G$ is virtually free.
    
    By \cite[Theorem 1.6]{Aougab_2017}, if $G\leq \hh \leq MCG(\bd V_2)$ is stable in $MCG(\bd V_2)$, then it is stable in $\hh$. By the above, $G$ must be virtually free.
\end{proof}

%%%%%%%%%%%%%%%%%%%%%%%%%%%%%%%%%%%%%%%%%
\subsection{Morse boundary}\label{subsec:morse}

In this section, we provide a topological characterization of the Morse boundary of $\hh$. First, we give a brief description of some basic definitions relating to the Morse boundary, which was introduced in \cite{MorseBddy} as a generalization of the Gromov boundary of hyperbolic spaces. For this section, let $X$ be a proper geodesic metric space.

To work towards defining the Morse boundary, let us start by defining Morse geodesics. We say that a geodesic $\gamma$ in $X$ is \emph{Morse} if there is some function $N:\R^{+} \times \R^{+} \to \R^{+}$, which we call a \emph{Morse gauge} for $\gamma$, such that any $(K,C)$-quasi-geodesic in $X$ with endpoints on $\gamma$ remains in a $N(K,C)$-neighborhood of $\gamma$.

Consider two Morse rays $\alpha$ and $\beta$. We say that these rays are equivalent, and write $\alpha\sim \beta$, if there is some constant $D$ such that $d(\alpha(t),\beta(t)) < D$ for all $t$. We denote the equivalence class of a ray by $[\alpha]$. We can now define for a Morse gauge $N$ and a basepoint $e\in X$ the set
\[
    \bdM^N X_e = \{ [\alpha] : \exists \beta\in[\alpha] \text{ such that } \beta \text{ is an } N \text{-Morse geodesic ray with } \beta(0) = e \}.
\]
We give these sets the quotient of the compact-open topology, and then define the \emph{Morse boundary} of $X$ to be the set
\[
    \bdM X = \varinjlim \bdM^N X_e.
\]
We give the Morse boundary the direct limit topology. Note that different choices of basepoint will result in homeomorphic Morse boundaries, so we are justified in omitting the basepoint notation on $\bdM X$.

An alternative way to construct the Morse boundary is given in \cite{stabMorseBddy}. This construction involves taking the direct limit of Gromov boundaries of some suitable spaces. Specifically, we define the set
\[
    X_e^{(N)} = \{ x\in X : \exists \text{ an } N \text{-Morse geodesic } [e,x] \text{ in } X\}.
\]
These sets can be shown to be hyperbolic, so we can construct the Gromov $\bd X_e^{(N)}$ of these sets. Cordes-Hume \cite{stabMorseBddy} show that $\bd X_e^{(N)}$ is homemorphic to $\bdM^N X_e$, and that this induces a homeomorphism on direct limits. Thus, an alternate definition of the Morse boundary is
\[
    \bdM X = \varinjlim \bd X_e^{(N)}.
\]

We now provide a topological characterization of $\bdM \hh$, the proof of which relies on several results from \cite{charney2020complete}. Here, an \emph{$\omega$-Cantor space} is defined as a direct limit $\varinjlim_{i\in \N} X_i$ such that each $X_i$ is a Cantor space, $X_i \subset X_{i+1}$ for all $i$, and $X_i$ has empty interior in $X_{i+1}$ for all $i$. Charney-Cordes-Sisto \cite{charney2020complete} show that any two $\omega$-Cantor spaces are homeomorphic.

\begin{prop}\label{prop:morsebddy}
    The Morse boundary $\bdM \hh$ is an $\omega$-Cantor space.
\end{prop}
\begin{proof}
    For this proof, we appeal to \cite[Theorem 1.4]{charney2020complete}, which states that if a finitely generated group $G$ has Morse boundary which is totally disconnected, $\sigma$-compact, and contains a Cantor space, then $\bdM G$ is a Cantor space when $G$ is hyperbolic, and is an $\omega$-Cantor space when $G$ is not hyperbolic. Since $\hh$ is not hyperbolic, if we can show that $\bdM \hh$ is totally disconnected, $\sigma$-compact, and contains a Cantor space, then it will follow that $\bdM \hh$ is an $\omega$-Cantor space.
    
    First, the Main Theorem and Theorem 2.14 of \cite{contractingBddies} show that the Morse boundary of a $\cat$ space is $\sigma$-compact by showing that it is homeomorphic to the contracting boundary of a $\cat$ space, (which is $\sigma$-compact). This means $\bdM \M$ is $\sigma$-compact, and because the Morse boundary is a quasi-isometric invariant, $\bdM \hh$ must also be $\sigma$-compact.
    
    Next, as discussed in the proof of Theorem 6.6 of \cite{abbott2017largest}, for an HHS with unbounded products, the $N$th strata of of the Morse boundary topologically embeds in the Gromov boundary of the maximal hyperbolic space. In our context, this means that $\bd (\hh)_e^{(N)}$ embeds in $\bd \fcontact{\M}$, for any choice of basepoint $e$ and any choice of Morse gauge $N$. Since $\fcontact{\M}$ is a quasi-tree, its Gromov boundary is totally disconnected. Consequently, $\bdM^N \hh$ is also totally disconnected for each $N$. Since $\bdM \hh$ is $\sigma$-compact, \cite[Proposition 4.4]{charney2020complete} implies that $\bdM \hh$ is totally disconnected.
    
    The last step is to show that $\bdM \hh$ contains a Cantor space. Let $g,h\in\hh$ be mapping classes that act as pseudo-Anosovs on $\bd V_2$. For sufficiently high powers $M$ and $N$, the subgroup $G = \langle g^M, h^N \rangle$ is a free convex cocompact, (and hence stable \cite{durham2015}), subgroup of $MCG(\bd V_2)$ \cite{FarbMosher}. By \cite[Theorem 1.6]{Aougab_2017}, we know $G$ is stable in $\hh$. Because $G$ is a free group, its Gromov boundary is a Cantor space, and hence so is $\bdM G$, (via \cite[Theorem 3.10]{MorseBddy}). Thus, $\bdM \hh$ contains a Cantor space. Applying \cite[Theorem 1.4]{charney2020complete} proves the proposition.
\end{proof}

%%%%%%%%%%%%%%%%%%%%%%%%%%%%%%%%%%%%%%%%%%%%%%%%%%%%%%%%%%%%%%%%%%%%%%%%
\section{Discussion of Higher Genus}\label{sec:highergenus}

In this section, we show that many of the properties of $\hh$ discussed in this paper do not hold for higher genus. First, we note that higher genus handlbody groups are not HHGs. This is because for genus $g\geq 3$, Hamenst\"adt and Hensel prove in \cite{hamHenDehn} that $\hg$ has exponential Dehn function; however, HHGs have quadratic Dehn functions \cite[Corollary 7.5]{HHSII}. In section \ref{subsec:highergenus-stable}, we also show that the analogue of Theorem \ref{thm:stable-orbit} does not hold for higher genus handlebody groups, and in Section \ref{subsec:highergenus-quasitree} we show that the disk graph for higher genus handlebodies is not a quasi-tree.

%%%%%%%%%%%%%%%%%%%%%%%%%%%%%%%%%%%%%%%%%
\subsection{Counterexample to stability characterization in higher genus}\label{subsec:highergenus-stable}

Suppose $g\geq 3$, and let us consider the stable subgroups of $\hg$. At a minimum, an application of \cite[Theorem 1.6]{Aougab_2017} tells us that if $H\leq \hg \leq MCG(\bd V_g)$ is a stable subgroup of $MCG(\bd V_g)$, then $H$ is a stable subgroup of $\hg$ as well. This means, for instance, that purely pseudo-Anosov subgroups $H \leq \hg \leq MCG(\bd V_g)$ will be stable in $\hg$ \cite{purelypA}.

While pseudo-Anosov mapping classes are the only elements that act loxodromically on the curve graph, there are reducible elements in the handlebody group that act loxodromically on the disk graph. It is from such mapping classes that we find a counterexample to the higher genus analogue of Theorem \ref{thm:stable-orbit}. Specifically, we prove the following.
\begin{prop}\label{prop:highergenus}
For $g\geq 3$, there exists an element $\Phi\in\hg$ such that the orbit map $\langle \Phi \rangle \to \mathcal{D}(V_g)$ is a quasi-isometric embedding but such that $\langle \Phi \rangle$ is not stable in $\hg$.
\end{prop}
To prove this, we will construct such a $\Phi$, show that it acts loxodromically on the disk graph, and then show that the cyclic subgroup generated by $\Phi$ embeds in a copy of $\Z^2 \leq \hg$.

To begin the construction of $\Phi$, let $S_0^{g+1}$, for $g\geq 3$, be a sphere with $g+1$ boundary components. Let $\delta_1$ and $\delta_2$ be two of the boundary components. Glue $\delta_1$ and $\delta_2$ together so that the resulting surface $S_1^{g-1}$ is a torus with $g-1$ boundary components. Say that $\alpha\subset S_1^{g-1}$ is the curve corresponding to $\delta_1$ and $\delta_2$. Let $N$ be a regular neighborhood of $\alpha$ and let $S = \overline{S_1^{g-1} - N}$, which is homeomorphic to $S_0^{g+1}$. Choose some reducible $\phi \in MCG(S_1^{g-1})$ that is the identity on $N$ and is pseudo-Anosov on $S$. Now let $V_g = S_1^{g-1}\times I$ where $I = [-1,1]$; $V_g$ is a genus $g$ handlebody. We define
\[ \Phi = \phi\times id\in MCG(V_g) \cong \mathcal{H}_g. \]
We will show that $\Phi$ satisfies the properties described in Proposition \ref{prop:highergenus}.

First we show that $\Phi$ is loxodromic. We say that an element $g\in G$ acting on a hyperbolic $G$-space $X$ is \emph{loxodromic} if the orbit map $\Z \to X$ given by $n\mapsto g^n\cdot x$ for some (any) $x\in X$ is a quasi-isometric embedding. Considering the orbit map of the entire group $G\to X$, being loxodromic easily implies $\langle g\rangle$ quasi-isometrically embeds in $G$.

To see that $\Phi$ is loxodromic, we use the idea of witnesses, (previously called holes due to Masur and Schleimer \cite{MasSchleim}). A $\emph{witness}$ for the disk graph $\mathcal{D}(V_g)$ is a essential subsurface $\Sigma\subset \bd V_g$ such that every representative of every meridian on $V_g$ has non-empty intersection with $\Sigma$. Masur and Schleimer show that distances in the disk graph can be estimated using distances in the curve graphs of witnesses. Specifically, for large enough $A$, there is a constant $B$ such that the following holds:
\[
    d_{\mathcal{D}(V_g)}(\alpha,\beta) =_B \sum_{X \text{ witness}} [d_{\mathcal{C}(X)} (\pi_X(\alpha),\pi_X(\beta))]_A.
\]
Here $=_B$ indicates equality up to additive and multiplicative errors, $[x]_A$ is $x$ if $x\geq A$ and is $0$ otherwise, and $\pi_X$ indicates the subsurface projection to $X$. For more details about witnesses and the distance formula, see \cite{MasSchleim}.

Using the distance formula and witnesses, we can prove the following lemma.
\begin{lemma}\label{lem:loxodromic}
$\Phi$ is loxodromic.
\end{lemma}
\begin{proof}
The upper bound follows from the fact that orbit maps of finitely generated groups are Lipschitz.

For the lower bound, recall that $S = \overline{S_1^{g-1} - N}$ and let 
\[S_i = S\times \{i\} \subset S_1^{g-1} \times \{i\} \subset \bd V_g\]
for $i\in\{-1,1\}$. By the construction of $V_g$, $S_1$ must be a witnesses for $V_g$. To see this, notice that the inclusions $S_{-1}\cup (N\times \{-1\})\hookrightarrow V_g$ and $N\times \{1\}\hookrightarrow V_g$ are $\pi_1$-injective, implying $S_{-1}\cup (N\times \{-1\})$ and $N\times \{1\}$ are incompressible in $V_g$. It follows that no meridian is contained in $S_{-1}\cup (N\times \{\pm 1\})$. Further, no meridian is contained in any component of $\bd S^{g-1}_1 \times I$. Hence, $S_1$ must be a witness.

Because $\Phi|_{S_1}$ is a pseudo-Anosov, then for any $\beta\in\mathcal{D}(V_g)^{(0)}$, the distance $d_{\mathcal{C}(S_1)}(\Phi^n\cdot \beta,\beta)$ must be growing linearly in $n$. Since $S_1$ is a witness for $V_g$, the distance formula tells us that $d_{\mathcal{D}(V_g)}(\Phi^n\cdot \beta,\beta)$ must also be growing linearly.
\end{proof}

In order to prove $\langle \Phi \rangle$ is not stable, we will show that $\langle \Phi \rangle \subset \hg$ is contained in a quasi-isometrically embedded copy of $\Z^2\subset \hg$. To this end, let $A_{\alpha}\subset V_g$ be the properly embedded annulus bounded by $\alpha\times \{-1\}$ and $\alpha\times \{1\}$, where $\alpha$ is as before. Let $\Psi$ be the annulus twist about $A_{\alpha}$, ie $\Psi = T_{\alpha\times \{1\}}T_{\alpha\times \{-1\}}^{-1}\in \hg$.

\begin{lemma}\label{lem:not-stable}
$\langle \Phi \rangle$ is not stable in $\hg$.
\end{lemma}
\begin{proof}
The map $\Psi$ commutes with $\Phi$, so $\langle \Phi ,\Psi \rangle\cong \Z^2$. Furthermore, by appealing to the Masur-Minsky distance formula for $MCG(\bd V_g)$, (see \cite{MMII}), we find that $\langle \Phi, \Psi \rangle \hookrightarrow MCG(\bd V_g)$ is a quasi-isometric embedding. Since the inclusion $\hg \hookrightarrow MCG(\bd V_g)$ is Lipschitz, the inclusion $\langle \Phi,\Psi \rangle \hookrightarrow \hg$ must be a quasi-isometric embedding. Finally, because $\langle \Phi \rangle$ is contained in a quasi-isometrically embedded copy of $\Z^2\subset \hg$, it cannot be stable.
\end{proof}

Lemmas \ref{lem:loxodromic} and \ref{lem:not-stable}  give us Proposition \ref{prop:highergenus}.

One consequence of Proposition \ref{prop:highergenus} is the following.
\begin{cor}\label{cor:notacylindrical}
    The action of $\hg$ for $g\geq 3$ on $\mathcal{D}(V_g)$ is not acylindrical.
\end{cor}
\begin{proof}
    We know via \cite[Corollary 2.9]{DahGuiOsin} and \cite[Theorem 1]{SistoQuasiConvex} that if $G$ is a group acting acylindrically on a hyperbolic space $X$, then any infinite order, loxodromic element $g\in G$ must be stable, (ie $\langle g \rangle$ is stable in $G$). Propostion \ref{prop:highergenus} provides us with an infinite order element acting loxodromically on $\mathcal{D}(V_g)$ that is not stable.
\end{proof}

We should point out that this does not mean that $\hg$ is not acylindrically hyperbolic. In fact, one can see that $\hg$ is acylindrically hyperbolic via the fact that the action of the mapping class group on the curve graph is acylindrical \cite{Bowditch}. Since $\hg \leq MCG(\bd V_g)$, the action of $\hg$ on the curve graph must also be acylindrical.

%%%%%%%%%%%%%%%%%%%%%%%%%%%%%%%%%%%%%%%%%
\subsection{Higher genus disk graphs are not quasi-trees}\label{subsec:highergenus-quasitree}

In this last section, we show that for $g \geq 3$, the disk graph $\DVg$ is not a quasi-tree.

\begin{prop}
    For $g \geq 3$, the disk graph $\DVg$ is not a quasi-tree.
\end{prop}
\begin{proof}
    To prove this proposition, we will show that the arc graph of a specific witness quasi-isometrically embeds in $\DVg$, and then show that the boundary of the arc graph is path connected. Since the boundary of quasi-trees are totally disconnected, this will prove the proposition.

    First, let $S_1^{g-1}$ be a torus with $g-1$ boundary components, and recall that \[V_g = S_1^{g-1} \times [-1,1]\] is a genus $g$ handlebody. Let $W = S_1^{g-1} \times \{1\}$, and let $\mathcal{A}(W)$ denote the arc graph of $W$, ie the simplicial graph whose vertices correspond to isotopy classes of essential arcs and whose edges correspond to disjointness. We can define a map $f:\mathcal{A}(W)^{(0)} \to \DVg^{(0)}$ by $\alpha \mapsto \alpha\times [-1,1]$. Notice that this map extends to a simplicial embedding $f:\mathcal{A}(W) \to \DVg$ since disjoint arcs will map to disjoint disks. We will show that $f$ is a quasi-isometric embedding.
    
    Since $f$ is a simplicial embedding, it is $1$-Lipschitz, so we must only prove the lower bound. To do this, we appeal again to the notion of witnesses, this time for the arc graph. The witnesses for the arc graph of a surface $S$ with boundary, as described in \cite{MasSchleim}, are all essential subsurfaces $X \subset S$ such that $\bd S \subset \bd X$. There is a analogous distance formula for the arc graph in which the summation is taken over all witnesses for the arc graph of $S$:
    \[
    d_{\mathcal{A}(S)}(\alpha,\beta) =_B \sum_{X \text{ witness}} [d_{\mathcal{C}(X)} (\pi_X(\alpha),\pi_X(\beta))]_A.
    \]
    For $W$, the only arc graph witnesses with infinite diameter curve graphs are $W$ itself, and spheres with $g$ or $g+1$ boundary components, (including the components of $\bd W$), depending on whether you cut $W$ along a separating or non-separating curve. As described in Section \ref{subsec:highergenus-stable}, a sphere with $g+1$ boundary components obtained from cutting $W$ along a non-separating curve is also a witness for the disk graph $\DVg$. By a similar argument, we find that a sphere $S_0^g$ with $g$ boundary components obtained by cutting $W$ along a separating curve will also be a witness for $\DVg$: all components of $\bd V_g - S_0^g$ are incompressible, so any meridian in $V_g$ must have non-empty intersection with $S_0^g$.
    
    The relationship between the witnesses for $\mathcal{A}(W)$ and the witnesses for $\DVg$, along with the respective distance formulas,  implies that distances in $\DVg$ are bounded linearly from below by distances in $\mathcal{A}(W)$. We have thus shown that $f$ is a quasi-isometric embedding, and it then follows that $\bd \mathcal{A}(W)$ is homeomorphic to a subspace of $\bd \DVg$.
    
    By \cite[Theorem 1.2]{pho-on2017}, $\bd \mathcal{A}(W)$ is homeomorphic to the space of peripherally ending laminations $\mathcal{EL}_0(W)$, and by \cite[Theorem 1.3]{klarreich2018boundary}, we know that $\bd \mathcal{C}(W)$ is homemorphic to the space of ending laminations $\mathcal{EL}(W).$ Furthermore, since $\mathcal{EL}(W)$  is a subspace of $ \mathcal{EL}_0(W)$, it follows that $\bd \mathcal{C}(W)$  is a subspace of $\bd\mathcal{A}(W).$

    By \cite[Theorem 0.1]{Gabai}, $\mathcal{EL}(W)$ is path-connected, and because
    \[ \bd \mathcal{C}(W) \subset \bd \mathcal{A}(W) \subset \bd \DVg,\]
    it follows that $\bd\DVg$ is not totally disconnected. Hence $\DVg$ cannot be a quasi-tree.
\end{proof}

\bibliographystyle{amsalpha}
\bibliography{bib}

\end{document}